\documentclass[10pt,reqno]{amsart}
\usepackage{amsmath,amsfonts,amsthm,amssymb,amsxtra,dsfont}
\usepackage{amsxtra, amssymb, mathrsfs, pstricks}
\usepackage[english]{babel}
\usepackage{amsmath,amsfonts,amstext,amsbsy,amsopn,amssymb,amsthm,amscd}
\usepackage{amsxtra,latexsym}
\usepackage{exscale}
\usepackage{graphicx,mathrsfs}
\usepackage{psfrag}
\usepackage{eucal,enumerate}
\usepackage{color,graphics}
\usepackage{extpfeil}

\definecolor{dred}{rgb}{.8,0,0}
\definecolor{ddmagenta}{rgb}{0.7,0,0.9}
\definecolor{ddcyan}{rgb}{0,0.2,1.0}
\definecolor{dblue}{rgb}{0,0,0.7}
\definecolor{ddorange}{rgb}{1,0.5,0}
\definecolor{ddgreen}{rgb}{0,0.4,0.4}
\definecolor{Turk}{rgb}{0,0.7,0.4}

\newtheorem{theorem}{Theorem}[section]
\newtheorem{proposition}[theorem]{Proposition}

\newtheorem{lemma}[theorem]{Lemma}

\theoremstyle{definition}

\newtheorem{definition}[theorem]{Definition}
\newtheorem{example}[theorem]{Example}

\newtheorem{problem}[theorem]{Problem}

\theoremstyle{remark}
\newtheorem{remark}[theorem]{\bf Remark}
\numberwithin{equation}{section}

\def\trait #1 #2 #3 {\vrule width #1pt height #2pt depth #3pt}
\def\fin{
    \trait .3 5 0
    \trait 5 .3 0
    \kern-5pt
    \trait 5 5 -4.7
    \trait 0.3 5 0
\medskip}

\newcommand{\QED}{\hfill $\square$}
 \def\bbN{{\mathbb N}} 
  \def\bbR{{\mathbb R}}

\newcommand{\R}{\bbR}
\newcommand{\N}{\bbN}
\def\BS{\boldsymbol} % griechisch

\def\bfsigma{{\BS\sigma}}

\newcommand{\bele}{\begin{lemm}\begin{sl}}
\newcommand{\enle}{\end{sl}\end{lemm}}
\newcommand{\bedef}{\begin{defi}\begin{sl}}
\newcommand{\eddef}{\end{sl}\end{defi}}
\newcommand{\bete}{\begin{teor}\begin{sl}}
\newcommand{\ente}{\end{sl}\end{teor}}
\newcommand{\beos}{\begin{osse}\begin{rm}}
\newcommand{\eddos}{\end{rm}\end{osse}}
\newcommand{\bepr}{\begin{prop}\begin{sl}}
\newcommand{\empr}{\end{sl}\end{prop}}
\newcommand{\bepro}{\begin{prob}\begin{rm}}
\newcommand{\empro}{\end{rm}\end{prob}}
\newcommand{\bede}{\begin{defin}\begin{sl}}
\newcommand{\edde}{\end{sl}\end{defin}}
\newcommand{\beco}{\begin{coro}\begin{sl}}
\newcommand{\enco}{\end{sl}\end{coro}}
\newcommand{\behy}{\begin{hypo}\begin{sl}}
\newcommand{\enhy}{\end{sl}\end{hypo}}

\newcommand{\thspace}{\hspace{3mm}}

\newcommand{\RR}{\mathbb{R}}

\newcommand{\beeq}[1]{\begin{equation}\label{#1}}
\newcommand{\eddeq}{\end{equation}}
\newcommand{\beeqa}[1]{\begin{eqnarray}\label{#1}}
\newcommand{\eddeqa}{\end{eqnarray}}
\newcommand{\beal}[1]{\begin{align}\label{#1}}
\newcommand{\eddal}{\end{align}}
\newcommand{\bespl}[1]{\begin{split}\label{#1}}
\newcommand{\edspl}{\end{split}}
\newcommand{\bega}[1]{\begin{gather}\label{#1}}
\newcommand{\edga}{\end{gather}}
\newcommand{\beeqax}{\begin{eqnarray*}}
\newcommand{\eddeqax}{\end{eqnarray*}}

\newcommand{\weakto}{\rightharpoonup}

\newcommand{\duav}[1]{\langle{#1}\rangle}

\newcommand{\itt}{\int_0^t}

\newcommand{\io}{\int_\Omega}

\newcommand{\e}{\varepsilon}

\newcommand{\bx}{\boldsymbol{x}}

 \DeclareMathOperator{\dive}{div}

\let\TeXchi\chi
\def\chi{{\setbox0 \hbox{\mathsurround0pt
$\TeXchi$}\hbox{\raise\dp0 \copy0 }}}

\newcommand{\ub}{\mathbf{u}}
\newcommand{\uu}{\mathbf{u}}
\newcommand{\vv}{\mathbf{v}}

\newcommand{\vb}{\mathbf{v}}

\newcommand{\teta}{\vartheta}

\newcommand{\fb}{\mathbf{f}}
\newcommand{\aein}{\text{a.e. in\;}}
\newcommand{\tensore}{\varepsilon({\bf u})}
\newcommand{\tensoret}{\varepsilon({\mathbf{u}_t})}

\newcommand{\forae}{\text{for a.a.}}
\newcommand{\foraa}{\text{for a.a.}}

\def\fine{\hfill\kern4pt \vrule height4pt depth0pt width4pt }

\numberwithin{equation}{section}
\numberwithin{equation}{section}

\newcommand{\Ha}{L^2 (\Omega;\R^d)}

\newcommand{\boZ}{H_{0}^1(\Omega;\R^d)}
\newcommand{\boY}{H_{\mathrm{Dir}}^2(\Omega;\R^d)}

\newcommand{\dd}{\, \mathrm{d}}
\newcommand{\pairing}[4]{ \sideset{_{#1 }}{_{ #2}}  {\mathop{\langle #3 , #4  \rangle}}}

\newcommand{\eps}{\varepsilon}
\newcommand{\condu}{\mathsf{K}}

\newcommand{\ftau}[1]{\mathbf{f}_{\tau}^{#1}}

\newcommand{\gtau}[1]{g_{\tau}^{#1}}
\newcommand{\htau}[1]{h_{\tau}^{#1}}

\newcommand{\pwc}[2]{\overline{#1}_{#2}}
\newcommand{\pwl}[2]{{#1}_{#2}}

\newcommand{\upwc}[2]{\underline{#1}_{#2}}

\newcommand{\DDD}[3]{\begin{array}[t]{c}#1\vspace*{-1em}\\_{#2}\vspace*{-.5em}\\_{#3}\end{array}}
\newcommand{\ddd}[3]{\DDD{\begin{array}[t]{c}\underbrace{#1}\vspace*{.6em}\end{array}}{\text{\footnotesize #2}}{\text{\footnotesize #3}}}
\newcommand{\down}{\downarrow}

\newcommand{\BV}{\mathrm{BV}}
\newcommand{\rmC}{\mathrm{C}}

\newcommand{\vism}{\mathbb{V}}

\newcommand{\sigmab}{{\boldsymbol{\sigma}}}
\newcommand{\db}{\mathbf{d}}
\newcommand{\ttau}{\ol{\mathsf t}_\tau}

\newcommand{\1}{\mathbf{1}}

\newcommand{\C}{\mathcal}
\newcommand{\ol}{\overline}
\newcommand{\ul}{\underline}
\newcommand{\dS}{\,\mathrm dS}
\newcommand{\dx}{\,\mathrm dx}

\newcommand{\ds}{\,\mathrm ds}
\newcommand{\dr}{\,\mathrm dr}

\newcommand{\dxs}{\,\mathrm dx\,\mathrm ds}

\newcommand{\weaklim}{\rightharpoonup}
\newcommand{\weakstarlim}{\stackrel{\star}{\rightharpoonup}}

\newcommand{\CC}{\mathbb{C}}

\newcommand{\VV}{\mathbb{V}}
\newcommand{\dk}{\mathbf{d}_\tau^k}

\newcommand{\uk}{\ub_\tau^k}
\newcommand{\ukk}{\ub_\tau^{k-1}}

\newcommand{\ck}{c_\tau^{k}}
\newcommand{\xk}{\xi_\tau^{k}}

\newcommand{\ellk}{\ell_\tau^{k}}
\newcommand{\zek}{\zeta_\tau^{k}}
\newcommand{\ckk}{c_\tau^{k-1}}
\newcommand{\muk}{\mu_\tau^{k}}

\newcommand{\vk}{\mathbf v_\tau^k}

\newcommand{\zk}{z_\tau^k}
\newcommand{\zkk}{z_\tau^{k-1}}

\newcommand{\fk}{\bold f_\tau^k}
\newcommand{\gk}{g_\tau^k}
\newcommand{\hk}{h_\tau^k}
\newcommand{\tk}{\vartheta_\tau^k}
\newcommand{\tkk}{\vartheta_\tau^{k-1}}

\newcommand{\Dt}{D_{\tau,k}}

\newcommand{\zM}{z_M^{k}}
\newcommand{\tM}{\vartheta_M^{k}}

\newcommand{\tE}[5]{\mathscr{E}(#1,#2,#3,#4,#5)}
\newcommand{\tF}[4]{\mathscr{F}(#1,#2,#3,#4)}
\newcommand{\tU}[4]{\mathscr{U}(#1,#2,#3,#4)}
\newcommand{\tP}[4]{\mathscr{P}(#1,#2,#3,#4)}

\usepackage[T3,T1]{fontenc}

\DeclareSymbolFont{tipa}{T3}{cmr}{m}{n}
\DeclareMathAccent{\invbreve}{\mathalpha}{tipa}{16}

\newcommand{\pd}[1]{W_{,#1}}

\newcommand{\convWc}{{\breve{W}_{1}^\omega}}
\newcommand{\concWc}{{\invbreve{W}_{1}^\omega}}
\newcommand{\convWz}{{\breve{W}_{3}^\omega}}
\newcommand{\concWz}{{\invbreve{W}_{3}^\omega}}
\newcommand{\convWcp}{\breve{W}_{1,c}^\omega}
\newcommand{\concWcp}{\invbreve{W}_{1,c}^\omega}
\newcommand{\convWzp}{\breve{W}_{3,z}^\omega}
\newcommand{\concWzp}{\invbreve{W}_{3,z}^\omega}
\newcommand{\conv}[1]{\breve{\text{$#1$}}}
\newcommand{\conc}[1]{\invbreve{#1}}

\newcommand{\bA}{\mathbf A}
\newcommand{\Hn}{H_N^2}
\newcommand{\mean}{\mathfrak{m}}

\def\Xint#1{\mathchoice{\XXint\displaystyle\textstyle{#1}}{\XXint\textstyle\scriptstyle{#1}}
{\XXint\scriptstyle\scriptscriptstyle{#1}}{\XXint\scriptstyle\scriptscriptstyle{#1}}\!\int}
\def\XXint#1#2#3{{\setbox0=\hbox{$#1{#2#3}{\int}$}
\vcenter{\hbox{$#2#3$}}\kern-.5\wd0}}
\def\dashint{\Xint-}

\newcommand{\EEE}{\color{black}}
\newcommand{\RRR}{\color{black}}%{\color{magenta}}

\begin{document}

\title[Phase separation and damage in  thermoviscoelasticity]
{A temperature-dependent phase-field model\\
for phase separation and damage}

\author{Christian Heinemann}
\address{Christian Heinemann\\ Weierstrass Institute for Applied
Analysis and Stochastics\\ Mohrenstr.~39 \\ D-10117 Berlin \\
Germany}
\email{heineman@wias-berlin.de}

\author{Christiane Kraus}
\address{Christiane Kraus\\ Weierstrass Institute for Applied
Analysis and Stochastics\\ Mohrenstr.~39 \\ D-10117 Berlin \\
Germany}
\email{kraus@wias-berlin.de}

\author{Elisabetta Rocca}
\address{Elisabetta Rocca\\ Weierstrass Institute for Applied
Analysis and Stochastics\\ Mohrenstr.~39 \\ D-10117 Berlin \\
Germany\\ and \\
Dipartimento di Matematica \\ Universit\`a di Milano\\ Via Saldini 50 \\ I-20133 Milano\\ Italy}
\email{rocca@wias-berlin.de and elisabetta.rocca@unimi.it}

\author{Riccarda Rossi}
\address{Riccarda Rossi\\  DIMI \\ Universit\`{a} di Brescia\\ Via Valotti 9\\ I-25133 Brescia\\ Italy}
\email{riccarda.rossi@unibs.it}

\date{October 13, 2015}

\begin{abstract}
In this paper we study a model for phase separation and damage in
thermoviscoelastic materials.  The main novelty of the paper consists in the
fact that, in contrast with  previous works in the literature (cf., e.g.,
\cite{hk1,hk2}), we  encompass in the model  thermal processes,  \EEE
nonlinearly coupled with the damage, concentration and displacement
evolutions. More in particular, we prove the existence of   ``entropic
weak  solutions'',  resorting to a solvability concept \EEE first introduced
in \cite{fei} in the framework of Fourier-Navier-Stokes systems and then
recently employed in  \cite{fpr09, RocRos14} \EEE for the study of PDE
systems for phase   transition \EEE and damage.   Our global-in-time existence result is obtained by passing to the limit in a carefully devised time-discretization scheme. \EEE
\end{abstract}

\maketitle

%%%%%%%%%%%%%%%%%%%%%%%%%%%%%%%%%%%%%%


 %in thermoviscolastic materials, existence of weak solutions, time discretization.
  \vspace{4mm}

\noindent {\bf AMS (MOS) subject clas\-si\-fi\-ca\-tion:}\thspace
35D30,  74G25, 93C55, 82B26, 74A45.

\section{\bf Introduction and modeling \EEE}
In this paper we
propose and analyze a model for phase separation and damage in a thermoviscoelastic body, occupying a spatial domain
$\Omega \subset \R^d$, where $d\in \{2,3\}$. 
We  shall \EEE consider here a suitable weak formulation of the following PDE system% $\overarc{4}$

\begin{subequations}
\label{eqn:PDEsystem-expli}
\begin{align}
  &c_t=\dive(m(c,z)\nabla\mu),
  	\label{e:c-expl-intro}\\
  &\begin{aligned}
    \mu ={}&-\Delta_p(c)+\phi'(c) +\frac12\big(b(c,z) \mathbb{C}(\e(\ub)-\e^*(c)):(\e(\ub)-\e^*(c))\big)_{,c}-\vartheta+c_t,
	\end{aligned}
  	\label{e:mu-expl-intro}\\
  &z_t+\partial I_{(-\infty,0]}(z_t) -\Delta_p(z)+\partial I_{[0,\infty)}(z) +\sigma'(z)\ni
    -\frac12 b_{,z}(c,z) \mathbb{C} (\e(\ub)-\e^*(c)):(\e(\ub)-\e^*(c))  +\vartheta,\label{e:z-expl-intro}\\
  &\vartheta_t+c_t\vartheta+z_t\vartheta+\rho\vartheta\dive(\ub_t)-\dive(\condu(\vartheta)\nabla\vartheta) =g+|c_t|^2+|z_t|^2+a(c,z)\e(\ub_t):\vism\e(\ub_t)+m(c,z)|\nabla\mu|^2,
  	\label{e:teta-expl-intro}\\
  &\ub_{tt}-\dive\big( a(c,z)\vism\e(\ub_t) + b(c,z) \mathbb{C} (\e(\ub)-\e^*(c)) -\rho\vartheta\mathds 1\big)=\fb
  	\label{e:u-expl-intro}
\end{align}
\end{subequations}
posed in $\Omega \times (0,T)$. 
The system couples 
\begin{itemize}
\item[-] the viscous Cahn-Hilliard equation \eqref{e:c-expl-intro}--\eqref{e:mu-expl-intro} ruling the evolution of the concentration $c$;
\item[-] the damage  flow rule \EEE \eqref{e:z-expl-intro} for the local proportion of the damage $z$;
\item[-] the internal energy balance \eqref{e:teta-expl-intro} for the absolute temperaure $\teta$ ;
\item[-] the momentum balance \eqref{e:u-expl-intro} describing the dynamics for the displacement $\ub$. 
\end{itemize}
The symbol $(\cdot)_t$ denotes the partial derivative with respect to time.  In the Cahn-Hilliard equation \eqref{e:c-expl-intro} $m$ denotes the mobility of the system
and
 $\mu$ the chemical potential, whose expression is \EEE given in
\eqref{e:mu-expl-intro}.  There, $\Delta_p(\cdot):=  \dive(|\nabla\cdot|^{p-2}\nabla\cdot)$ denotes the $p$-Laplacian, $\phi$ is a mixing potential, $b$ is an elastic coefficient function depending possibly on both $c$ and $z$, $\mathbb{C}$ represents the elasticity tensor, $\e^*$ a residual strain tensor, and 
$(\cdot)_{,c}$ the partial derivate with respect to the variable $c$
(with an analogous notation for  the \EEE other variables). 
In the damage  flow rule \EEE \eqref{e:z-expl-intro} $\partial I_{(-\infty,0]}: \R \rightrightarrows \R$ denotes the subdifferential of the indicator function of the set $(-\infty,0]$, 
 given by 
\[
\partial I_{(-\infty,0]}(v) = \begin{cases}
\{ 0\} & \text{for } v <0,
\\
[0,+\infty) & \text{for } v=0
\end{cases}
\]
%which is equal to $0$ when $z_t\leq 0$ and $+\infty$ otherwise, 
while  $\partial  I_{[0,\infty)}: \R \rightrightarrows \R$  is  \EEE   the subdifferential of the indicator function of the set $[0,\infty)$,  i.e.
\[
\partial I_{[0,\infty)}(z) = \begin{cases}
(-\infty, 0]& \text{for } z=0,
\\
\{ 0\} & \text{for } z >0\,.
\end{cases}
\]
\EEE
 The presence of these  two \EEE  maximal monotone \EEE  graphs,  enforcing in particular \EEE the irreversibility of the damage phenomenon,
% and the positivity of the damage parameter, \EEE 
entails the constraint $z(t)\in [0,1]$ for $t\in (0,T)$ as soon as $z(0)\in
[0,1]$. This is physically meaningful because $z$ denotes the damage parameter
which is set to be equal to $0$ in case the material is completely damaged and
it is \EEE equal to $1$ in the completely safe case, while $z\in (0,1)$
indicates partial damage. The function $\sigma$ in \eqref{e:z-expl-intro}
represents a smooth function, possibly non-convex, of the damage 
 variable \EEE $z$. 
In the temperature equation  \eqref{e:teta-expl-intro}, $\rho$ denotes a positive thermal expansion coefficient, $\condu$ the heat conductivity of the system, $g$ a given heat source and $a$ a viscosity coefficient possibly depending on $c$ and $z$, while $\mathbb{V}$ is the viscosity tensor.  Finally, in the momentum balance \eqref{e:u-expl-intro} $\fb$ denotes a given  volume force.

We will supplement system \eqref{eqn:PDEsystem-expli}  with \EEE the
initial-boundary conditions  %\GGG ( n bold in the following) \EEE
%\begin{subequations}
\begin{subequations}
\label{init-bdry-conditions}
\begin{align}
  &c(0)=c^0,
  &&z(0)=z^0,
  &&\vartheta(0)=\vartheta^0,
    &&\ub(0)=\ub^0,
  &&\ub_t(0)=\mathbf v^0
  &&\text{a.e.\ in }\Omega,
  	\label{init-conditions}\\
  &\nabla c\cdot  { \bf  n \EEE}=0,
   &&m(c,z)\nabla\mu\cdot { \bf  n \EEE}=0,
  &&\nabla z\cdot  { \bf  n \EEE}=0,
  &&\condu(\vartheta)\nabla\vartheta\cdot { \bf  n \EEE  }=h,
    &&\ub=\db
  &&\text{a.e. on }\partial\Omega\times(0,T),
  	\label{bdry-conditions}
\end{align}
\end{subequations}
where ${ \bf n \EEE  }$ indicates the outer unit normal to $\partial\Omega$,  while $h$ and $\db$ denote, respectively, a given boundary heat source and displacement.  

The PDE system \eqref{eqn:PDEsystem-expli} may be written in the more compact form
\begin{subequations}
\label{eqn:PDEsystem}
\begin{align}
  &c_t=\dive(m(c,z)\nabla\mu),
  	\label{e:c}\\
  &\mu = -\Delta_p(c)+\phi'(c)+W_{,c}(c,\e(\ub),z)-\vartheta+c_t,
  	\label{e:mu}\\
  &z_t +\partial I_{(-\infty,0]}(z_t)  -\Delta_p(z)+\partial I_{[0,\infty)}(z) +\sigma'(z)\ni -W_{,z}(c,\e(\ub),z)+\vartheta,
  	\label{e:z}\\
  &\vartheta_t+c_t\vartheta+z_t\vartheta+\rho\vartheta\dive(\ub_t)-\dive(\condu(\vartheta)\nabla\vartheta)=g+|c_t|^2+|z_t|^2+a(c,z)\e(\ub_t):\vism\e(\ub_t)+m(c,z)|\nabla\mu|^2,
  	\label{e:teta}\\
  &\ub_{tt}-\dive\big(a(c,z)\vism\e(\ub_t)+W_{,\e}(c,\e(\ub),z)-\rho\vartheta\mathds{1}\big)=\fb,
  	\label{e:u}
\end{align}
\end{subequations}
%which yields \eqref{eqn:PDEsystem-expli}
with the following choice of the elastic energy density
\begin{equation}
\label{elastic-energy}
    W(c,\e,z)=\frac 12b(c,z)\mathbb C(\e-\e^*(c)):(\e-\e^*(c)).
\end{equation}
The  expression \EEE of $W$ is typically quadratic
 as a function of the strain tensor $\e(\ub)$, whereas  \EEE
 the coefficient $b$ can depend on $c$ and $z$.   This accounts for possible inhomogeneity of elasticity on the one hand, 
 and is characteristic for damage on the other hand. Indeed, the natural choice would be 
 % model possible inhomogenuos elasticity and to allow the possibility of having
 that $b$ vanishes \EEE for $z=0$, i.e.\ when the material is completely damaged.
 %  However this leads to remarkable analytical difficulties, and will be 
 %outside the scope of our investigation. \EEE
%%%%%%% CK
\paragraph{\bf Derivation of the model}
Let us briefly discuss the thermodynamically consistent derivation of the
PDE-system \eqref{eqn:PDEsystem-expli}.

The {\it state variables} that determine the local thermodynamic state of the
material and the {\it dissipative variables} whose evolution describes the way along which
the system tends to dissipate energy are as follows: \\
{\it State variables} 
$$ \teta, \, c, \, \nabla c, \, \e(\ub), \,  z, \, \nabla z $$
{\it Dissipation variables}
$$ \nabla \teta, \, c_t, \, \e(\ub_t), \, z_t $$

By classical principles of thermodynamics, the evolution  of the \EEE system is based on
the free energy $\mathscr{F}$ and the pseudopotential of dissipation $\mathscr{P}$, for which we assume 
the following general form: 
\begin{align}
 \tF{c}{z}{\teta}{ \eps(\ub)} = \int_\Omega F(c,\nabla c,z,\nabla z,\teta,  \eps(\ub) )\dx
 \qquad \text{and} \qquad 
\tP{\nabla \teta}{c_t}{\e(\ub_t)}{z_t} = \int_\Omega P(\nabla \teta, c_t,
\e(\ub_t), z_t) \dx .
\end{align}
Our  evolutionary \EEE system has been obtained by the principle of virtual power
and by balance equations of  micro-forces,  \EEE
 a generalization of the approaches by Fremond \cite{fremond} and Gurtin \cite{gur96}.
In addition, we also include  temperature-dependent \EEE effects by means of the 
 balance equation of energy. \\ 
The system relies altogether on the balance equations of mass, forces, 
micro-forces and energy:\\
%supplemented with the boundary conditions \eqref{bdry-conditions}
{\it Evolution system}
\begin{subequations}
\label{eqn:PDEsystem-derivation}
\begin{align}
 \text{Mass balance} \hspace{0.98cm}& \notag\\
 c_t  +  \dive {\BS J} &=0,
  	\label{e:c-deri}\\
 \text{Force balance}\hspace{0.9cm}& \notag\\
\ub_{tt}  -  \dive { \bfsigma} &= {\mathbf f}, 
          \label{e:u-deri} \\
  \text{Micro-force balance} & \notag\\
  B  - \dive{\bf H}& =  0 ,\\
  \Pi  - \dive { \BS \xi} &= 0 ,\\
  \text{Energy balance} \hspace{0.6cm}& \notag \\
  U_t  +   \dive { \BS q} \;&= g + \bfsigma: \e(\ub_t) + 
    \dive ({\bf H}) z_t + {\bf H} \cdot \nabla z_t
    + \dive ( {\BS \xi}) c_t
     + { \BS \xi}  {\cdot} \nabla c_t -  \dive ({\BS J}) \, \mu -{\BS J} \cdot \nabla \mu
%&\phantom{=} ,
 \end{align}
\end{subequations}
where %$U_t$ denotes 
the internal energy density is given by $U= F - \teta \partial_\teta F$.\\
Note that the system is not closed for the variables. Therefore, constitutive laws have to be
imposed for the mass flux ${ \bf J}$, the stress tensor $\bfsigma$, the internal
microforce $B$ for $z$, the microstress ${\bf H}$ for $z$, the internal
microforce $\Pi$ for $c$, the microstress ${ \BS \xi}$ for $c$
and the heat flux ${\bf q}$.\\
{\it Constitutive relations}\\ 
Following Fr{\'e}mond's perspective, we assume that the stress tensor $\bfsigma$, the microforce $B$ and the
microstress ${\bf H}$, may be additively decomposed into their
non-dissipative and dissipative components, i.e.
\begin{align}
&& &  {\BS \sigma }= {\BS \sigma }^{nd} + {\BS \sigma }^d & \text{with}&
  \qquad {\BS \sigma }^{nd}=  \partial_{\e(\ub)} F , & & {\BS \sigma }^d =
  \partial_{\e(\ub_t)} P, 
& & && &&\\
&& &B= B^{nd} + B^d & \text{with}& \qquad B^{nd} \in  \partial_z F , & & B^d \in
  \partial_{z_t} P, 
& & && &&\\
&& & {\bf H} = {\bf H}^{nd} + {\bf H}^d &\text{with}  &  \qquad
{\bf H}^{nd}= \partial_{\nabla z} F  , & & {\bf H}^d=  \partial_{\nabla
  z_t} P=0 .& & && & &
\end{align}
In a similar way, by choosing Gurtin's approach, cf.~\cite{gur96} equations (3.19)-(3.23), we get the constitutive relations:
\begin{align}
{ \bf J}= - m(c,z) \nabla \mu,\qquad\qquad  \Pi= \partial_c F + \partial_{c_t} P
- \mu, \qquad \qquad
{\BS \xi} = \partial_{\nabla c} F.
\end{align}
The heat source is given by the standard constitutive relation:
$$  { \BS q}= - \frac{\partial P}{ \partial \nabla \teta}. $$
In the framework of the formulation of the damage
and phase separation theory \cite{fremond, gur96}, we choose for our system 
the following free energy and dissipation potential:
%%%%%%%
% \par 
%The above system has been obtained by means of the principle of virtual powers,  following the approach by \textsc{M.\ Fr\'%emond}, \EEE applied to the 
%following free energy $\mathscr{F}$ functional
\begin{align}
    \label{free-energy}
    \tF{c}{z}{\teta}{ \eps(\ub)}:={}&\int_\Omega \frac 1p|\nabla c|^p+\frac 1p|\nabla z|^p+W(c,\e(\ub),z)+\phi(c)+\sigma(z)+ I_{[0,+\infty)}(z)\dx\notag\\
        &+\int_\Omega-\teta\log\teta-\teta\big(c+z+\rho\dive(\ub)\big)\dx,\\
\label{dissipation}
 \tP{\nabla \teta}{c_t}{\e(\ub_t)}{z_t} :={} & \int_\Omega  \frac{1}{2} 
\condu(\vartheta)|\nabla\vartheta|^2 + \frac{1}{2} z_t^2 + \frac{1}{2} c_t^2 + \frac{1}{2}
a(c,z)\e(\ub_t):\vism\e(\ub_t) %+m(c,z)|\nabla\mu|^2
    + I_{(-\infty,0]}(z_t)    \dx \, .
\end{align}
\par
 The first two gradient terms in \eqref{free-energy} represent the 
 nonlocal interactions in phase separation and damage processes. 
 The 
analytical study of gradient \RRR  theories \EEE goes back to \cite{LM89,Mod98}, where phase 
separation processes were investigated. \EEE 
A typical choice for $W$ has
 been introduced in \eqref{elastic-energy}. The functions 
 $\phi$ and $\sigma$ represent the mixing potentials. 
 The term $\teta(c+z+\rho\dive \ub)$ models the phase and thermal expansion
 processes in the system. It may also be regarded as linear approximation near
 to the thermodynamical equilibrium. In the following lines we will get
further insight into the choices of these functionals.
Exploiting \eqref{eqn:PDEsystem-derivation}-\eqref{dissipation} results in system \eqref{eqn:PDEsystem-expli}, for which the
Clausius-Duhem inequality is satisfied.
 % (cf.~the following lines for more examples and explanations on the choices
 % of the functionals). 

As discussed, our approach is based on a gradient theory of phase separation and damage
processes due to \cite{fremond,gur96,CH58}. 
For a non-gradient approach \RRR to \EEE damage models we refer to \cite{FG06, GL09,
  Bab11}. 
There, the damage variable
$z$ takes \RRR  only \EEE two distinct values, i.e.
$\{0,1\}$, in contrast to phase-field models where intermediate values
$z \in [0,1]$ are also allowed.
In addition, the mechanical properties of damage phenomena are described in \cite{FG06, GL09, Bab11}
differently. They choose a $z$-mixture of a linearly elastic strong and weak material with two different
elasticity tensors. We also refer to \cite{FKS11}, where a non-gradient damage
model was studied by \RRR means of \EEE Young measures.
\EEE
%\newpage
\paragraph{\bf Mathematical difficulties.} The main mathematical difficulties  attached with \EEE the proof of existence of solutions to such a PDE system are related to the presence of the quadratic dissipative terms on the  right-hand  \EEE side in the internal energy balance \eqref{e:teta},  as well as the doubly nonlinear and possibly nonsmooth carachter of the damage relation \eqref{e:z}. 
This is the reason why we  shall \EEE resort here to a weak solution notion for \eqref{eqn:PDEsystem} coupled with \eqref{init-bdry-conditions}.  In this solution concept, partially drawn from \cite{RocRos14},  the Cahn-Hilliard system (\ref{e:c}--\ref{e:mu}) and the balance of forces \eqref{e:u} (read a.e. in $\Omega\times(0,T)$)  
\EEE
are coupled with an  {\sl ``entropic'' formulation} of the heat equation \eqref{e:teta} and  a weak formulation of the damage  flow rule \eqref{e:z} \EEE taken from \cite{hk1,hk2}. 
 Let us briefly illustrate them. \EEE
\paragraph{\bf The ``entropic'' formulation  of the heat equation.} It consists of  \EEE a {\sl weak entropy inequality}
	    \begin{equation}
	        \label{entropy-ineq-intro}
	        \begin{aligned}
	          &\int_s^t \int_\Omega (\log(\teta) + c+z) \varphi_t  \dd x \dd r  -
	          \rho \int_s^t \int_\Omega \dive(\uu_t) \varphi  \dd x \dd r
	          -\int_s^t \int_\Omega  \condu(\teta) \nabla \log(\teta) \cdot \nabla \varphi  \dd x \dd r\\
	          &\begin{aligned}
	  	        \leq
	   	       	\int_\Omega (\log(\teta(t))+c(t)+z(t)){\varphi(t)} \dd x
	           	&-\int_\Omega (\log(\teta(s))+c(s)+z(s)){\varphi(s)} \dd x\\
	            &-\int_s^t \int_\Omega \condu(\teta)|\nabla\log(\teta)|^2\varphi\dd x \dd r
	          \end{aligned}\\
	          &\quad-\int_s^t  \int_\Omega \left( g +|c_t|^2+ |z_t|^2  + a(c,z) \eps(\uu_t):\vism \eps(\uu_t) + m(c,z)|\nabla \mu|^2\right)
	          \frac{\varphi}{\teta} \dd x \dd r
	          -\int_s^t \int_{\partial\Omega} h \frac\varphi\teta  \dd S \dd r
	        \end{aligned}
	      \end{equation}
	       required to be valid \EEE for almost all $0\leq s \leq t \leq T$ and for $s=0$, and for all sufficiently regular and positive test functions $\varphi$, coupled with a 
	    {\sl total energy inequality}:
	    \begin{equation}
	    \label{total-enid-intro}
	    \begin{aligned}
	      \tE{c(t)}{z(t)}{\teta(t)}{\ub(t)}{\ub_t(t)}
	      	\leq{}&\tE{c(s)}{z(s)}{\teta(s)}{\ub(s)}{\ub_t(s)}\\
			     &+ \int_s^t\int_\Omega g \dd x \dd r
			     + \int_s^t\int_{\partial\Omega} h \dd S \dd r\\
	 		     &+ \int_s^t \int_\Omega \mathbf{f} \cdot \mathbf{u}_t \dd x \dd r
	    		 + \int_s^t \int_{\partial\Omega}\big(\sigmab { \bf
                            n \EEE} \big)\cdot \db_t \dd S \dd r,
	    \end{aligned}
	    \end{equation}
	     valid \EEE for almost all $0\leq s \leq t \leq T$, and for $s=0$, where %for $s=0$ we read $\teta(0)= \teta^0$ and where
	     the 
total energy $\mathscr E$  is the sum of the internal energy and
the kinetic energy, i.e.
\begin{align}
    \tE{c}{z}{\teta}{\ub}{\ub_t}:={}&\tU{c}{z}{\teta}{\RRR \eps(\ub)\EEE}+\int_\Omega\frac12|
    \ub_t \EEE|^2\dx,
    \label{total-energy}
\end{align}
being 
the internal energy $\mathscr{U}$ specified by (cf.~also \eqref{free-energy}):
\begin{align}
    \tU{c}{z}{\teta}{\RRR \eps(\ub)\EEE}:={}&\tF{c}{z}{\teta}{\eps(\ub)}-\teta\cdot\partial_\teta\mathscr{F}(c,z,\teta,\RRR \eps(\ub)\EEE)\notag\\
        ={}&\int_\Omega \frac 1p|\nabla c|^p+\frac 1p|\nabla z|^p+W(c,\e(\ub),z)+\phi(c)+\sigma(z)+ I_{[0,+\infty)}(z)+\teta\dx.
\end{align}

From an analytical viewpoint, observe that the entropy
inequality \eqref{entropy-ineq-intro} has the advantage that all the
quadratic terms on the right-hand side of \eqref{e:teta} are multiplied  by a negative test  function, which, together with the fact that we are only requiring an {\sl inequality} and not an equation,  will allow \EEE us to 
apply upper semicontinuity arguments  for \EEE the limit passage in  the time-discrete approximation 
 of system \eqref{eqn:PDEsystem} set up in \EEE
 Section~\ref{s:5}.

 The \EEE {\sl ``entropic'' formulation},  first introduced in \cite{fei} in the framework
 of  heat  conduction  in fluids, and then applied to a  phase separation
 \EEE   model \EEE derived according to \textsc{Fr\'emond}'s approach
  \cite{fremond} in \cite{fpr09},  has been successively
  %successfully
    used  also in  models for different kinds of special materials.  Besides the aforementioned
    work on damage \cite{RocRos14}, we may mention  the papers \cite{ffrs}, \cite{frsz1},  and \cite{frsz2} on liquid crystals,  \EEE
 and more recently    the analysis of  a  model for the evolution of non-isothermal  binary incompressible immiscible fluids (cf.\ \cite{ERS}).

 Let us also mention that other approaches to treat PDE systems with an  $L^1$-right-hand side
 are available
 in the literature: among others, we refer to \cite{zimmer}, 
 resorting to the notion of {\em renormalized solution},
   \EEE
 and \cite{roubiSIAM10} where
 the coupling of rate-independent and thermal processes is considered. The heat equation therein, with an  $L^1$-right-hand side, is tackled
by means of Boccardo-Gallo\"uet type techniques. \EEE

\paragraph{\bf The weak formulation of the damage flow rule.} Following the lines of \cite{hk1, hk2}, we replace the damage inclusion \eqref{e:z} by 
		the {\sl damage energy-dissipation
	    inequality} 
	    \begin{align}
	    &\label{energ-ineq-z-intro}
	    \begin{aligned}
	      \int_s^t   \int_{\Omega} |z_t|^2 \dd x \dd r  & +\io\left(
	      \frac1p  |\nabla z(t)|^p +  \sigma(z(t))\right)\dd x\\ & \leq\io\left(
	      \frac1p |\nabla z(s)|^p+ \sigma(z(s))\right)\dd x
	      +\int_s^t  \int_\Omega z_t \left(-
	      \pd{z}(c,\e(\ub), z)
	      +\teta\right)\dd x \dd r,
	    \end{aligned}
	    \end{align}    
	   imposed \EEE for all $t \in (0,T]$, for $s=0$, and for almost all $0< s\leq t$ and the {\sl one-sided variational inequality for the damage process}
	    \begin{align}
	    \label{var-ineq-z-intro}
	      &\begin{aligned}
	      \int_\Omega  \Big( z_t  \zeta +|\nabla z|^{p-2} \nabla z \cdot \nabla \zeta  + \xi \zeta +
	      \sigma'(z(t)) \zeta & + \pd{z}(c,\e(\ub), z) \zeta -\teta \zeta \Big)\,\mathrm{d}x % \, \mathrm{d}t
	      \geq 0  \quad \aein\, (0,T),
	    \end{aligned}
	    \end{align}
	    required to be valid \EEE for all sufficiently regular test functions $\zeta$, where $\xi \in \partial I_{[0,+\infty)}(z)$ $\aein\, Q$, and $z(x,t) \in [0,1]$, $z_t(x,t)\in(-\infty,0]$ $ \aein Q.$
 \paragraph{\bf Entropic weak solutions.} In what follows, we shall refer to the formulation consisting of (\ref{e:c}--\ref{e:mu}), \eqref{e:u}, \eqref{entropy-ineq-intro}, \eqref{total-enid-intro}, \eqref{energ-ineq-z-intro}, \eqref{var-ineq-z-intro}, supplemented with  the initial and boundary conditions \eqref{init-bdry-conditions}, 
% for $c$ and $\ub$, 
as the \emph{entropic weak formulation} of (the initial-boundary value problem for) system \eqref{eqn:PDEsystem}.  
Let us point out  that, in case of regular solutions, it can be seen that the {\sl ``entropic'' formulation} is equivalent to the internal energy balance \eqref{e:teta} (cf.\  
Remark
\ref{rmk:weak-sol} 
as well as
\cite[Rmk. 2.6]{RocRos14} for more details). Likewise,   the {\sl weak formulation of the damage flow rule} would give rise to the damage inclusion \eqref{e:z}  for sufficiently regular solutions. \EEE  In this sense, we can observe that our formulation is  consistent \EEE with the PDE system \eqref{eqn:PDEsystem-expli}.

\paragraph{\bf  Our results and  related literature.} In this paper we prove existence of global-in-time   entropic weak  \EEE solutions
% for the initial boundary value problem coupling  with t
under the following assumptions on the  data: 
\begin{itemize}
\item[-] the mixing potential $\phi$ is the sum of a convex possibly non-smooth part and a regular $\lambda$-concave part (cf.~Hyp.~(I)). Hence, both the sum of a logarithmic potential (e.g.~ $(1+c)\log(1+c)+(1-c)\log(1-c)$) or an indicator function (e.g.~$I_{[-1,1]}(c)$) and a smooth concave perturbation (e.g.~$-c^2$) are allowed as choices of $\phi$,  cf.\ Remark   \ref{rmk:l-convex-splitting} ahead); \EEE
\item[-] the mobility $m$ is a smooth function bounded from below by a positive constant;
\item[-] the function $\sigma$ is regular;
\item[-] the heat conductivity $\condu$ is a continuous function growing like a power of $\teta$. This choice is motivated by mathematics, indeed it is needed  in order to get  suitable estimates \EEE on the temperature $\teta$, but it is also justified  by \EEE the physical behavior of certain materials (cf.~\cite{klein,zr}); 
\item[-] the function $a$ is bounded away from zero and bounded from above as well as its partial derivatives with respect to both $c$ and $z$. These assumptions are mainly made in order to prevent the full degeneracy of the momentum balance \eqref{e:u} and in order to obtain from it   the \EEE sufficient regularity on $\ub$ needed to handle the nonlinear coupling with the temperature and damage relations.  Instead, the coefficient $b$ in the elastic energy 
 density \EEE
\eqref{elastic-energy} can possibly vanish, and both $b$ and the eigenstrain $\e^*$  are required to \EEE be sufficiently regular functions;
\item[-] the thermal expansion coefficient $\rho$ is assumed to be a positive constant. For more general behavior of $\rho$ possibly  depending  \EEE on the damage parameter $z$ the reader can refer to \cite{hr}, while the fact that $\rho$ is chosen to be independent of $\teta$ is justified by the fact that we  assume  \EEE to have a constant specific heat $c_v$ (equal to 1 in \eqref{e:teta} for simplicity): indeed they are related (by thermodynamical laws) by the relation $\partial_\teta c_v=\teta\partial_\teta\rho$;
\item[-] the initial  data are taken in the energy space, except  for
  \EEE the initial displacement and velocity which,  jointly \EEE with the boundary Dirichlet datum for $\ub$, must  enjoy \EEE the regularity needed in order to perform  %the aforementioned 
  \EEE elliptic regularity estimates on the momentum balance \eqref{e:u}. 
\end{itemize}
 Furthermore, we consider a \emph{gradient theory} for damage.   From the physical viewpoint,  the term $\frac{1}{p}|\nabla z|^p$  contributing to 
\eqref{free-energy} models \EEE
nonlocality of  the damage process, since the gradient of  $z$
accounts for the influence of damage at a material point, undamaged in its neighborhood.   The mathematical advantages attached to the presence  of this term, and of the analogous contribution
$\frac{1}{p}|\nabla c|^p$, 
 are rather obvious. Let us mention that, in fact, 
 throughout  the paper we  shall assume that the exponent $p$ in \eqref{e:mu} and \eqref{e:z} fulfills $p>d$. This assumption is mainly mathematically motivated by the fact that it ensures that $c$  and $z$ are 
  estimated in $W^{1,p}(\Omega) \subset \mathrm{C}^0
  (\overline\Omega)$, and has been 
 adopted for the analysis of  other damage models (cf., e.g., \cite{bmr,MieRou06,mrz,krz2}).
 \par
Regarding the previous results on this type of problems in the literature,
 let us point out that, by now,  \EEE several contributions on systems coupling rate-dependent damage and   thermal processes  \EEE (cf.,   e.g.~\cite{BoBo, RocRos12, RocRos14, hr}) \EEE as well as rate-dependent damage and phase separation (cf., e.g., \cite{hk1,hk2})  are available in the literature. 
Up to our knowledge, this is  one of  the  first contributions on the
analysis of  a model encompassing 
all of  the three processes  (temperature evolution, damage, phase separation) in
  a thermoviscoelastic \EEE material.
 Recently, \EEE
 a thermodynamically consistent,  quite general \EEE   model describing
 diffusion of a solute or a fluid in a solid undergoing possible phase
 transformations and rate-independent damage, beside possible visco-inelastic processes, has been  studied   in \cite{TomRou}.  Let us highlight the main difference to our own model: 
the evolution of the damage process is therein considered \emph{rate-independent}, which clearly affects the 
  weak solution concept adopted in \cite{TomRou}. In particular, we may point out   
  that dealing with a \emph{rate-dependent} flow rule for the damage variable is one of the challenges of our own analysis,  \EEE due to the presence of the quadratic nonlinearity in $\e(\ub)$   on the right-hand side of \EEE \eqref{e:z}.
  % Secondly, the damage flow rule is derived in a different way: in particular, it does not feature 
   %in $z$, it is also not including the one anayzed here due which cannot be handled with the techniques used in \cite{TomRou}.
   \par
     Let us conclude by  mentioning \EEE some open problems which are currently under study, such  as  uniqueness of solutions, at least for the isothermal case, and 
    the global-in-time existence analysis  for \EEE the complete damage (degenerating) case, in which  the coefficient $a$ in the momentum balance \eqref{e:u} is allowed to vanish in some parts of the domain (cf.~\cite{RocRos12} for the case without phase separation and \cite{hk3} for the isothermal case).

\paragraph{\bf Plan of the paper.} In \underline{Section~\ref{s:3}},   after   listing all  the assumptions on the data of the problem, we   rigorously state  the \emph{entropic weak} formulation of the problem and   give \EEE the main result of the paper,  i.e.\ \EEE Theorem  \ref{thm:1}  ensuring the  global-in-time existence of  entropic weak solutions. 
\par
In
\underline{Section~\ref{s:4}} we (formally) derive all  the  a priori estimates on system \eqref{eqn:PDEsystem}   which will be  at the core of our existence analysis. \EEE  
\par
 As previously mentioned, Thm.\   \ref{thm:1}  is proved by passing to the limit in a carefully devised time-discretization scheme,  also coupled with regularization procedures, which could also be of interest in view of possible numerical simulations on the model. 
   To its analysis, the whole \underline{Section 
\ref{s:5}}  is devoted. While postponing more detailed comments on its
   features, let us mention here that our time-discrete scheme will be   {\sl
      thermodynamically \EEE consistent}, in that it will ensure 
the validity of  the  discrete versions of the entropy  and energy inequalities  \eqref{entropy-ineq-intro} and \eqref{total-enid-intro}. This will play a crucial role in the limit passage, developed in \underline{Section \ref{s:6}}, where the proof of Theorem  \ref{thm:1}  will be carried out.  \EEE
 %This approach  

\section{\bf Weak formulation and statement of the main result}
\label{s:3}
In this section, first of all we recall some notation and preliminary results that will be used throughout the paper. Next,
we list all of the conditions on the nonlinearities featuring in system \eqref{eqn:PDEsystem}, as well as on the
data $f,\, g,\, h$ and on the initial data. We are thus in the position to
give our notion of weak solution to the initial-boundary value problem for system \eqref{eqn:PDEsystem} and
state our main existence result, Theorem \ref{thm:1}.
%%%%
\subsection{\bf Preliminaries}
\label{ss:3.1}
In what follows, we will suppose that
\begin{equation}
\label{smoothness-omega}
\Omega\subset\RR^d, \quad d\in \{2,3\}, \EEE \ \ \text{is
 a bounded  domain   with \EEE $\mathrm{C}^2$-boundary $\partial\Omega$.}
\end{equation}
This smoothness requirement will allow us to apply regularity  results for elliptic systems, at the basis of a
regularity estimate that we shall perform on the momentum equation and that will have a key role in the proof of our existence result for
system \eqref{eqn:PDEsystem}.
%%%%%%%
\paragraph{\bf Notation for function spaces, norms, operators}
Given a Banach space $X$,
%we shall
%denote by $\|\cdot\|_{X}$ %both
%its norm,
%and the norm of the
%space  $X^3$,
we will  use the symbol $\pairing{}{X}{\cdot}{\cdot}$ for the
duality pairing between $X'$ and $X$.
 Moreover,  we shall denote by
  ${\rm BV}([0,T];X)$ (by $\mathrm{C}^0_{\mathrm{weak}}([0,T];X)$, respectively),
 the space
of functions from $[0,T]$ with values in $ X$ that are defined at
every  $t \in [0,T]$ and  have  bounded variation on  $[0,T]$  (and
are \emph{weakly} continuous   on  $[0,T]$, resp.).
%We shall consider any function $v \in {\rm BV}([0,T];X)$
%to be defined at all $t \in [0,T]$.

Let $\Omega \subset \R^d$ be a bounded domain, $d \in \{2,3\}$.
We set $Q:= \Omega \times (0,T)$ and $\Sigma:=\partial\Omega\times (0,T)$.
We identify both $L^2 (\Omega)$ and $\Ha$ with their dual spaces, and denote by
$(\cdot,\cdot)$ the scalar product in $\R^d$, by $(\cdot,\cdot)_{L^2(\Omega)}$
both the scalar product in $L^2(\Omega)$  and in \EEE $\Ha$, and by
$\boZ$, $\boY$ and $ \Hn(\Omega)$ \EEE the spaces
\begin{align*}
  &\boZ:=\big\{\vv \in H^1(\Omega;\R^d) \,:\ \vv= 0 \ \hbox{ on
    }\partial\Omega \,\big\},
    \text{ endowed with the  norm } \|
    \vv\|_{H_0^1(\Omega;\R^d)}^2: = \int_{\Omega} \e(\vv) \colon \e(\vv)\,\dd x,
  \\
  &\boY:=  \boZ \cap H^2(\Omega; \R^d) = \EEE  \big\{\vv \in H^2(\Omega;\R^d)\,:\ \vv ={0} \ \hbox{ on }\partial\Omega \,\big\},\\
  &\Hn(\Omega):=\big\{v\in H^2(\Omega)\,:\ \partial_n v=0\text{ on }\partial\Omega\big\}.
\end{align*}
Note that by Korn's inequality  $\|\cdot\|_{H_0^1(\Omega;\R^d)}$ is a  norm  equivalent to the standard one on $H^1(\Omega;\R^d)$.
We denote by $\mathcal{D} (\overline Q)$ the space of
the $\rmC^\infty$-functions with compact support on $Q$.
For $q\geq 1$ we will adopt the notation
\begin{equation}
\label{label-added}
	W_+^{1,q}(\Omega):= \left\{\zeta \in
	W^{1,q}(\Omega)\, : \ \zeta(x) \geq 0  \quad \foraa\, x \in
	\Omega \right\}, \quad \text{ and analogously for }
	W_-^{1,q}(\Omega).
\end{equation}

Finally, throughout the paper we shall denote by the symbols
$c,\,c',\, C,\,C'$  various positive constants depending only on
known quantities. Furthermore, the symbols $I_i$,  $i = 0, 1,... $,
will be used as place-holders for several integral terms popping in
the various estimates: we warn the reader that we will not be
self-consistent with the numbering, so that, for instance, the
symbol $I_1$ will occur several times with different meanings.
\paragraph{\bf Preliminaries of mathematical elasticity}
We postpone to Sec.\ \ref{ss:3.2} the precise statement of all assumptions on the \emph{elastic} contribution
 $\pd{\eps}(c,\e(\ub),z)$ \EEE to the elliptic operator in \eqref{e:u}. Concerning the
stiffness tensor $\CC$ (we will take the viscosity tensor to be a multiple of $\CC$, cf.\ \eqref{eqn:assbV} ahead), \EEE we suppose that
%\begin{subequations}
\begin{equation}
\label{ass-elas}
   \CC=(c_{ijkh})
  \in \mathrm{C}^{1}(\Omega;\R^{d \times d \times d \times d})\,
\end{equation}
with coefficients satisfying the classical symmetry and ellipticity
conditions (with the usual summation convention)
\begin{equation}
\label{ellipticity}
\begin{aligned}
c_{ijkh}=c_{jikh}=c_{khij},\qquad \qquad
\exists\,  \nu_0>0 \,:  \quad c_{ijkh} \xi_{ij}\xi_{kh}\geq
\nu_0\xi_{ij}\xi_{ij}  \ \    \forall\, \xi_{ij}\colon \xi_{ij}=
\xi_{ji}.
\end{aligned}
\end{equation}
%\end{subequations}
%where the usual summation convention is used. Moreover, we require
%\begin{equation}
%\label{anisotropy2}
%a_{ijkh}, b_{ijkh} \in L^{\infty}(\Omega)\,, \quad  i,j,k,h=1,2,3.
%\end{equation}
Observe that with \eqref{ellipticity},
we also encompass in our analysis the case of
an anisotropic and inhomogeneous material.
Thanks to \eqref{ellipticity} and to the $\mathrm{C}^2$-regularity of $\Omega$
we have the following elliptic regularity result (cf.\ e.g.\ \cite[Lemma~3.2,  p.\
260]{necas}) or \cite[Chap.\ 6, p.\ 318]{Hughes}):
\begin{align}
\label{cigamma}
\begin{aligned}
	\exists \, c_1,\, c_2>0 \quad \forall\,  \uu \in
	\boY\, : \qquad
	c_{1} \| \uu \|_{H^2(\Omega;\R^d)}
		\leq \|\dive (\CC\eps(\uu))\|_{L^2(\Omega;\R^d)}
		\leq c_{2} \| \uu \|_{H^2(\Omega;\R^d)}\,.
\end{aligned}
\end{align}
Under the assumption that $\uu$ has   prescribed  \EEE boundary values
$\db\in H^2(\Omega;\R^d)$, i.e. $\uu=\db$ a.e. on $\partial\Omega$,
we obtain by  applying  \eqref{cigamma} \EEE to   $\uu-\db$
\begin{align}
\label{H2reg}
\begin{aligned}
	&\exists \,   \widetilde c_1,\, \widetilde c_2>0  \EEE \quad \forall\,  \uu \in H^2(\Omega;\R^d)
		\text{ with }\uu=\db\text{ a.e. on }\partial\Omega\, :\\
	&\qquad\qquad\widetilde c_{1} \| \uu \|_{H^2(\Omega;\R^d)}
		\leq \|\dive (\CC\eps (\uu))\|_{L^2(\Omega;\R^d)}+\|\db\|_{H^2(\Omega;\R^d)}
		\leq \widetilde c_{2}\big(\|\uu\|_{H^2(\Omega;\R^d)} + \|\db\|_{H^2(\Omega;\R^d)}\big)\,.
\end{aligned}
\end{align}

%%%
\paragraph{\bf Useful inequalities}
For later reference, we recall here the Gagliardo-Nirenberg inequality
in a particular case: for
all $r,\,q\in [1,+\infty],$ and for all $v\in L^q(\Omega)$ such that
$\nabla v \in L^r(\Omega)$, there holds
\begin{equation}
\label{gn-ineq}
	\|v\|_{L^s(\Omega)}\leq C_{\mathrm{GN}}
	\|v\|_{W^{1,r}(\Omega)}^{\theta} \|v\|_{L^q(\Omega)}^{1-\theta} \qquad
	\text{ with } \frac{1}{s}=\theta
	\left(\frac{1}{r}-\frac{1}{d}\right)+(1-\theta)\frac{1}{q}, \ \  0
	\leq \theta \leq 1,
\end{equation}
the positive constant $C_{\mathrm{GN}}$ depending only on $d,\,r,\,q,\,\theta$.

We will also make use of the following interpolation inequality from \cite[Thm.\ 16.4, p.\ 102]{LM}
\begin{align}
\label{interpolationIneq}
	\forall\varrho>0\quad\exists\,C_\varrho>0\quad\forall u\in X:\qquad\|u\|_Y\leq \varrho\|u\|_X+C_\varrho\|u\|_Z,
\end{align}
where
$X\subseteq Y\subseteq Z$ are Banach spaces with compact embedding $X\Subset Y$.

Combining  this with  the compact embedding
\begin{equation}
\label{dstar}
     \boY \Subset W^{1,d^\star{-}\eta}(\Omega;\R^d),
    \quad \text{with } d^{\star}=
    \begin{cases} \infty & \text{if }d=2,
    \\
    6 & \text{if }d=3,
    \end{cases}
 \quad \text{for all $\eta >0$},
\end{equation}
(where for $d=2$ we mean that $\boY \Subset W^{1,q}(\Omega;\R^d)$ for all $1 \leq q <\infty$),
%
%and a generalization of Korn's inequality to $W_0^{1,p}(\Omega;\R^d)$-spaces (here $p:=d^*-\eta$; see \cite[Chapter 7]{DD12}),
%
we have
\begin{equation}
\label{interp} \forall\, \varrho>0 \ \ \exists\, C_\varrho>0 \ \ \forall\, \eta>0 \ \
	\forall\, \uu \in \boY\,: \ \
	\|\e(\uu)\|_{L^{d^\star{-}\eta}(\Omega; \R^{d\times d}\EEE)}\leq \varrho
	\|\uu\|_{H^2(\Omega; \R^{d}\EEE)}+C_\varrho\|\uu\|_{L^2(\Omega; \R^{d}\EEE)}.
\end{equation}
We also obtain by interpolation
\begin{equation}
\label{interp2}
	\forall\, \varrho>0 \ \ \exists\, C_\varrho>0 \ \ \forall\, \eta>0 \ \
	\forall\, \uu \in H^1(\Omega;\R^d)\,: \ \
		\|u\|_{L^{d^\star{-}\eta}(\Omega;\R^d)}
		\leq \varrho\|\uu\|_{H^1(\Omega;\R^d)}+C_\varrho\|\uu\|_{L^2(\Omega;\R^d)}.
\end{equation}

We will also resort to  the following \emph{nonlinear}  Poincar\'{e}-type inequality
 (proved in,  e.g.,  \cite[Lemma 2.2]{gmrs}), with  $\mean(w)$ the mean value of $w$:
 \begin{equation}
 \label{poincare-type}
 \forall\, q>0 \quad \exists\, C_q >0 \quad \forall\, w \in H^1(\Omega)\, : \qquad
 \| |w|^{q} w \|_{H^1(\Omega)} \leq C_q (\| \nabla (|w|^{q} w )\|_{L^2(\Omega)} + |\mean(w)|^{q+1})\,.
\end{equation}
%%%%
%%%%
\subsection{Assumptions}
\label{ss:3.2}
We now collect  all the conditions on the functions $\phi,\,m,\,\sigma,\,\condu,\,a,\,W,\,\VV$ in system \eqref{eqn:PDEsystem}.
\par
%%%%%%%%  ************************************************************************
%%%%%%%%
%%%%%%%%	HYPOTHESIS (I)
%%%%%%%%
%%%%%%%%  ************************************************************************
\noindent \textbf{Hypothesis (I).}
Concerning the potential $\phi$ for the concentration variable $c$, we require that
\begin{equation}
\label{potential-phi}
\begin{gathered}
\phi = \widehat{\beta} + \gamma \quad \text{with } \widehat{\beta}: \R \to [0,+\infty] \text{ proper, convex, and l.s.c., with } \widehat{\beta}(0)=0,
\text{ and }
\\
 \gamma \in \rmC^1(\R), \qquad \gamma \text{ $\lambda_{\gamma}$-concave for some $\lambda_{\gamma}\geq0$, and }
\\
\text{such that } \exists\, C_\phi \in \R\, \ \forall\, c \in \mathrm{dom}(\phi) : \ \ \phi(c) \geq C_\phi\,.
\end{gathered}
\end{equation}
In what follows, we will denote the convex-analysis subdifferential $\partial\widehat{\beta}:\R \rightrightarrows \R$ by $\beta$, and by $\mathrm{dom}(\beta)$ the set $\{ c \in \R\, : \ \beta(c)\neq \emptyset\}$.
 From $0\in \mathrm{Argmin}_{r\in \R} \widehat{\beta}(r)$, it follows that $0\in \beta(0)$. \EEE
\begin{remark}[Consequences of Hypothesis (I)]
\upshape
\label{rmk:l-convex-splitting}
For later use we observe that, since  the map $c \mapsto \gamma(c) - \lambda_\gamma\tfrac{c^2}{2}$ is concave,
we have the following convex-concave decomposition for $\phi$:
 \begin{equation}
\label{decomposition} \phi(c)= \ddd{\widehat{\beta}(c) +
\lambda_\gamma \frac{c^2}{2}}{convex}{} + \ddd{\gamma(c) -
\lambda_\gamma\frac{c^2}{2}}{concave}{}\,.
\end{equation}
\end{remark}

\begin{example}
\label{ex:phi}
\upshape
 Admissible  choices  for   $\widehat\beta$ are both the physically meaningful  potentials    $\widehat\beta(c)=(1+c)\log(1+c)+(1-c)\log(1-c)$ and $\widehat\beta(c)=I_{[-1,1]}(c)$,  while $\gamma$ can be a general smooth concave perturbation, e.g.~$\gamma(c)=-\lambda_\gamma c^2$. \EEE
% is allowed by our assumption \eqref{potential-phi}.  
\end{example}
\par\noindent
%%%%%%%%  ************************************************************************
%%%%%%%%
%%%%%%%%	HYPOTHESIS (II)
%%%%%%%%
%%%%%%%%  ************************************************************************
\textbf{Hypothesis (II).}
As for the nonlinear functions $m$ and $\sigma$, we %merely 
suppose that
\begin{align}
	&m \in \mathrm{C}^1 (\R\times\R) \ \text{ and  } \ \exists\, m_0>0 \ \forall\, (c,z) \in \R\times \R \, : \ m(c,z) \geq m_0,
		\label{hyp-m}\\
	&\sigma \in \mathrm{C}^2 (\R).
		\label{hyp-sigma}
\end{align}
%Observe that the requirement that $\mathrm{dom}(\widehat \beta)\subset [0,1]$ is made just for  consistency with the physical meaning of the concentration
%variable $c$. As it will be clear from the proof of Theorem \ref{thm:1}, it has no role in the analysis; moreover,
%The last of \eqref{hyp-sigma} clearly rephrases as $I_{[0,+\infty)} + \sigma$ is bounded from below.
\par
\noindent
%%%%%%%%  ************************************************************************
%%%%%%%%
%%%%%%%%	HYPOTHESIS (III)
%%%%%%%%
%%%%%%%%  ************************************************************************
\textbf{Hypothesis (III)}
The heat conductivity function
\begin{align}
\label{hyp-K}
	\begin{gathered}
		\condu:[0,+\infty)\to(0,+\infty)  \  \text{	is continuous and}\\
	\exists \, c_0, \, c_1>0 \quad\exists\kappa>1   \ \
	\forall\teta\in[0,+\infty)\, :\quad	c_0 (1+ \teta^{\kappa})
		\leq \condu(\teta) \leq c_1 (1+\teta^{\kappa})\,.
\end{gathered}
\end{align}
We will  denote by $\widehat{\condu}$ the primitive $\widehat{\condu} (x):= \int_0^x \condu(r) \dd r $ of $\condu$.
%%%%%%%%%%%%%%
%%%%%%%%%%%%%%
\par
%%%%%%%%  ************************************************************************
%%%%%%%%
%%%%%%%%	HYPOTHESIS (IV)
%%%%%%%%
%%%%%%%%  ************************************************************************
\noindent \textbf{Hypothesis (IV).}
We require
\begin{align}
\label{data-a}
\begin{aligned}
	a \in \mathrm{C}^1(\R\times\R) \quad\text{ and }\quad
	&\exists\, a_0,a_1>0 \quad\forall c, z\in \R\, : \quad 
		&&a_0\leq a(c,z) \leq a_1,\\
	&\exists\, a_2>0 \quad\forall c, z\in \R\, : \quad
		&&|a_{,c}(c,z)|+|a_{,z}(c,z)|\leq a_2.
\end{aligned}
\end{align}
%%%%%%
%%%%
%%%%%%%%  ************************************************************************
%%%%%%%%
%%%%%%%%	HYPOTHESIS (V)
%%%%%%%%
%%%%%%%%  ************************************************************************
\noindent \textbf{Hypothesis (V).}
We suppose that
\begin{align}
\label{eqn:assumptionW}
    W(x,c,\e,z)=\frac12  b(c,z)\CC(x)(\e-\e^*(c)):(\e-\e^*(c)),
\end{align}
where we recall that  $b(c,z)$ models the influence of the concentration and damage on the stiffness tensor $\CC$
and $\e^*$ models the eigenstrain. We assume 
\begin{align}\label{eqn:assbV}
    &\e^*\in \mathrm{C}^2(\R),\qquad b\in \mathrm{C}^2(\R\times\R)
			\quad\text{ and }\quad\exists\, b_0>0\quad\forall c, z\in \R\, : \quad 0\leq b(c,z)\leq b_0,\quad  \VV=\omega\CC, \quad\omega>0. \EEE
\end{align}
The tensor function $\CC$ should satisfy  conditions \eqref{ass-elas} and \eqref{ellipticity}.
 Let us mention in advance that the last condition on $\VV$ will play a crucial role in the proof of $H^2(\Omega;\R^d)$-regularity for the discrete displacements,
cf.\ Lemma \ref{lemma:4.16} ahead. \EEE
\par
For notational convenience, from now on  we  shall  neglect the $x$-dependence of $W$.
 For later reference, we observe that 
\begin{equation}
\label{later-ref}
\begin{aligned}
&
W_{,c}(c,\eps,z) = \frac12 b_{,c}(c,z) \CC(\e-\e^*(c)):(\e-\e^*(c))  - \EEE b(c,z) (\e^*)'(c) \CC :(\e-\e^*(c)),
\\
& 
\begin{aligned}
W_{,cc}(c,\eps,z) =  & \frac12 b_{,cc}(c,z) \CC(\e-\e^*(c)):(\e-\e^*(c)) - b_{,c}(c,z)  (\e^*)'(c) \CC :(\e-\e^*(c)) 
\\
& 
-b(c,z) (\e^*){''}(c) \CC :(\e-\e^*(c))
+ b(c,z) (\e^*)'(c) \CC :(\e^*)'(c),
\end{aligned}
\\
& 
W_{,z}(c,\eps,z) = \frac12 b_{,z}(c,z) \CC(\e-\e^*(c)):(\e-\e^*(c)),
\\
& 
W_{,zz}(c,\eps,z) = \frac12 b_{,zz}(c,z) \CC(\e-\e^*(c)):(\e-\e^*(c)),
\\
& W_{,\eps}(c,\eps,z) = b(c,z) \CC(\e-\e^*(c)),
\\
& %W_{,\eps\eps}(c,\eps,z) = b(c,z) \CC\,.
W_{,\eps c }(c,\eps,z) = b_{,c}(c,z) \CC(\e-\e^*(c)) - b(c,z) (\eps^*)'(c)\CC,
\\
&
W_{,\eps z }(c,\eps,z) = b_{,z}(c,z) \CC(\e-\e^*(c))\,. 
\end{aligned}
\end{equation}
\EEE

Finally, we will suppose throughout the work that $p>d$ and that the data $\db$, $\mathbf{f}$, $g$, and $h$
comply with
\begin{subequations}
\label{hyp:data}
\begin{align}
	&\db\in H^1(0,T;H^2(\Omega;\R^d))\cap W^{1,\infty}(0,T;W^{1,\infty}(\Omega;\R^d))\cap H^2(0,T;H^1(\Omega;\R^d)),
		\label{dirichlet-data}\\
	&\mathbf{f}\in L^2(0,T;\Ha),
		\label{bulk-force}\\
	&g \in L^1(0,T;L^1(\Omega)) \cap L^2 (0,T; H^1(\Omega)'),\quad g\geq 0 \quad\hbox{a.e.  in }Q\,,
		\label{heat-source}\\
  &h \in L^1 (0,T; L^2(\partial \Omega)), \quad h \geq 0 \quad\hbox{a.e.  in }\Sigma\,,
  	\label{dato-h}
\end{align}
\end{subequations}
and that the initial data fulfill
\begin{subequations}
\label{h:initial}
\begin{align}
  &c^0\in   W^{1,p}(\Omega),\quad  \widehat{\beta}( c^0\EEE) \in L^1(\Omega), \quad
  	\mean_0:= \mean( c^0\EEE) \text{ belongs to the interior of } \mathrm{dom}(\beta), 
  %\alpha<c^0<\beta\text{ a.e. in }\Omega,%\phi(c_0)\in L^1(\Omega),
  	\label{data_c}\\
  &z^0\in   W^{1,p}(\Omega),\quad 0 \leq  z^0 \leq 1 \text{ in }\Omega,
  	\label{data_z}\\
  &\teta^0 \in L^{1}(\Omega), \quad \log\teta^0\in L^1(\Omega),\quad\exists\,
    \teta_*>0\,: \;\teta^0\geq\teta_*>0\;\aein\Omega, 
 	 \label{data_teta}\\
  &\ub^0\in  H^2(\Omega;\R^d)\text{ with }\ub^0=\db(0)\;\text{ a.e. on }\partial\Omega,
 		\label{data_u}\\
  &\vb^0\in  H^1(\Omega;\RR^d).
 		\label{data_v}
\end{align}
\end{subequations}
\begin{remark}
	\upshape
	\label{rmk:on-init-data}
	Let us point out  explicitly that,
	 if we choose \EEE $\phi$ as the logarithmic potential from Example \ref{ex:phi}, or  with $\phi$ given by the sum
	$I_{[0,1]} + \gamma$,  we   enforce \EEE the
	(physically meaningful) property that $c \in (0,1)$ ($c\in [0,1]$, respectively) in $\Omega$.
	From \eqref{data_c} we read that this constraint has to be enforced on the initial datum $ c^0$ \EEE as well, in the same was as we require $ z^0\in [0,1]$ \EEE with \eqref{data_z}.
	 
	The latter condition,
	combined with the information that $z(\cdot,x)$ is nonincreasing for almost all $x\in\Omega$ thanks to the term $\partial I_{(-\infty,0]}(z_t)$ in \eqref{e:z},
	will yield that the solution component  $z$ is in $[0,1]$ a.e.\ in $Q$. This property,  albeit \EEE not needed for the analysis of
	\eqref{eqn:PDEsystem},  is in accordance with the physical meaning of the damage parameter.

	Clearly,  in the case
	the concentration variable $c$ is forced to vary between two fixed values,
	and $z$ is forced to be in $[0,1]$, values of the functions $m$, $\sigma$, $a$ and $b$ outside
	 these \EEE ranges  do not affect the PDE  system. 
\end{remark}
%In the following, let the initial data $(c^0,\ub^0, \mathbf
%v^0,z^0,\vartheta^0)$ with $c^0,z^0\in W^{1,p}(\Omega)$, $\ub^0\in
%H_0^2(\Omega;\R^d)$, $\mathbf v^0\in H_0^1(\Omega;\R^d)$ and
%$\vartheta^0\in H^1(\Omega)$ be given.
\subsection{Entropic solutions and main result}
\label{ss:3.3}
Prior to the precise statement of our weak solution notion for
the initial-boundary value problem for system
\eqref{eqn:PDEsystem}, we shortly introduce and motivate its main ingredients, namely a suitable
weak formulation of the flow rule \eqref{e:z} for the damage variable
and the ``entropic'' formulation of the heat equation \eqref{e:teta}. To them, the standard weak formulation of the Cahn-Hilliard equation, and the pointwise (a.e.\ in $Q$) momentum equation will be coupled.
\\
\paragraph{\bf Entropy and total energy inequalities for the heat equation}
Along the footsteps
of
\cite{fei, fpr09}, cf.\ also \cite{RocRos14} in the case of a PDE system in thermoviscoelasticity,  we will weakly formulate
\eqref{e:teta}  by means of an ``entropy inequality'', and of a   ``total energy (in)equality''.  The
former is obtained by testing  \eqref{e:teta} by $\varphi/\teta$, with  $\varphi$ a   \emph{positive} \EEE smooth test function.
Integrating over space and time leads to
\begin{equation}
\label{later-4-comparison}
\begin{aligned}
  &
  \begin{aligned}
      \int_0^T \int_\Omega \big(\partial_t \log(\teta) + c_t + z_t  & + \rho \mathrm{div}(\uu_t) \big) \varphi \dd x \dd t
      +\int_0^T \int_\Omega \mathsf{K}(\teta) \nabla \log(\teta)\cdot\nabla \varphi  \dd x \dd t
	  \\ & -  \int_0^T \int_\Omega  \mathsf{K}(\teta) \frac{\varphi}{\teta}  \nabla \log(\teta) \cdot \nabla \teta  \dd x \dd t
  \end{aligned}
  \\
  &
  = \int_0^T \int_\Omega  \big(g + |c_t|^2+ |z_t|^2 + a(c,z) \eps(\uu_t):\vism \eps(\uu_t) + m(c,z)|\nabla \mu|^2\big) \frac\varphi\teta  \dd x \dd t
  + \int_0^T \int_{\partial\Omega} h \frac\varphi\teta  \dd S \dd t
\end{aligned}
\end{equation}
for all $\varphi \in \mathcal{D}(\overline Q)$. Then, the entropy
inequality \eqref{entropy-ineq} below follows.

The total energy inequality   (cf.\ the forthcoming \eqref{total-enid})  associated with system
\eqref{eqn:PDEsystem} corresponds to its
standard \emph{energy} estimate. Formally, it is indeed  obtained by
 testing \eqref{e:c} by $\mu$, \eqref{e:mu}
by $c_t$, \eqref{e:z} by $z_t$, \eqref{e:teta} by $1$, and \eqref{e:u} by $\uu_t$, and it features the total energy \eqref{total-energy} of the system.\\
\paragraph{\bf Weak flow rule for the damage parameter}
We will adopt the solution notion from \cite{hk1,hk2},
which can be motivated by
observing that, due to the convexity of $I_{(-\infty,0]}$, the  flow rule \EEE \eqref{e:z}
reformulates as  $z_t\leq 0$ a.e.\ in $Q$ and
\begin{subequations}
\label{ineq-system}
\begin{align}
\label{ineq-system2}
	\Big( z_t-\Delta_p(z)+\xi + \sigma'(z)+\pd{z}(c,\e(\ub),z)-\teta\Big) \zeta \geq{}&
     0\quad  && \qquad \aein Q,\text{ for all } \zeta \leq 0,
	\\
	\label{ineq-system3}
	\Big( z_t-\Delta_p(z)+
	\xi + \sigma'(z)    +\pd{z}(c,\e(\ub),z) -\teta\Big) z_t \leq{}& 0 &&   \qquad
	\aein Q,
\end{align}
\end{subequations}
with $\xi \in \partial I_{[0,+\infty)}(z)$ in $\Omega
\times (0,T)$.
Our weak formulation of \eqref{e:z} in fact consists of
the condition $z_t \leq 0$,
of the integrated version of \eqref{ineq-system2}, with negative test functions from
$W^{1,p}(\Omega)$, and of the \emph{damage energy-dissipation} inequality obtained by integrating \eqref{ineq-system3}.

We are now in the position to give the following notion of weak solution:
%%%%%%%%  ************************************************************************
%%%%%%%%
%%%%%%%%	DEFINITION WEAK SOLUTION
%%%%%%%%
%%%%%%%%  ************************************************************************
\begin{definition}[Entropic weak formulation]
\label{def-entropic}
	Given  data $(\db, \fb, g, h)$ fulfilling \eqref{hyp:data} 
	and initial  values  \linebreak
	$ (c^0,z^0,\teta^0,\ub^0,\vb^0)  $ \EEE fulfilling \eqref{h:initial}, we call a quintuple
	$(c, \mu, z,\vartheta,\ub)$  an \emph{entropic weak solution} to the PDE system
	\eqref{eqn:PDEsystem}, supplemented with the initial and boundary conditions \eqref{init-bdry-conditions},
	if
	\begin{align}
  	&c\in L^\infty(0,T;W^{1,p}(\Omega))\cap H^1(0,T;L^2(\Omega)),\,
			\Delta_p(c)\in L^2(0,T;L^2(\Omega)),\label{reg-c}\\
  	&\mu\in L^2(0,T;\Hn(\Omega)),\label{reg-mu}\\
  	&z\in L^\infty(0,T;W^{1,p}(\Omega))\cap H^1(0,T;L^2(\Omega)),\label{reg-z}\\
  	&\teta\in L^2(0,T;H^1(\Omega))\cap L^\infty(0,T;L^1(\Omega)),\,
	 	\teta^{\frac{\kappa+\alpha}{2}}\in L^2(0,T;H^1(\Omega))\text{ for all }\alpha\in(0,1),
	  	\label{reg-teta}\\
	  &\ub\in H^1(0,T; H^2(\Omega;\R^d))\cap W^{1,\infty}(0,T; H^1(\Omega;\R^d))\cap H^2(0,T;L^2(\Omega;\R^d)),\label{reg-u}
	\end{align}
	and subgradients (specified in \eqref{eta-beta} and \eqref{xi-def} below)
	\begin{align}
	&\eta \in L^2(0,T;L^2(\Omega)),\\
	&\xi \in L^2(0,T;L^2(\Omega)),
	\end{align}
	
	where $(c,z,\teta,\uu)$ comply
	the initial conditions
	(note that the initial condition for $\teta$ is implicitly formulated in \eqref{total-enid}
	below)
	%\eqref{init-conditions}
	\begin{align}
	\label{better-init}
	    &&&c(0)=c^0,
	    &&z(0)=z^0,
	      &&\ub(0)=\ub^0,
	    &&\ub_t(0)=\vb^0
	    &&\text{a.e. in }\Omega,&&
	\end{align}
	the Dirichlet condition
	\begin{align}
	\label{boundary-cond}
		\uu=\db\quad\text{ a.e. on }\partial\Omega\times(0,T)
	\end{align}
	
	and  the following  relations: \EEE
	\begin{itemize}
	  \item[(i)] Cahn-Hilliard system: %$c\in (\alpha,\beta)$ a.e.\ in $Q$,
	    \begin{align}
	      c_t={}&\dive(m(c,z)\nabla\mu)
	      	&&\aein\,  Q,\label{ch-1}\\
	    \mu ={}&-\Delta_p(c)+\eta + \gamma'(c)+W_{,c}(c,\e(\ub),z)-\vartheta+c_t
	    	&&\aein\,  Q,\label{ch-2}\\
	%      \int_\Omega \mu\psi \dd x ={}&\int_\Omega|\nabla c|^{p-2}\nabla c\cdot\nabla\psi+\eta  \psi + \gamma'(c)\psi\dx\notag\\
	%        &+\int_\Omega W_{,c}(c,\e(\ub),z)\psi-\vartheta\psi+c_t\psi\dd x \quad
	%        \text{for all }\psi\in W^{1,p}(\Omega),\;\aein (0,T),
	    	\eta \in{}& \partial  \hat{\beta}(c) \EEE &&\aein\, Q;
	    \label{eta-beta}
	    \end{align} 
	  \item[(ii)]
	    balance of forces:
	    \begin{align}
	      &\ub_{tt}-\dive\sigmab=\mathbf{f}
	      	&&\qquad\qquad\aein\, Q,
	      	\label{momentum-a.e.}\\
	      &\sigmab=a(c,z)\vism\e(\ub_t)+W_{,\e}(c,\e(\ub),z)-\rho\vartheta\mathds 1
	      	&&\qquad\qquad\aein\, Q;
	      	\label{stress-tensor}
	    \end{align}
	  \item[(iii)]
	    weak formulation of the damage flow rule:\\
		{\sl damage energy-dissipation
	    inequality} for all $t \in (0,T]$, for $s=0$, and for almost all $0< s\leq t$
	    \begin{align}
	    &\label{energ-ineq-z}
	    \begin{aligned}
	      \int_s^t   \int_{\Omega} |z_t|^2 \dd x \dd r  & +\io\left(
	      \frac1p  |\nabla z(t)|^p +  \sigma(z(t))\right)\dd x\\ & \leq\io\left(
	      \frac1p |\nabla z(s)|^p+ \sigma(z(s))\right)\dd x
	      +\int_s^t  \int_\Omega z_t \left(-
	      \pd{z}(c,\e(\ub), z)
	      +\teta\right)\dd x \dd r
	    \end{aligned}
	    \end{align}    
	    and the {\sl one-sided variational inequality for the damage process}
	    \begin{align}
	    \label{var-ineq-z}
	      &\begin{aligned}
	      \int_\Omega  \Big( z_t  \zeta +|\nabla z|^{p-2} \nabla z \cdot \nabla \zeta  + \xi \zeta +
	      \sigma'(z(t)) \zeta & + \pd{z}(c,\e(\ub), z) \zeta -\teta \zeta \Big)\,\mathrm{d}x % \, \mathrm{d}t
	      \geq 0 \\ &  \text{for all }  \zeta\in W_-^{1,p}(\Omega), \quad \aein\, (0,T),
	    \end{aligned}
	    \end{align}
	    where% $\xi \in \partial I_{[0,+\infty)}(\chi)$ in the sense that
	    \begin{align}
	    \xi \in \partial I_{[0,+\infty)}(z)\qquad\aein\, Q,
	    \label{xi-def}
	%    \begin{aligned}
	%      &\xi \in L^1(0,T;L^1(\Omega)),\\
	%      &\int_\Omega \xi(\zeta-z) \dd x \leq 0
	%      \quad\text{for all }\zeta \in W_+^{1,p}(\Omega),\;\aein (0,T)
	%    \end{aligned}
	    \end{align}
	    as well as the  constraints
	    \begin{align}
	    \label{constraint-chit}
	      &z \in [0,1],\qquad
	      z_t\in(-\infty,0] \qquad \aein Q;
	    \end{align}
	  \item[(iv)]
	    strict positivity and  entropy inequality: %there exists $\ul\teta>0$ such that
	    \begin{equation}
			\label{strict-pos-teta}
	    \exists\,\underline{\teta}>0  \    \forae\, (x,t) \in Q\, : \ \
	      \teta(x,t)\geq  \underline{\teta}>0
	    \end{equation}
			and	for almost all $0\leq s \leq t \leq T$, and for $s=0$ the entropy inequality holds: 
	    \begin{equation}
	        \label{entropy-ineq}
	        \begin{aligned}
	          &\int_s^t \int_\Omega (\log(\teta) + c+z) \varphi_t  \dd x \dd r  -
	          \rho \int_s^t \int_\Omega \dive(\uu_t) \varphi  \dd x \dd r
	          -\int_s^t \int_\Omega  \condu(\teta) \nabla \log(\teta) \cdot \nabla \varphi  \dd x \dd r\\
	          &\begin{aligned}
	  	        \leq
	   	       	\int_\Omega (\log(\teta(t))+c(t)+z(t)){\varphi(t)} \dd x
	           	&-\int_\Omega (\log(\teta(s))+c(s)+z(s)){\varphi(s)} \dd x\\
	            &-\int_s^t \int_\Omega \condu(\teta)|\nabla\log(\teta)|^2\varphi\dd x \dd r
	          \end{aligned}\\
	          &\quad-\int_s^t  \int_\Omega \left( g +|c_t|^2+ |z_t|^2  + a(c,z) \eps(\uu_t):\vism \eps(\uu_t) + m(c,z)|\nabla \mu|^2\right)
	          \frac{\varphi}{\teta} \dd x \dd r
	          -\int_s^t \int_{\partial\Omega} h \frac\varphi\teta  \dd S \dd r
	        \end{aligned}
	      \end{equation}
	      for all $\varphi \in \mathrm{C}^0 ([0,T]; W^{1,d+\epsilon}(\Omega))  \cap H^1 (0,T; L^{({d^\star})'}(\Omega))$
	      for some $\epsilon>0$,  with $\varphi \geq 0$;
	  \item[(v)]
	    total energy inequality for almost all $0\leq s \leq t \leq T$, and for $s=0$:
	    \begin{equation}
	    \label{total-enid}
	    \begin{aligned}
	      \tE{c(t)}{z(t)}{\teta(t)}{\ub(t)}{\ub_t(t)}
	      	\leq{}&\tE{c(s)}{z(s)}{\teta(s)}{\ub(s)}{\ub_t(s)}\\
			     &+ \int_s^t\int_\Omega g \dd x \dd r
			     + \int_s^t\int_{\partial\Omega} h \dd S \dd r\\
	 		     &+ \int_s^t \int_\Omega \mathbf{f} \cdot \mathbf{u}_t \dd x \dd r
	    		 + \int_s^t \int_{\partial\Omega}\big(\sigmab { \bf  n \EEE} \big)\cdot \db_t \dd S \dd r,
	    \end{aligned}
	    \end{equation}
	    where for $s=0$ we read $\teta(0)= \teta^0$,  and $\mathscr{E}$ is given by \eqref{total-energy}. \EEE
	\end{itemize}
	%and $0\leq c^0\leq 1$ and $\vartheta^0\geq\eta>0$ be given.
\end{definition}

\noindent
\begin{remark}
\label{rmk:weak-sol}
	A few comments on Definition \ref{def-entropic} are in order:
	\begin{itemize}
	  \item[--]
	    First of all, observe that inequalities \eqref{var-ineq-z} and \eqref{energ-ineq-z}
	    yield  the \emph{damage variational inequality} (with $\xi$ fulfilling \eqref{xi-def})
	    \begin{equation}
	    \label{dam-var-ineq}
	    \begin{aligned}
	      \int_s^t\int_\Omega|\nabla z|^{p-2}\nabla z\cdot\nabla\zeta \dd x \dd r  & -\int_\Omega\frac 1p|\nabla z(t)|^p \dd x   +\int_\Omega\frac 1p|\nabla z(s)|^p \dd x\\
	      &+\int_s^t\int_\Omega\Big(z_t(\zeta-z_t)+\sigma'(z)(\zeta-z_t)+\xi(\zeta-z_t)\Big) \dd x \dd r \\
	      &\geq \int_s^t\int_\Omega\Big(-W_{,z}(c,\e(\ub),z)(\zeta-z_t)+\vartheta(\zeta-z_t)\Big) \dd x \dd r
	    \end{aligned}
	    \end{equation}
	    for all $t \in (0,T]$, for $s=0$,  and for almost all $0< s\leq t$ and
	    for all test functions $\zeta \in L^p (0,T; W_-^{1,p}(\Omega)) \cap L^\infty (0,T; L^\infty(\Omega))$.
		\item[--]
	    Concerning the \emph{entropic} formulation (=entropy+total energy inequalities) of the heat
	    equation, we point out that it is consistent
	    with the classical  one.  Namely, \EEE if  \EEE the functions $\teta,\, c,\, z$ are sufficiently smooth,
	    then inequalities \eqref{entropy-ineq} and \eqref{total-enid}, combined with
	    \eqref{e:c}--\eqref{e:z} and \eqref{e:u} yield the pointwise formulation of
	    \eqref{e:teta}, cf.\ \cite[Rmk.\ 2.6]{RocRos14} for all details.
	  \item[--]
	    Observe that the \emph{damage energy-dissipation} inequality \eqref{energ-ineq-z} is
	    required to hold for all $t\in (0,T]$ and for almost all $0 \leq s<t$, and $s=0$.
	    Indeed we will not be able to improve it to an equality, or to an inequality holding
	    on \emph{every} subinterval $[s,t]\subset[0,T]$. This is due to the fact that we will
	    obtain \eqref{energ-ineq-z} by passing to the limit in its time-discrete version
	    (cf.\ Lemma \ref{l:energy-est}), exploiting lower semicontinuity arguments to take the
	    limit of the left-hand side, and pointwise, almost everywhere in $(0,T)$, convergences
	    to take the limit of the right-hand side.
	    Analogous considerations apply to the \emph{entropy} and \emph{total energy} inequalities
	    \eqref{entropy-ineq} and \eqref{total-enid}.
	  \item[--]
	    We remark that the  \emph{damage energy-dissipation} and the \emph{total energy}
	    inequalities are  obtained independently one of another: while this will be clear from the
	    proof of Theorem \ref{thm:1} below, we refer to \cite[Rmk.\ 2.8]{RocRos14}
	    and \cite[Sec. 2.4]{RocRos12} for further comments.    
	  \item[--]
	  	The quasi-linear $p$-Laplacian operator $\Delta_p:W^{1,p}(\Omega)\to W^{1,p}(\Omega)'$
	  	with homogeneous Neumann conditions occurring in \eqref{ch-2} is defined in the distributional
	  	sense as
	  	$$
	  		\langle - \EEE\Delta_p(v),w\rangle_{W^{1,p}(\Omega)}=\int_\Omega|\nabla v|^{p-2}\nabla v\cdot\nabla w\dx.
	  	$$
	  	However, since $\Delta_p(c)\in L^2(0,T;L^2(\Omega))$ due to \eqref{reg-c},
	  	the Cahn-Hilliard system can be interpreted in a pointwise formulation.
	  	In view of the regularity result \cite[Thm.\ 2, Rmk.\ 3.5]{savare98},
	  	we infer the enhanced regularity
		  \begin{align*}
	      c \in L^2 (0,T; W^{1+\sigma,p}(\Omega)) \qquad \text{for all } 1 \leq \sigma< \frac1p.
	  	\end{align*}
	  \item[--]
	  	All the terms in the total energy inequality \eqref{total-enid}
	  	 have \EEE a physical interpretation:
	  	The second and the third term on the  right-hand \EEE side of \eqref{total-enid}
	  	describe energy changes due to external heat sources.
	  	The integrand $\fb\cdot\uu_t$   in \EEE the fourth term on the right-hand side of \eqref{total-enid}
	  	specifies the power expended  by \EEE the external volume force $\fb$,
	  	whereas the integrand $\big(\sigmab n\big)\cdot \db_t$ of the fifth term
	  	indicates the power expended  by \EEE the time-dependent Dirichlet data $\db$
	  	on the boundary $\partial\Omega$
	  	(remember that $\sigmab$ is the stress tensor given in \eqref{stress-tensor}).
	  	
	\end{itemize}
\end{remark}

We  can now state our existence result for the entropic formulation of system
\eqref{eqn:PDEsystem}. Observe that, while the basic time-regularity for $\teta$  (in fact for $\log(\teta)$) is  poor in the general case,
under an additional restriction on the exponent $\kappa$ from Hypothesis (III)
we will be able to obtain $\BV$-time regularity for $\teta$.
           %%%
\begin{theorem}
\label{thm:1}
  Assume \textbf{Hypotheses (I)--(V)}, and let the data $(\db,\mathbf{f}, g, h)$
  comply with \eqref{hyp:data}. Then, for any quintuple
  $(c^0,z^0,\teta^0, \ub^0,\vb^0)$ fulfilling
  \eqref{h:initial} there exists an entropic weak solution
  $(c,\mu,z,\teta,\uu)$ 
   to the PDE system
	\eqref{eqn:PDEsystem}, supplemented with the initial and boundary conditions \eqref{init-bdry-conditions},  such that \EEE
  %initial-boundary value problem (\ref{eqn:PDEsystem}, \ref{init-conditions}, \ref{bdry-conditions}), such that
  \begin{align}
    &\label{BV-log}
      \log(\teta) \in L^\infty(0,T;W^{1,d+\epsilon}(\Omega)') \qquad \text{for all } \epsilon >0.
%      &\label{enhanced-4-c}
%        c \in L^2 (0,T; W^{1+\sigma,p}(\Omega)) \qquad \text{for all } 1 \leq \sigma< \frac1p.\\
  \end{align}
	Furthermore, if in addition the exponent $\kappa$ in \eqref{hyp-K} satisfies
	\begin{equation}
	\label{range-k-admissible}
	  \kappa \in (1, 5/3) \quad\hbox{if $d=3$ and } \kappa \in (1, 2) \quad\hbox{if $d=2$ },
	\end{equation}
	then we have
	\begin{equation}
	\label{furth-reg-teta} \teta\in \BV([0,T];
	  W^{2,d+\epsilon}(\Omega)') \qquad \text{for every } \epsilon>0,
	\end{equation}
	 and the total energy inequality \eqref{total-enid} holds  \underline{for all}  $t \in [0,T]$, for $s=0$, and for almost all $s \in (0,t)$.
\end{theorem}

We will prove Theorem \ref{thm:1} 
 throughout  Sections
\ref{s:5} \& \ref{s:6}
by passing to the limit in a carefully devised time discretization scheme
 and several regularizations.
  Namely,   in Section \ref{s:5} we  are going to set up our
time discretization scheme for system
\eqref{eqn:PDEsystem} and perform on it  all the a priori estimates allowing us to prove, in Sec.\ \ref{s:6}, that
(along
a suitable subsequence) the approximate solutions converge to an entropic weak solution to \eqref{eqn:PDEsystem}.
 However, to enhance the readability of the paper in Section \ref{s:4} we will
(formally) perform all estimates on the time-continuous level, i.e.\ on system  \eqref{eqn:PDEsystem} itself. \EEE
%%%%
%%%%%
\section{\bf Formal a priori estimates}
\label{s:4}
%In 
%\par
%More precisely:
 Let us briefly outline all the estimates that will be formally developed on the time-continuous system  \eqref{eqn:PDEsystem}: \EEE
\begin{itemize}
	\item[--]
		in the \underline{\bf First estimate},
		from the (formally written) \emph{total energy identity} (cf.\
		\eqref{calc1} below)
		we will derive a series of bounds on the \emph{non-dissipative} variables $c,\, z,\, \teta,\, \uu $, as well as on $\|\uu_t\|_{L^\infty (0,T; L^2(\Omega;\R^d))}$.
	\item[--]
		Then, with the \underline{\bf Second estimate},
		we shall adapt some calculations first developed in \cite{fpr09} (see also \cite{RocRos14}) to derive a bound for
		$\|\teta\|_{ L^2 (0,T; H^1(\Omega))\EEE}$ via a clever test of the heat equation \eqref{e:teta}.
	\item[--]
		Exploiting the previously obtained estimates,
		in the  \underline{\bf Third estimate} we will obtain bounds for the \emph{dissipative} variables $c_t,\, z_t,\, \e(\uu_t)$, as well as for $\nabla \mu$.
	\item[--]
		The \underline{\bf Fourth estimate} is an elliptic regularity estimate on the momentum equation, along the footsteps of \cite{bss} where it was developed in the case of  a
		\emph{scalar} displacement variable. With this, in particular  we  gain a (uniform in time) bound on $\|\uu\|_{H^2(\Omega;\R^d)}$ which translates into an   (uniform in time)  $L^2(\Omega)$-bound for the term $\pd{c}(c,\e(\ub),z)$ in \eqref{e:mu}.
	\item[--]
		Using this, in the  \underline{\bf Fifth estimate} we obtain a bound on the
		$L^2(0,T;H^1(\Omega))$-norm of $\mu$ from a bound on its mean value $\dashint_\Omega \mu \dd x$, combined  with the previously obtained
		bound for
		$\nabla\mu$ via the Poincar\'e inequality. 
		To develop the related calculations, we will momentarily suppose that
		\begin{equation}
		\label{mir-zelik}
		\begin{gathered}
		\widehat\beta \in \rmC^1(\R) \text{ and satisfies the following property:}
		\\
		\forall\, \mathfrak{m} \in\R\  \exists\, C_{\mathfrak{m}},\, C_{\mathfrak{m}}'>0  %\ \forall\, c \in (\alpha-m,\beta+m) \, :
		 \quad |\beta(c+\mean)|\leq C_{\mean} \beta(c+\mean)c +C_{\mean}'\,.
		 \end{gathered}
		\end{equation} 
	\item[--]
		We are then in the position to obtain a $L^2(0,T; L^2(\Omega;\R^d))$-estimate for each single  term  in \eqref{e:mu} in the \underline{\bf Sixth estimate}.
	\item[--]
		With the \underline{\bf Seventh}  and \underline{\bf Eighth} estimates  we gain some information on the ($\mathrm{BV}$-)time regularity of $ \log(\teta)$
		and $\teta$, respectively (in the latter case, under the  further condition  \eqref{range-k-admissible}  on the growth exponent $\kappa$ of $\condu$).	
	\item[--]
		Finally, in the \underline{\bf Ninth estimate} we  resort to \EEE higher elliptic regularity results to gain
		a uniform bound on $\|\mu\|_{L^2(0,T;H^2(\Omega))}$.
		
\end{itemize}

In the proof of 
 the forthcoming \EEE
Proposition \ref{prop:aprio-discr} we will discuss how to make
all of the following calculations rigorous in
the context of the time-discretization scheme from Definition \ref{def:time-discrete} (let us mention in advance that, for the
\emph{Fifth estimate} we will need the analogue of \eqref{mir-zelik} on the level of the Yosida regularization of $\beta$), 
 with the exception of the computations related to the ensuing
\textbf{Seventh a priori estimate}. Indeed, while in the present time-continuous context this
formal estimate will provide a $\BV$-in-time bound for $\log(\teta)$, on the time-discrete level it will be possible to render it
only 
in a \emph{weaker} form, albeit still
 useful for the compactness arguments developed in Section
 \ref{s:6}.
\par
In the following calculations,  at several spots we will follow the footsteps of \cite{RocRos14},
hence we will give the main ideas, skipping some details and  and referring to the latter paper.
 In comparison to \cite{RocRos14}, the additional coupling with the Cahn-Hilliard
system \eqref{e:c}--\eqref{e:mu} requires new a priori estimates (see the \textbf{Fifth},
\textbf{Sixth} and \textbf{Ninth estimates} below).
Beyond this the remaining system \eqref{e:z}--\eqref{e:u} also depends on the
phase field variable $c$ and the estimation techniques used in \cite{RocRos14} need to
be adapted to this situation.
And, finally, the time-dependent Dirichlet boundary conditions for $\uu$ requires
substantial modifications especially in the \textbf{First}, but also in the \textbf{Third} and
\textbf{Fourth estimates} below.
%%%%
\paragraph{\bf Strict positivity of $\teta$}
Along the lines of \cite{fpr09}, we rearrange terms in
\eqref{e:teta}  and (formally,  disregarding the -positive- boundary datum $h$)
we obtain
\begin{equation}
\label{formal-positivity}
\begin{aligned}
	\teta_t-\dive(\condu (\teta)\nabla\teta)
	 & = g  +|c_t|^2+|z_t|^2  + a( c, z)\eps(\uu_t):\vism \eps(\uu_t) + m(c,z) |\nabla \mu|^2 - c_t\teta-z_t\teta
	- \rho \teta \mathrm{div}(\uu_t)
	\\ & \geq
	g +\frac12|c_t|^2+\frac12|z_t|^2  + c |\eps(\uu_t)|^2 +  m(c,z) |\nabla \mu|^2 -C
	\teta^2
	\geq -C\teta^2 \quad \aein \, Q.
\end{aligned}
\end{equation}
Here, for the  first inequality we have used
that $\vism$ is positive definite by  \eqref{eqn:assbV} and   \EEE 
\eqref{ellipticity}, that $a$ is strictly positive thanks to \eqref{data-a}, and that %  the fact that
\begin{equation}
\label{eps-estim}
	| \dive(\uu_t)  | \leq c(d)
	|\tensoret|  \quad \text{a.e.\ in $Q$}
\end{equation}
 with $c(d)$ a positive
constant only depending on the space dimension $d$. The second  inequality in \eqref{formal-positivity} \EEE  also relies on the fact that $g \geq 0$ a.e.\ in $Q$.
 Therefore we conclude that
  $v$  solving  the Cauchy problem
\[
	v_t=-\frac12 v^2, \quad v(0)=\teta_*>0
\]
is a subsolution of \eqref{e:teta}, and  a comparison argument yields that there exists $\ul\teta>0$ such that
\begin{equation}\label{teta-pos}
  \teta(\cdot,t)\geq v(t)>\ul\teta>0\quad \hbox{for all }t\in [0,T]\,.
\end{equation} 
%%%%%%%%  ************************************************************************
%%%%%%%%
%%%%%%%%	FIRST ESTIMATE
%%%%%%%%
%%%%%%%%  ************************************************************************
\paragraph{\bf First estimate:}
We test \eqref{e:c} by $\mu$, \eqref{e:mu} by $c_t$, \eqref{e:z} by $z_t$,
\eqref{e:teta} by 1, \eqref{e:u} by $\uu_t$,
add the resulting relations and integrate
over the time interval  $(0,t)$, $t\in (0,T]$.
Here the second term in the force balance equation is treated
by integration by parts in space as follows
(notice that $\uu_t=\db_t$ a.e. on $\partial\Omega\times(0,T)$):
\begin{align}
\label{sigmaInt}
\begin{aligned}
	&\itt\io-\dive\big(a(c,z)\vism\e(\ub_t)+W_{,\e}(c,\e(\ub),z)-\rho\vartheta\mathds{1}\big)\cdot\uu_t\dxs\\
		&\qquad=\itt\io a(c,z)\vism\e(\ub_t):\e(\ub_t)+W_{,\e}(c,\e(\ub),z):\e(\ub_t)-\rho\vartheta\dive(\uu_t)\dxs
		-\itt\int_{\partial\Omega}(\sigmab{ \bf  n \EEE} )\cdot\db_t\dd S\ds.
\end{aligned}
\end{align}
Furthermore,
we  use that, by the chain rule,
\begin{align*}
\begin{aligned}
	& \text{(i) } 
		 && \begin{aligned}
	 		&\int_0^t \int_\Omega \pd{c}(c,\e(\ub),z) c_t +  \pd{z}(c,\e(\ub),z) z_t +  \pd{\eps}(c,\e(\ub),z)\colon \e(\ub_t) \dd x \dd s\\
			&= \int_\Omega W(c(t),z(t),\e(\ub(t))) \dd x - \int_\Omega W(c(0),z(0),\e(\ub(0))) \dd x,
	  \end{aligned}\\
	& \text{(ii) } &&  \int_0^t \int_\Omega \left( \eta + \gamma'(c)  \right )  c_t \dd x \dd s = \int_\Omega \phi(c(t)) \dd x - \int_\Omega \phi(c(0))  \dd x,\\
	& \text{(iii) }  &&  \int_0^t \int_\Omega\left(  \partial I_{[0,+\infty)}(z) +\sigma'(z) \right)   z_t \dd x \dd s = \int_\Omega I_{[0,+\infty)}(z(t)) + \sigma(z(t)) \dd x - \int_\Omega  I_{[0,+\infty)}(z(0))  + \sigma(z(0)) \dd x,
\end{aligned}
\end{align*}
%(where we have formally written the multivalued operators $\beta$ and $\partial I_{[0,+\infty)}$ as single-valued),
as well as the identity  $\int_0^t \int_\Omega  \partial
I_{(-\infty,0]}(z_t) z_t \dd x \dd s =  \int_0^t \int_\Omega
      I_{(-\infty,0]}(z_t) \dd x \dd s=  0 $ due to the 
 positive \EEE $1$-homogeneity of
$ \partial I_{(-\infty,0]}$.
Also taking into account the cancellation of a series of terms,
we arrive at the \emph{total energy identity}
\begin{equation}\label{calc1}
\begin{aligned}
	\tE{c(t)}{z(t)}{\teta(t)}{\ub(t)}{\ub_t(t)} ={}& \tE{c_0}{z_0}{\teta_0}{\ub_0}{\vb_0}  +  \int_0^t \int_\Omega g \dd x\ds+
		\int_0^t\int_{\partial\Omega} h \dd S \dd s\\
	 	&+\int_0^t \int_\Omega \mathbf{f} \cdot \mathbf{u}_t \dd x \dd s
	 	+\itt\int_{\partial\Omega}(\sigmab { \bf  n \EEE})\cdot\db_t\dd S\ds\,,
\end{aligned}
\end{equation}
which incorporates the initial conditions \eqref{better-init}.

%From \eqref{calc1}  we derive a first set of a priori estimates observing that  there exist $c,\ C>0$ such that
%\begin{equation}
%\label{coerc-props-W}
% \tE{c}{z}{\teta}{\ub}{\ub_t} \geq c\left(\int_\Omega |\nabla c|^p \dd x  + \int_\Omega |\nabla z|^p \dd x+\int_\Omega \teta \dd x
%	+ \int_\Omega |\ub_t|^2 \dd x \right) -C, 
%\end{equation}
%where we have exploited that $W$ is positive and that, 
%$\phi$ and $I_{[0,+\infty)} + \sigma$ are bounded from below.

We estimate the  second, third and fourth  terms on the right-hand side of \eqref{calc1}
via \eqref{hyp:data} and obtain  
\begin{align*}
\begin{aligned}
	&  \left|  \int_0^t \int_\Omega g \dd x \dd s \right|
	\stackrel{\eqref{heat-source}}{\leq} C,
	\qquad
	\left|  \int_0^t \int_{\partial\Omega} h \dd S \dd s  \right|  \stackrel{\eqref{dato-h}}{\leq} C,\\
	 & \left|   \int_0^t \int_\Omega \mathbf{f} \cdot \mathbf{u}_t  \dd x \dd s  \right|
	 	\stackrel{\eqref{bulk-force}}{\leq}   C
	 		+\|\uu_t\|_{L^2(0,T;L^2(\Omega;\R^d))}^2.
\end{aligned}
\end{align*}
 We now carefully handle the last term on the right-hand side of \eqref{calc1}.  Since \EEE
 no viscous term of the type $\e(\uu_t)$
occurs on its left-hand side, 
to absorb the last term on the right-hand side %of \eqref{calc1} 
 and close the estimate \EEE
we will extensively
make use of integration by parts
%The last term on the right-hand side of \eqref{calc1} can be estimated by using
%integration by parts 
in space,   as well as of \EEE the force balance equation 
\eqref{e:u} of \EEE
integration by parts in time, and  of Young's inequality ($\delta>0$ will be chosen later):
%the properties of the coefficient functions
%$a$ and $b$ stated in Hypothesis (IV) and (V):
\begin{align*}
	\itt\int_{\partial\Omega}(\sigmab{ \bf n \EEE})\cdot\db_t\dd S\ds
	={}&\itt\io\dive(\sigmab)\cdot\db_t\dxs
		+\itt\io \sigmab:\e(\db_t)\dxs\\
	={}&\itt\io(-\fb+\uu_{tt})\cdot\db_t\dxs
		+\itt\io\sigmab:\e(\db_t)\dxs\\
	\leq{}& \|\fb\|_{L^2(0,T;L^2(\Omega;\R^d))}\|\db_t\|_{L^2(0,T;L^2(\Omega;\R^d))}
		+\itt\|\ub_t\|_{L^2(\Omega;\R^d)}\|\db_{tt}\|_{L^2(\Omega;\R^d)}\ds\\
		&+\delta\|\uu_t(t)\|_{L^2(\Omega;\R^d)}^2+C_\delta\|\db_t(t)\|_{L^2(\Omega;\R^d)}^2
		+\|\vb^0\EEE \|_{L^2(\Omega;\R^d)}\|\db_t(0)\|_{L^2(\Omega;\R^d)}\\
		&+\underbrace{\itt\io a(c,z)\VV\e(\uu_t):\e(\db_t)\dxs}_{\doteq I_1}
		+\underbrace{\itt\io b(c,z)\CC(\e(\uu)-\e^*(c)):\e(\db_t)\dxs}_{\doteq I_2}\\
		&+\rho\|\dive(\db_t)\|_{L^\infty(Q)}\itt\io|\teta|\dxs.
\end{align*}
Moreover, by using integration by   parts \EEE in space again,
the properties of the coefficient functions
$a$ and $b$ stated in Hypothesis (IV) and (V),
and by using \eqref{dirichlet-data}  on $\db$, \EEE  $\uu_t=\db_t$ a.e. on $\partial\Omega\times(0,T)$
and the trace theorem  we obtain \EEE 
\begin{align*}
	I_1={}&-\itt\io\uu_t\cdot\dive\big(a(c,z)\VV\e(\db_t)\big)\dxs
			+\itt\int_{\partial\Omega}\uu_t\cdot\big(a(c,z)\VV\e(\db_t)
{ \bf  n \EEE} \big)\dd S\ds\\
		={}&-\itt\io\uu_t\cdot\Big(\big(a_{,c}(c,z)\nabla c+a_{,z}(c,z)\nabla z\big)\cdot\VV\e(\db_t)\Big)\dxs -\itt\io \uu_t \cdot\left(a(c,z)\VV\dive(\e(\db_t))\right)\dxs \EEE\\
		&+\itt\int_{\partial\Omega}\db_t\cdot\big(a(c,z)\VV\e(\db_t)
{ \bf  n \EEE}
\big)\dd S\ds\\
		\leq{}&C\|\e(\db_t)\|_{L^\infty(Q; \R^{d\times d} \EEE)}\Big(\itt\|\uu_t\|_{L^2(\Omega;\R^d)}^2\ds
			+\|a_{,c}(c,z)\|_{L^\infty(0,T;L^\infty(\Omega))}^2\itt\|\nabla c\|_{L^2(\Omega;\R^d)}^2\ds\\
			&\qquad\qquad\qquad\quad+\|a_{,z}(c,z)\|_{L^\infty(0,T;L^\infty(\Omega))}^2\itt\|\nabla z\|_{L^2(\Omega;\R^d)}^2\ds\Big)\\
                           &+C\itt\|\uu_t\|_{L^2(\Omega;\RR^d)}^2\ds+C\|a(c,z)\|^2_{L^\infty(0,T;L^\infty(
                          \Omega))\EEE
}\|\e(\db_t)\|_{L^2(0,T;H^1(\Omega;  
                           \R^{d\times d} ))}^2 \EEE\\
			&+C\|\db_t\|_{L^2(0,T;H^1(\Omega;\R^d))}\|\e(\db_t)\|_{L^2(0,T;H^1(\Omega;\R^{d\times d}))}
				\|a(c,z)\|_{L^\infty(0,T;L^\infty(\partial\Omega))}\\
		\leq{}&C\itt\Big(\|\uu_t\|_{L^2(\Omega;\R^d)}^2+\|\nabla c\|_{L^2(\Omega;\R^d)}^2
			+\|\nabla z\|_{L^2(\Omega;\R^d)}^2\Big)\ds+C,\\
	I_2\leq{}&C\itt\io b(c,z)^2\CC(\e(\uu)-\e^*(c)):(\e(\uu)-\e^*(c))+\CC\e(\db_t):\e(\db_t)\dxs\\
		\leq{}&C\|b(c,z)\|_{L^\infty(0,T;L^\infty(\Omega))}\itt\io W(c,\e(\ub),z)\dxs
			+C\|\e(\db_t)\|_{L^2(0,T;L^2(\Omega;\R^{d\times d}))}^2.
\end{align*}
All in all,  again taking into account \eqref{hyp:data}, \EEE  we gain the estimate
\begin{align*}
	&\tE{c(t)}{z(t)}{\teta(t)}{\ub(t)}{\ub_t(t)}\\
	&\qquad\leq C_\delta+\delta\|\uu_t(t)\|_{L^2(\Omega;\R^d)}^2
		+\itt C\big(\|\db_{tt}\|_{L^2(\Omega;\R^d)}^2+1\big)\times \EEE \\
				&\qquad\qquad\times  \itt \EEE\Big(  \int_\Omega \EEE W(c,\e(\ub),z)\dx
		 		+\|\nabla c\|_{L^2(\Omega;\R^d)}^2+\|\nabla z\|_{L^2(\Omega;\R^d)}^2
		 		+\int_\Omega|\teta|\dx
		 		+\|\ub_t\|_{L^2(\Omega;\R^d)}^2\Big)\ds\\
	&\qquad\leq C_\delta+\delta\|\uu_t(t)\|_{L^2(\Omega;\R^d)}^2
		+\itt C\big(\|\db_{tt}\|^2_{L^2(\Omega;\R^d)}+1\big)\tE{c(s)}{z(s)}{\teta(s)}{\ub(s)}{\ub_t(s)}\ds.\EEE
\end{align*}
Choosing $\delta=1/4$, using Gronwall Lemma together with \eqref{dirichlet-data}
and taking the positivity of $\teta$ into account, we conclude

\begin{equation}
\label{est1}
	\| \teta \|_{L^\infty (0,T;L^1(\Omega))}
	+\| \uu\|_{W^{1,\infty}(0,T;L^2(\Omega;\R^d))}
	+\| c  \|_{L^\infty(0,T;W^{1,p}(\Omega))}
	+\| \nabla z  \|_{L^\infty(0,T;L^{p}(\Omega;\R^d))}
	\leq C.
\end{equation}
Note that we have also used the Poincar\'e inequality to obtain the boundedness
for $c$ in $L^\infty(0,T;W^{1,p}(\Omega))$ because
it holds $\int_\Omega c(t)\dx\equiv const$ for all $t\in[0,T]$
(this follows from \eqref{e:c} and the no-flux condition for  $\mu$  in \eqref{bdry-conditions}). \EEE

%%%%%%%%  ************************************************************************
%%%%%%%%
%%%%%%%%	SECOND ESTIMATE
%%%%%%%%
%%%%%%%%  ************************************************************************
\paragraph{\bf Second estimate:}
Let $F(\teta) = \teta^\alpha/\alpha$, with $\alpha \in (0,1)$.
 We test \eqref{e:teta} by $F'(\teta):= \teta^{\alpha-1}$ ,
and integrate on $(0,t)$ with $t \in (0,T]$, thus obtaining
\[
	\begin{aligned}
	  &\int_\Omega F(\teta_0)\dd x+
		\int_0^t \int_\Omega  g F'(\teta) \dd x \dd s
		+\int_0^t \int_{\partial \Omega} h F'(\teta) \dd S \dd s
		+\int_0^t \int_\Omega (|c_t|^2+ |z_t|^2)   F'(\teta) \dd x \dd s\\
		&+\int_0^t \int_\Omega
		 a(c,z) \tensoret: \vism \tensoret F'(\teta) \dd x \dd s + \int_0^t \int_\Omega
		m(c,z) |\nabla \mu|^2 F'(\teta) \dd x \dd s \\
		&\quad=\int_\Omega F(\teta(t))\dd x + \int_0^t \int_\Omega (c_t + z_t)  \teta F'(\teta) \dd x \dd s +\rho \int_0^t \int_\Omega \teta \dive(\uu_t) F'(\teta)  \dd x \dd  s
	  \int_0^t \int_\Omega \condu(\teta) \nabla \teta\cdot\nabla (F'(\teta)) \dd x \dd s.
  \end{aligned}
\]
By the positivity of $g$ and $h$ we can neglect the second and third terms on the left-hand side, whereas, taking into account  the ellipticity condition \EEE
\eqref{ellipticity} and   the positivity 
\eqref{hyp-m} and 
\eqref{data-a} of $m$ and $a$, \EEE   we infer
\begin{align}
\label{eqn:secondEstPre}
	\begin{aligned}
		&\frac{4(1-\alpha)}{\alpha^2} \int_0^t \int_\Omega
		\condu(\teta) |\nabla (\teta^{\alpha/2})|^2 \dd x \dd s
		+   \bar{c} \EEE\int_0^t\int_\Omega(|\tensoret|^2+|\nabla\mu|^2)F'(\teta) \dd x \dd s
		\\
		& \qquad 
		+ \int_0^t \int_\Omega (|c_t|^2 + |z_t|^2)  F'(\teta) \dd x \dd s 
		\leq \int_\Omega |F(\teta_0)|\dd x +I_1 +I_2+I_3,
	\end{aligned}
\end{align}
with
$\bar{c}>0$ depending on $\nu_0$, $m_0$, and 
 $a_0$, where $I_3 \doteq |\rho| \int_0^t \int_\Omega |\teta \dive(\uu_t) F'(\teta) | \dd x \dd  s$.  \EEE 
We estimate
\[
	\begin{aligned}
		I_1= \int_\Omega |F(\teta(t))|\dd x \leq \frac1{\alpha}\int_\Omega \max\{
		\teta(t), 1\}^\alpha \dd x \leq  \frac1{\alpha}\int_\Omega \max\{  \teta(t), 1\}
		\dd x \leq C
	\end{aligned}
\]
since $\alpha <1$ and taking into account the previously obtained
inequality \EEE
\eqref{est1}. Analogously we can estimate $\int_\Omega
|F(\teta_0)|\dd x$ thanks to \eqref{data_teta}; moreover,
\[
	I_2 = \int_0^t \int_\Omega |(c_t +z_t) \teta F'(\teta)| \dd x \dd s
	\leq \frac14 \int_0^t \int_\Omega \left( |c_t|^2 + |z_t|^2\right) F'(\teta) \dd x \dd s +
	2 \int_0^t \int_\Omega F'(\teta)\teta^2 \dd x \dd s.
\]
Using inequality \eqref{eps-estim} \EEE and Young's inequality, we have that
\[
	\begin{aligned}
		I_3 =|\rho|  \int_0^t \int_\Omega | \teta \dive(\uu_t) F'(\teta)|  \dd
		x \dd  s
		\leq   \frac {  \bar{c}\EEE} 4 \int_0^t \int_\Omega |\tensoret|^2 F'(\teta) \dd x \dd s +
		C\int_0^t \int_\Omega F'(\teta)\teta^2 \dd x \dd s\,.
	\end{aligned}
\]
%with $c_2$ from \eqref{data-a}.
All in all, we conclude
\begin{equation}
\label{all-in-all}
	\begin{aligned}
		&\frac{4(1-\alpha)}{\alpha^2} \int_0^t  \int_\Omega  \condu(\teta)  |\nabla (\teta^{\alpha/2})|^2 \dd x \dd s
		+ \frac{ 3  \bar{c} \EEE}{4}\int_0^t \int_\Omega(|\tensoret|^2 +|\nabla\mu|^2)F'(\teta) \dd x \dd s\\
		&+ \frac34 \int_0^t \int_\Omega (|c_t|^2 + |z_t|^2)  F'(\teta) \dd x \dd s \leq C + C  \int_0^t \int_\Omega \teta^{\alpha+1} \dd x \dd s.
	\end{aligned}
\end{equation}

Observe that
\[
\int_0^t  \int_\Omega  \condu(\teta)  |\nabla (\teta^{\alpha/2})|^2 \dd x \dd s
 \geq
  c_1 \int_0^t \int_\Omega \teta^\kappa  |\nabla (\teta^{\alpha/2})|^2 \dd x \dd s
 =  \tilde c_1
 \int_0^t \int_\Omega  |\nabla (\teta^{(\kappa+\alpha)/2} )|^2   \dd x \dd s\,.
\]
Hence, from \eqref{all-in-all} we infer the estimate
\begin{equation}
	\label{calc2.1}
	 \tilde c_1 \int_0^t \int_\Omega  |\nabla (\teta^{(\kappa+\alpha)/2} )|^2   \dd x \dd s
  \leq C_{ 0} + C_{ 0}\int_0^t\int_\Omega \teta^{\alpha+1} \dd x \dd s.
\end{equation}
We now
 repeat the very same calculations as in the
\emph{Second} and \emph{Third} estimates in \cite[Sec.\ 3]{RocRos14}, to which we refer for all details.
Namely, we introduce the auxiliary quantity
$w : = \max\{ \teta^{(\kappa+\alpha)/2}, 1 \}$ and observe that
\begin{align}
	&\int_0^t \int_\Omega  |\nabla (\teta^{(\kappa+\alpha)/2} )|^2   \dd x
  \dd s \geq \int_0^t \EEE\int_{\{  \teta(s) \EEE \geq 1\}}
  |\nabla (\teta^{(\kappa+\alpha)/2} )|^2\dd x \dd s =  \int_0^t \int_\Omega |\nabla w |^2 \dd x \dd s,\\
  &\teta^{\alpha+1} =\left(  \teta^{(\alpha+1)/q} \right)^q \leq w^q \qquad \aein\, Q,
  	\label{eqn:thetaWEst} 
\end{align}
for all $q \geq 1 $ such that
\begin{equation}
\label{1st-restr-alpha}
	\frac{\kappa+\alpha}2 \geq \frac{\alpha +1}{q}
	\Leftrightarrow \ q \geq 2- 2\frac{\kappa-1}{\kappa+\alpha}.
\end{equation}
Therefore from \eqref{calc2.1} we infer that
\begin{equation}\label{calc2.2}\begin{aligned}
 &  \tilde c_1\int_0^t \int_\Omega |\nabla w|^2 \dd x \dd s
	 \leq C_{ 0} + C_{ 0}  \int_0^t \| w \|_{L^q(\Omega)}^q\dd s. \end{aligned}
\end{equation}
We now apply the Gagliardo-Nirenberg inequality  for $d=3$, yielding
\begin{equation}\label{gagliardo}
\| w \|_{L^q(\Omega)} \leq c_1  \| \nabla w\|_{L^2(\Omega;\R^d)}^\theta
\| w \|_{L^r(\Omega)}^{1-\theta} + c_2   \| w \|_{L^r(\Omega)}
\end{equation}
with $ 1 \leq  r  \leq  q $ and $\theta $ satisfying $1/q= \theta/6
+ (1-\theta)/r$. Hence $\theta= 6 (q-r)/q (6-r)$.
Observe that  $\theta \in (0,1)$ if $q<6$.
Applying the Young inequality with exponents $2/(\theta q)$
and $2/(2-\theta q)$ we infer
\begin{equation}
\label{added-label}
 C_{ 0}\int_0^t \| w \|_{L^q(\Omega)}^q\dd s
 \leq   \frac{\tilde c_1}2\int_0^t \int_\Omega |\nabla w|^2 \dd x \dd s
 +C\int_0^t \| w\|_{L^r(\Omega)}^{2q(1-\theta)/(2-q\theta)} \dd s + C'  \int_0^t \| w\|_{L^r(\Omega)}^{q} \dd s.
\end{equation}

We then plug \eqref{added-label}\, into \eqref{calc2.2},
and obtain
\begin{equation}
\label{added-label2}
 \frac{\tilde c_1}2\int_0^t \int_\Omega |\nabla w|^2 \dd x \dd s
 \leq C_0+C\int_0^t \| w\|_{L^r(\Omega)}^{2q(1-\theta)/(2-q\theta)} \dd s + C'  \int_0^t \| w\|_{L^r(\Omega)}^{q} \dd s.
\end{equation}

%We then plug \eqref{gagliardo} into \eqref{calc2.2}.
%Applying the Young inequality with exponents $2/(\theta q)$
%and $2/(2-\theta q)$ we infer
%\begin{equation}
%\label{added-label}
%  \tilde c_1\int_0^t \int_\Omega |\nabla w|^2 \dd x \dd s \leq   C +  \frac{ \tilde c_1}2\int_0^t \int_\Omega |\nabla w|^2 \dd x \dd s
% +C\int_0^t \| w\|_{L^r(\Omega)}^{2q(1-\theta)/(2-q\theta)} \dd s + C'  \int_0^t \| w\|_{L^r(\Omega)}^{q} \dd s.
%\end{equation}
Hence, we choose $1\leq r \leq 2/(\kappa+\alpha)$ so that for almost all $t\in (0,T)$
\begin{equation}
\label{r-estimate}
	\| w(t) \|_{L^r(\Omega)} = \left( \int_\Omega \max\{ \teta(t)^{r(\kappa+\alpha)/2} ,1\} \dd x \right)^{1/r}  \leq
	C\left( \|\teta(t)\|_{L^1(\Omega)} + |\Omega|\right) \leq C'
\end{equation}
where we have used the bound for $   \| \teta\|_{L^\infty (0,T;L^1(\Omega))}$
from  estimate \eqref{est1}. Observe that the inequalities
\[
	\begin{cases}
		\theta q \leq 2 \ \Leftrightarrow \ 6\frac{q-r}{6-r} \leq 2 \
		\Leftrightarrow \ q \leq 2 +\frac23 r,\\
		r \leq \frac2{\kappa+\alpha}
	\end{cases}
\]
lead to $q\leq 2 +\frac{4}{3(\kappa+\alpha)}$ which is still
compatible with \eqref{1st-restr-alpha}, since
$\frac{\kappa-1}{\kappa+\alpha}<1$. Inserting \eqref{r-estimate}
into \eqref{added-label2}\, we ultimately deduce $ \int_0^t \int_\Omega
|\nabla w|^2 \dd x \dd s \leq C$.
 Taking also \eqref{added-label} and \eqref{r-estimate} into account we 
  then conclude \EEE
  $\int_0^t\|w\|_{L^q(\Omega)}^q\ds \leq C$.
By using this as well as  estimates \eqref{calc2.1} and \eqref{eqn:thetaWEst}
we see that

\begin{equation}
\label{additional-info}
\begin{aligned}
c  \int_0^t \int_\Omega \teta^{\kappa+\alpha  - 2}  |\nabla \teta|^2   \dd x \dd s = \int_0^t \int_\Omega  |\nabla   (\teta^{(\kappa+\alpha)/2})|^2   \dd x \dd s\leq C.
 \end{aligned}
  \end{equation}

From \eqref{additional-info} and the strict positivity of $\teta$  (see 
\eqref{teta-pos})\,
it follows that
$
\int_0^t \int_\Omega |\nabla \teta|^2 \dd x \dd s \leq C,
$
provided that $\kappa +\alpha-2 \geq 0$. Observe that, since $\kappa>1$ we can choose
$\alpha \in (0, 1)$  such that this inequality holds.
%Choosing $\alpha  \in [1/2, 1)$ such that    $(\kappa+\alpha)/2= 1$  we also infer  from \eqref{additional-info} that
 %$\int_0^t \int_\Omega  |\nabla  \teta |^2   \dd x \dd s \leq C$.
  Hence,  in view  of \EEE  estimate \eqref{est1} and
 applying Poincar\'e inequality,
  we gather
 \begin{equation}
\label{crucial-est3.2} \| \teta   \|_{L^{2}  (0,T; H^1(\Omega))}
\leq C .
\end{equation}

With the very same calculations as
in \cite[Sec.\ 3]{RocRos14}
we also obtain
 \begin{equation}\label{estetainterp}
\|\teta\|_{L^q(Q)}\leq C\quad\hbox{with }q=8/3
\quad \hbox{if } d=3, \quad q=3 \quad \hbox{if } d=2\,
\end{equation}
\EEE interpolating between estimate \eqref{crucial-est3.2} and
estimate \eqref{est1} for $\|\teta\|_{L^\infty (0,T; L^1(\Omega))}$ and
using the Gagliardo-Nirenberg inequality \eqref{gn-ineq}.
% that gives
%$\|\teta\|_{L^h(\Omega)} \leq \|\teta\|_{H^1(\Omega)}^{\theta}
%\|\teta\|_{L^1(\Omega)}^{1-\theta} $ with $\theta \in (0,1)$ and
%$h\in [1,\infty]$ related by $\tfrac1h = \theta(\tfrac12{-}\tfrac1d)
%+1-\theta$. Hence, we get
Furthermore,
we observe
\begin{equation}
\label{quoted-later-ohyes}
\int_\Omega |\nabla \teta^{(\kappa -\alpha)/2}|^2 \dd x =
c \int_\Omega \teta^{\kappa-\alpha-2} |\nabla \teta|^2 \dd x \leq
\frac c{{\ul\teta}^{2\alpha}}\int_{\Omega} \teta^{\kappa+\alpha-2}
|\nabla \teta|^2 \dd x \leq C,
\end{equation}
thanks to the positivity property \eqref{teta-pos} and estimate
\eqref{additional-info}. Resorting to a nonlinear version of the Poincar\'e inequality (cf.\
e.g.\ \eqref{poincare-type}), we  then infer
\begin{equation}
\label{necessary-added}
\| \teta^{(\kappa -\alpha)/2} \|_{L^2 (0,T; H^1(\Omega))}, \, \| \teta^{(\kappa +\alpha)/2} \|_{L^2 (0,T; H^1(\Omega))}  \leq C.
\end{equation}
%%%%%%%%  ************************************************************************
%%%%%%%%
%%%%%%%%	THIRD ESTIMATE
%%%%%%%%
%%%%%%%%  ************************************************************************
\paragraph{\bf Third estimate:}
We test \eqref{e:teta} by $1$, integrate in time, and subtract the resulting relation from
the  total energy balance \eqref{calc1}. We thus obtain
%We test \eqref{eqI} by $\uu_t$, \eqref{eqII} by $\chi_t$, integrate over $(0,t)$, with $t\in (0,T]$ and sum the resulting equations up, obtaining
\begin{equation}
\label{mech-energ-est}
\begin{aligned}
	&
	\itt\io |c_t|^2  \dd x \dd s  +\io  \tfrac1p |\nabla c(t)|^p + \phi(c(t))\dd x
	+ \itt\io m(c,z)|\nabla \mu|^2 \dd x \dd s+
	\\
	& \quad
	+\itt\io  |z_t|^2 \dd x \dd s  +\io \tfrac1p |\nabla z(t)|^p + I_{[0,+\infty)}(z(t)) + \sigma(z(t))  \dd x
	\\
	&
	\quad +
	 \frac 12\io |\uu_t(t)|^2\dd x
	+\int_0^t\int_\Omega   a(c,z) \vism \e(\ub_t)\colon \e(\ub_t) \dd x \dd s
	+ \int_\Omega W(c(t),\e(\ub(t)),z(t)) \dd x
	\\
	& = \io  \tfrac1p |\nabla c_0|^p + \phi(c_0) \dd x
	 +\io \tfrac1p |\nabla z_0|^p + I_{[0,+\infty)}(z_0) + \sigma(z_0)  \dd x  +
	\frac12 \io|\vb_0|^2\dd x
	\\
	& \quad + \int_\Omega W(c_0,\e(\ub_0),z_0) \dd x
	+
	 \itt\io\teta \left(\rho \dive \uu_t+c_t+z_t\right)\dd x \dd
	s+\itt \io {\bf f}\cdot \uu_t\dd x \dd s\\
	&\quad+\itt\int_{\partial\Omega}(\sigmab { \bf  n \EEE})\cdot\db_t\dd S\ds.
\end{aligned}
\end{equation}
%where we have used also the integration by parts formula
%\eqref{int-parts}.
Observe that  the \EEE first, second, third, and fourth integral terms on the right-hand side are bounded thanks to conditions
\eqref{h:initial} on $(c_0,z_0,\uu_0,\vv_0)$.
 As in the First estimate we deduce boundedness of the last and the  last but one \EEE  integral terms
on the right-hand side. 
Since $\phi$,  $I_{[0,+\infty)} +\sigma$, and $W$ are bounded from below,
  exploiting 
\eqref{ellipticity},  \eqref{data-a}, and   \eqref{eqn:assbV} \EEE to deduce that
$\int_0^t\int_\Omega  a( c, z) \vism \e(\ub_t)\colon \e(\ub_t) \dd x \dd x  \geq c  \itt\io |\e(\ub_t)|^2 \dd x \dd s$,
and  using \EEE  \eqref{hyp-m} to deduce that $\int_0^t\int_\Omega m(c,z)|\nabla\mu|^2\dxs\geq m_0\int_0^t\int_\Omega|\nabla\mu|^2\dxs$,
we find that
% the terms on the left-hand side
%estimate from above $\itt\io |\e(\ub_t)|^2 \dd x \dd s $, as well as $\itt\io  |c_t|^2 \dd x \dd s$
% and $\itt\io  |\chi_t|^2 \dd x \dd s$.
\begin{equation}
\label{dissip-est}
\begin{aligned}
	&\itt\io(|c_t|^2+|z_t|^2+|\nabla\mu|^2+|\e(\ub_t)|^2)\dd x \dd s \leq C+\itt\io\teta \left(\rho \dive \uu_t+c_t+z_t\right)\dd x \dd s.
\end{aligned}
\end{equation}

Then, we can estimate  the integral term \EEE on the  right-hand side by
\[
	\varrho\itt \io \left( |\e(\ub_t)|^2 + |c_t|^2 + |z_t|^2 \right) \dd x \dd s + C_\varrho \itt \io |\teta|^2 \dd x \dd s,
\]
for a sufficiently small constant $\varrho>0$, in such a way as to absorb the first integral term into the left-hand side
 of \eqref{dissip-est}.
Exploiting \eqref{crucial-est3.2} on $\teta$, we thus conclude, also with the aid of Korn's inequality
 and  condition \eqref{dirichlet-data} on the boundary value $\db$, \EEE % boundary values $\db$ for $\uu$,
\begin{equation}\label{est5}
	\|c_t\|_{L^2(Q)} + \| \nabla \mu \|_{L^2(Q;\R^d)} + \|z_t\|_{L^2(Q)}+\|\uu_t\|_{L^2(0,T; H^1(\Omega;\RR^d))}
	\leq
	C\,.
\end{equation}
Furthermore, taking into account the previously proved bound \eqref{est1}, we also gather
\begin{equation}
 \label{est5-added}
	\| z \|_{L^\infty (0,T;W^{1,p}(\Omega))} 
	+\| \ub \|_{H^1(0,T;H^1(\Omega;\R^d))}\leq C.
\end{equation}
%%%%%%%%  ************************************************************************
%%%%%%%%
%%%%%%%%	FOURTH ESTIMATE
%%%%%%%%
%%%%%%%%  ************************************************************************
\paragraph{\bf  Fourth estimate:}
%\footnote{\bebe  controlla che abbia messo tutte le norme vettoriali per favore? \ebe}
%We use here the crucial assumption that the exponent $p$
%in \eqref{e:c} and \eqref{e:z} fulfills
 %$p>d$.
We test \eqref{e:u} by $-\mathrm{div}(\vism\eps(\uu_t))$ and integrate in  time.
This leads to
\begin{equation}
\label{added-4-clarity}
\begin{aligned}
	& -\int_0^t  \uu_{tt}\,\cdot\, \mathbf{\dive} (\vism\eps(\uu_t)) \dd
	x \dd s  + \int_0^t \int_{\Omega} \dive( a(c,z)\vism\eps(\uu_t))
	\,\cdot\, \mathbf{\dive} (\vism\eps(\uu_t)) \dd x \dd s
	\\
	&
	 =-
	 \int_0^t \int_{\Omega}
	\dive(\pd{\eps}(c,\eps(\ub),z))  \,\cdot\, \mathbf{\dive}
	(\vism\varepsilon(\uu_t)) \dd x \dd s + \rho \int_0^t \int_\Omega
	\nabla \teta \cdot \mathbf{\dive} (\vism\eps(\uu_t)) \dd x \dd s
	\\
	& \quad
	-
	\int_0^t \int_\Omega \mathbf{f}\,\cdot \mathbf{\dive}
	(\vism\eps(\uu_t)) \dd x \dd s\,.
\end{aligned}
\end{equation}

The following  calculations  \EEE  are based on
\cite[Sec.\ 5]{RocRos14}, to which we refer for details.
However in the present case we have to take care of the
non-homogeneous Dirichlet boundary condition for $\ub$. 
The first term  on the left-hand side of \eqref{added-4-clarity} gives
\begin{align}
\label{est-added-0}
\begin{aligned}
	&-\int_0^t \io  \uu_{tt}\cdot \mathbf{\dive} (\vism\eps(\uu_t)) \dd x \dd s\\
	&\qquad=-\itt\int_{\partial\Omega}\uu_{tt}\cdot\big(\VV\e(\uu_t) 
{ \bf  n \EEE}
\big)\dd S\ds
		+\io \frac12 \eps(\uu_t (t)):\vism \eps(\uu_t (t)) \dd x
		-\io \frac12 \eps(\uu_t (0)):\vism \eps(\uu_t (0)) \dd x.
\end{aligned}
\end{align}
On the boundary cylinder $\partial\Omega\times(0,T)$ we find
$\uu_{tt}=\db_{tt}$  a.e. \EEE
(note that not necessarily $\e(\uu_{t})=\e(\db_{t})$ a.e. on $\partial\Omega\times(0,T)$)
which yields by using the trace theorem and Young's inequality ($\delta>0$ will
be chosen later)
\begin{align*}
	\Big|\itt\int_{\partial\Omega}\uu_{tt}\cdot\big(\VV\e(\uu_t)
{ \bf  n \EEE}
\big)\dd S\ds\Big|
	={}&\Big|\itt\int_{\partial\Omega}\db_{tt}\cdot\big(\VV\e(\uu_t)
{ \bf  n \EEE}\big)\dd S\ds\Big|\\
	\leq{}& \delta\|\uu_t\|_{L^2(0,T;H^2(\Omega;\R^d))}^2+C_\delta\|\db_{tt}\|_{L^2(0,T;H^{1}(\Omega;\R^d))}^2.
\end{align*}
The last term on the right-hand side can be estimated by using \eqref{dirichlet-data}.

For the second  term on the left-hand side of \eqref{added-4-clarity}  we find
\begin{equation}
\label{est-to-fill-2}
\begin{aligned}
	&\int_0^t \int_{\Omega} \mathbf{\dive} ( a(c,z)\vism\eps(\uu_t))
		\,\cdot\, \mathbf{\dive} (\vism\eps(\uu_t)) \dd x \dd	s\\
	&=\int_0^t\int_\Omega  a(c,z) \mathbf{\dive} (\vism
		\varepsilon(\uu_t))\,\cdot\,\mathbf{\dive} (\vism\varepsilon(\uu_t))\dd x \dd s
	 	+\int_0^t\int_\Omega (\nabla  a(c,z)\cdot \vism
		\varepsilon(\uu_t))\,\cdot\,\mathbf{\dive} (\vism\varepsilon(\uu_t))\dd x \dd s\\
	&\geq c\int_0^t \| \uu_t \|_{H^2(\Omega;\RR^d)}^2 \dd s-\|\db_t\|_{L^2(0,T;H^2(\Omega;\R^d))}^2+I_1,
\end{aligned}
\end{equation}
where the second inequality follows from \eqref{H2reg}.
 The second term on the right-hand side is bounded due to \eqref{dirichlet-data}. 
We move $I_1$ to the right-hand side of \eqref{added-4-clarity} and estimate
\begin{equation}
\label{est-to-fill-bis}
\begin{aligned}
	|I_1|&=\left| \int_0^t\int_\Omega \left(\EEE\nabla  a(c,z)\vism \tensoret\right)\EEE\,\cdot\,\mathbf{\dive} (\vism\eps(\uu_t)) \right| \dd x \dd s \\
	&\leq C\int_0^t\|\nabla  a(c,z)\|_{L^{d+\zeta}(\Omega;\R^d)}\|\tensoret\|_{L^{d^{\star}-\eta}(\Omega;\RR^{d\times d})}\|{\mathbf{\dive} (\vism\varepsilon(\uu_t))}\|_{L^2(\Omega; \RR^d)}\dd s\\
	&\leq \delta \int_0^t\|\uu_t\|_{H^2(\Omega;\RR^d)}^2 \dd s
		+C_\delta\int_0^t \|\nabla  a(c,z)\|_{L^{d+\zeta}(\Omega;\R^d)}^2  \|\tensoret\|_{L^{d^{\star}-\zeta}(\Omega;\RR^{d\times d})}^2 \dd s \\
	&\leq \delta \int_0^t\|\uu_t\|_{H^2(\Omega;\RR^d)}^2 \dd s
		+C_\delta\varrho^2\int_0^t \big(\|c\|_{W^{1,p}(\Omega)}^2+ \|z\|_{W^{1,p}(\Omega)}^2\big)\|\uu_t\|_{H^2(\Omega;\RR^d)}^2 \dd s\\
		&\quad+C_\delta C_\varrho\int_0^t \big(\|c\|_{W^{1,p}(\Omega)}^2+ \|z\|_{W^{1,p}(\Omega)}^2\big)\|\uu_t\|_{L^2(\Omega;\RR^d)}^2 \dd s.
\end{aligned}
\end{equation}
In the first line, we have chosen $\zeta>0$ fulfilling $p\geq d+\zeta$, and $\eta>0$ such
that $\tfrac1{d+\zeta} + \tfrac{1}{d^\star -\eta} + \tfrac12 \leq 1$, with $d^{\star}$
from  \eqref{dstar},  in order to apply the H\"older inequality. Moreover, we
have exploited \eqref{data-a},
giving $\|\nabla a( c, z)\|_{L^{d+\zeta}(\Omega;\R^d)} \leq C(\|c\|_{W^{1,p}(\Omega)}+\|z\|_{W^{1,p}(\Omega)})$,
as well as \eqref{interp} to estimate
$\|\tensoret\|_{L^{d^{\star}-\zeta}(\Omega;\RR^{d\times d})}$.
Finally, $\delta $ and $\varrho$
are positive constants that we will choose later, accordingly determining
$C_\delta, \, C_\varrho>0$ via the Young inequality. For  the right-hand side of \eqref{added-4-clarity}
we proceed as follows
\begin{equation}
\label{est-to-fill-1}
\begin{aligned}
	& -\int_0^t \int_{\Omega}
		\dive(\pd{\eps}(c,\eps(\ub), z))  \,\cdot\, \mathbf{\dive}
		(\vism\varepsilon(\uu_t)) \dd x \dd s\\ &
	= -\int_0^t\int_\Omega \pd{\eps c}(c,\e(\ub),z) \nabla c   \,\cdot\, \mathbf{\dive}(\vism\varepsilon(\uu_t))\dd x \dd s
		-\int_0^t\int_\Omega \pd{\eps z}(c,\e(\ub),z) \nabla z \,\cdot\, \mathbf{\dive}(\vism\varepsilon(\uu_t))\dd x \dd s\\
		&\quad -\int_0^t\int_\Omega  \big(\pd{\eps \eps}(c,\e(\ub),z) :\hspace*{-0.2em}\cdot\,\nabla (\e(\ub))\big)\,\cdot\,
		\mathbf{\dive}(\vism\varepsilon(\uu_t))\dd x \dd s	\\
	&\leq C_4 \int_0^t\int_\Omega\big(|\nabla c|+|\nabla z|\big)|\big(|\e(\ub)|+1\big)|\dive(\VV\e(\ub_t))|\dxs
		+C_4\int_0^t\int_\Omega|\nabla(\e(\ub))||\dive(\VV\e(\ub_t))|\dxs\\
	&\leq C_4 \int_0^t \left( \|\nabla c\|_{L^{d+\zeta}(\Omega;\R^d)}+\|\nabla z\|_{L^{d+\zeta}(\Omega;\R^d)}\right)
		\big(\|\tensore\|_{L^{d^{\star}-\eta} (\Omega;\RR^{d\times d})}+1\big)\|{\mathbf{\dive} (\vism\varepsilon(\uu_t))}\|_{L^2(\Omega;\RR^d)}\dd s\\
		&\quad+ C_4'\int_0^t\|\ub\|_{H^2(\Omega;\RR^d)}\|\ub_t\|_{H^2(\Omega;\RR^d)} \dd s\\
	&\leq  \sigma\int_0^t\|\ub_t\|_{H^2(\Omega;\RR^d)}^2 \dd s
	  +C_\sigma \int_0^t  \left(
	\left( \|c\|_{W^{1,p}(\Omega)}^2 + \|z\|_{W^{1,p}(\Omega)}^2+1\right)\left( \|\ub\|_{H^2(\Omega; \RR^d)}^2+1 \right) \|\ub\|_{H^2(\Omega;\RR^d)}^2\right)  \dd s\,.
\end{aligned}
\end{equation}
Here, the positive constants $\zeta$ and $\eta$ again fulfill
$p\geq d+\zeta$ and $\tfrac1{d+\zeta} + \tfrac{1}{d^\star -\eta} + \tfrac12 \leq 1$, and
we have exploited inequality \eqref{interp}
with a constant $\sigma$ that we will choose later, and some $C_\sigma>0$.
Moreover, we have used  the structural assumption \eqref{eqn:assumptionW} on $W$  (cf.\ also \eqref{later-ref}), \EEE  and 
estimates \eqref{est1} and \eqref{est5-added}, yielding
$\| c \|_{L^\infty(Q)} + \| z \|_{L^\infty(Q)}  \leq C$, whence
$$
	\|b(c,z) \|_{L^{\infty}(Q)}+\|b_{,c}(c,z) \|_{L^{\infty}(Q)}+\|b_{,z}(c,z) \|_{L^{\infty}(Q)}
	+ \|\eps^*(c)\|_{L^{\infty}(Q)}+\|(\eps^*)'(c)\|_{L^{\infty}(Q)}\leq C.
$$

Finally, we estimate
\begin{equation}
\label{est-to-fill-3} \left| \rho \int_0^t \int_\Omega \nabla \teta
	\cdot \mathbf{\dive} (\vism\eps(\uu_t)) \dd x \dd s \right|
	\leq \eta \int_0^t\|\uu_t\|_{H^2(\Omega;\RR^d)}^2 \dd s+C_\eta\int_0^t
	\|\nabla\teta\|_{L^2(\Omega;\R^d)}^2 \dd s
\end{equation}
for some positive constant $\eta$   to be fixed  later and for some $C_\eta>0$.
Combining estimates
\eqref{est-added-0}--\eqref{est-to-fill-3} with \eqref{added-4-clarity}
taking into account the previously proved  estimates \eqref{est1},
\eqref{crucial-est3.2},  and exploiting \eqref{bulk-force}  on
$\mathbf{f}$ to estimate the last term on the right-hand side of
\eqref{added-4-clarity}, we obtain
\begin{align*}
	&\frac{\nu_0\EEE}{2} \io |\eps(\uu_t (t) ) |^2  \dd x
		+c\int_0^t \| \uu_t \|_{H^2(\Omega;\RR^d)}^2 \dd s\\
	&\qquad\leq C\io|\eps( \vv^0 \EEE)|^2 \dd x + C\| \mathbf{f}\|_{L^2 (0,T;\Ha)}^2
		+C\| \mathbf{d}\|_{H^1(0,T;H^2(\Omega;\R^d))\cap H^2(0,T;H^1(\Omega;\R^d))}^2\\
		&\qquad\quad+\frac{c}2 \int_0^t \| \uu_t \|_{H^2(\Omega;\RR^d)}^2 \dd s
		+C\left(1+\|\uu^0\EEE\|_{H^2(\Omega;\RR^d)}^2+\int_0^t\int_0^s\|\uu_t\|_{H^2(\Omega;\RR^d)}^2\, \dd r \dd s\right)\,,
\end{align*}
with $\beta_0$ from \eqref{ellipticity}  (cf.~also \eqref{eqn:assbV}), \EEE
where we have used the fact that
$\int_0^t\|\uu\|_{H^2(\Omega;\RR^d)}^2 \dd s \leq \|\uu_0\|_{H^2(\Omega;\RR^d)}^2+\itt\int_0^s \|\uu_t\|_{H^2(\Omega;\RR^d)}^2 \dd r \dd s $
and chosen $\delta$, $\varrho$,  $\sigma$,  and $\eta$ sufficiently small.
Therefore, using the standard Gronwall lemma and conditions \eqref{data_u}--\eqref{data_v} 
on the initial data  $\uu^0$ and $\vv^0$, \EEE we conclude
\begin{equation}
\label{palla}
	\| \uu_t \|_{ L^{2}(0,T;  H^2(\Omega;\R^d))\cap L^{\infty}(0, T;  H^1(\Omega;\R^d)) } \leq C \quad \text{whence} \quad  \| \uu \|_{ L^{\infty}(0,T;  H^2(\Omega;\R^d))}
	\leq C.
\end{equation}
%for some positive constant $\ell_2$.

 By comparison in \eqref{e:u} %  taking into account the regularity property \eqref{reg-pavel-b},
we also get
\begin{equation}
	\| \uu_{tt} \|_{L^{2}(0,T; \Ha)}
	\leq C.
	\label{utt-comparison}
\end{equation}

In the end, taking into account the form \eqref{eqn:assumptionW} of $W$, 
%on $\pd{c}(c,\e(\ub),z)$ and $\pd{z}(c,\e(\ub),z)$
we infer from \eqref{est5-added} and \eqref{palla}, taking into account the
continuous embedding $H^2(\Omega;\R^d)\subset W^{1,d^\star}(\Omega;\R^d)$, that
\begin{equation}
	\label{est-for-Ws}
	\| \pd{c}(c,\e(\ub),z) \|_{L^\infty (0,T; L^2(\Omega))}
	+\| \pd{z}(c,\e(\ub),z) \|_{L^\infty (0,T; L^2(\Omega))}\leq C.
\end{equation}
%%%%%%%%  ************************************************************************
%%%%%%%%
%%%%%%%%	FIFTH ESTIMATE
%%%%%%%%
%%%%%%%%  ************************************************************************
\paragraph{\bf Fifth estimate:} %In this
Recall that, for the time being we suppose $\widehat\beta \in
\rmC^1(\R)$, and we will use the notation $\phi'= \beta +\gamma'$.
It follows from \eqref{e:c} and the no-flux boundary conditions
on $c$ that  $\dashint_\Omega c_t \dd x = 0$ a.e.\ in $(0,T)$, hence
there exists $\mathfrak{m}_0 \in \R$ with
\begin{equation}
\label{constant-mean}
	\dashint_\Omega  c(t) \dd x  =\mathfrak{m}_0 \quad \text{for all } t \in [0,T].
\end{equation}
Now, from \eqref{e:mu} we deduce that
\begin{equation}
\label{mean-mu}
	\dashint_\Omega \mu \dd x =   \dashint_\Omega \phi'(c) \dd x +  \dashint_\Omega W_{,c}(c,\e(\ub),z) \dd x -  \dashint_\Omega \vartheta \dd x \quad \aein \, (0,T).
\end{equation}
Thanks to  estimates \eqref{crucial-est3.2}  and   \eqref{est-for-Ws}, we have that
\begin{equation}
	\label{used-here}
	\|  \textstyle{\dashint_\Omega \vartheta \dd x} \|_{L^2(0,T)} +
	\left\|   \dashint_\Omega W_{,c}(c,\e(\ub),z) \dd x \right\|_{L^\infty (0,T)} \leq C.
\end{equation}
Therefore, in order to estimate $\dashint_\Omega \mu \dd x$ it is  sufficient to gain a
bound for $\dashint_\Omega \phi'(c) \dd x$. We shall do so by testing \eqref{e:mu} by
$c- \dashint_\Omega c \dd x  = c - \mathfrak{m}_0 $.
This gives for  a.a.\  $t\in (0,T)$
\begin{equation}
\label{clever-c}
\begin{aligned}
	&\int_\Omega |\nabla c(t)|^p \dd x + \int_\Omega  \beta(c (t)) (c(t)-\mathfrak{m}_0) +  \gamma'(c (t)) (c(t)-\mathfrak{m}_0)   \dd x\\
	&=\int_\Omega  \big(\teta(t) - W_{,c}(c(t),\e(\ub(t)),z(t))\big) (c(t)-\mathfrak{m}_0) \dd x
		+\int_\Omega  \left(\mu(t) -\dashint_\Omega \mu(t)  \dd x \right) ( c(t)- \mathfrak{m}_0 )\dd x\\
		&\quad-\io c_t(t)c(t)\dx\\
	&\leq C \left( \|\teta(t)\|_{L^2(\Omega)}+ \| W_{,c}(c(t),\e(\ub(t)),z(t)) \|_{L^2(\Omega )} \right)   \| c(t)\|_{L^2(\Omega)} + \| \nabla \mu(t)\|_{L^2(\Omega)}  \| \nabla c(t)\|_{L^2(\Omega)}\\
		&\quad+\|c_t(t)\|_{L^1(\Omega)}\|c(t)\|_{L^\infty(\Omega)}
\end{aligned}
\end{equation}
where for the first equality we have used that $ (\dashint_\Omega \mu(t)  \dd x )( \int_\Omega  (c(t)- \mathfrak{m}_0 )\dd x ) =0$
 and $\mathfrak{m}_0\io c_t(t)\dx=0$, and for the second one
the Poincar\'e inequality for the second integral.
Now, observe that
\begin{equation}
\label{added-4-gamma}
  \int_\Omega \gamma'(c (t)) (c(t)-\mathfrak{m}_0) \dd x \geq -C
\end{equation}
since, by the $L^\infty (0,T;W^{1,p}(\Omega))$-estimate for
$c$ and the fact that $p>d$, we have
\begin{equation}
\label{gamma'-bounded}
	\| \gamma'(c) \|_{L^\infty (Q)} \leq C\,.
\end{equation}
%with $\gamma' $ the antimonotone and linear contribution to $\phi'$, cf.\ \eqref{decomposition}.
   Combining \eqref{clever-c} and \eqref{added-4-gamma}  with   \eqref{mir-zelik}, yielding
  %  there exist constants $$
 %such that
\begin{equation}
\label{danke_zelik}
	\exists\, C_{\mathfrak{m}_0},\,C_{\mathfrak{m}_0}'>0  \ \ \text{for a.a. }
          t \in (0,T)\, : \quad
	\int_\Omega | \beta(c(t))| \dd x \leq C_{\mathfrak{m}_0}  \int_\Omega \beta(c (t)) (c(t)-\mathfrak{m}_0) \dd x  + C_{\mathfrak{m}_0}',
\end{equation}
and taking into account estimates \eqref{crucial-est3.2}, \eqref{est5}, \eqref{est5-added},    and  \eqref{est-for-Ws}, we conclude that
$\left\|\beta(c) \right\|_{L^2 (0,T; L^1(\Omega))} \leq C$,
whence $\left\|\phi'(c) \right\|_{L^2 (0,T; L^1(\Omega))} \leq C.$ 
Then, arguing by comparison in \eqref{mean-mu} and taking into account \eqref{used-here} we ultimately conclude
$\| \dashint_\Omega \mu\dd x \|_{L^2(0,T)} \leq C$.
Combining this with \eqref{est5-added} and using the Poincar\'e inequality we infer that
\begin{equation}
\label{est-for-mu}
\| \mu\|_{L^2(0,T;H^1(\Omega))} \leq C.
\end{equation}
%%%%%%%%  ************************************************************************
%%%%%%%%
%%%%%%%%	SIXTH ESTIMATE
%%%%%%%%
%%%%%%%%  ************************************************************************
\paragraph{\bf  Sixth estimate:}
We now argue by comparison in  \eqref{e:mu}
and take into account estimates  \eqref{crucial-est3.2},  \eqref{est5}, \EEE \eqref{est-for-Ws}, and \eqref{est-for-mu}, as well as
\eqref{gamma'-bounded}. 
 Then we conclude that
\[
\| \Delta_p(c) +\eta  \|_{L^2(0,T; L^2(\Omega))} \leq C.
\]
Now, in view of  the monotonicity of the function $\beta$,   it is not
difficult to deduce
from the above estimate  that 
\begin{equation}
	\| \Delta_p(c)\|_{L^2(0,T; L^2(\Omega))}+ \| \eta  \|_{L^2(0,T; L^2(\Omega))} \leq C.
\end{equation}
%%%%%%%%  ************************************************************************
%%%%%%%%
%%%%%%%%	SEVENTH ESTIMATE
%%%%%%%%
%%%%%%%%  ************************************************************************
\paragraph{\bf Seventh  estimate:}
We test  \eqref{e:teta} by
$\frac{w}{\teta}$, with $w$ a test function in
$ W^{1,d+\epsilon}(\Omega)$ with $\epsilon>0 $, which then ensures $w\in  L^\infty(\Omega)$.
% (in particular, this is true for $w \in$).
%Observe that $
%\subset
%L^\infty (\Omega)$.
Thus, using  the place-holders
\begin{align*}
	&H :=   - c_t - z_t -\rho\mathrm{div}(\uu_t),\\
	&J:= \frac1\teta (g+  a(c,z) \eps(\uu_t):\vism \eps(\uu_t) + |c_t|^2 + |z_t|^2 + m(c,z)|\nabla \mu|^2),
\end{align*}
we obtain that
\[
\begin{aligned}
&
\left|  \int_\Omega \partial_t \log(\teta) w \dd x \right|
\\
&
=  \left|   \int_\Omega \left( H w -  \frac{\mathsf{K}(\teta)}\teta \nabla \teta
\cdot \nabla w   - \frac{\mathsf{K}(\teta)}{\teta^2}
|\nabla\teta|^2 w + Jw \right)   \dd x
 +\int_{\partial\Omega} h \frac{w}{\teta}    \dd S
\right|
\\
&
 \leq \left|\int_\Omega H w \dd x \right|
+ \left| \int_\Omega \frac{\mathsf{K}(\teta)}\teta \nabla \teta
\cdot \nabla w \dd x  \right| +
 \left| \int_\Omega  \frac{\mathsf{K}(\teta)}{\teta^2}
|\nabla\teta|^2 w   \dd x  \right| + \left|\int_\Omega J w \dd x \right|
+ \left|\int_{\partial\Omega} h\frac{w}\teta  \dd S\right|
\\
&
 \doteq I_1+I_2+I_3+I_4 +I_5.
\end{aligned}
\]
From estimate \eqref{est5} we deduce  that $\|H\|_{L^2(0,T; L^2(\Omega))} \leq C$,
therefore
%\begin{equation}
%\label{to-quote-1}
\[
|I_1| \leq \mathcal{H}(t) \|w \|_{L^2(\Omega)} \quad \text{ with }
\mathcal{H}(t)= \| H(\cdot,t)\|_{L^2(\Omega)} \in L^2(0,T).
\]
Analogously, also in view of \eqref{heat-source}, of
\eqref{teta-pos}
and of estimate \eqref{est5},  we infer
  %(provided  that $\bar{q}$ is such that
%$W^{1,\bar q}(\Omega) \subset L^\infty(\Omega)$)
 that
%\begin{equation}
%\label{q-big1}
\[
 |I_4| \leq \frac{1}{\ul\teta}\mathcal{J}(t)
\|w\|_{L^\infty(\Omega)} \qquad \text{with } \mathcal{J}(t) := \|
J(\cdot,t)\|_{L^1(\Omega)} \in L^1(0,T).
%\end{equation}
\]
Moreover, $
|I_5| \leq  \frac{1}{\ul\teta}  \|h(t)\|_{L^2(\partial \Omega)} \| w\|_{L^2(\partial \Omega)} $, with $  \|h(t)\|_{L^2(\partial \Omega)}    \in L^1(0,T)$ thanks to \eqref{dato-h}.
In order to estimate $I_2$ and $I_3$ we develop the very same calculations as in the proof of
\cite[Sec.\ 3, \emph{Sixth estimate}]{RocRos14}. Referring to the latter paper for all details,  we mention here that,
exploiting the growth condition \eqref{hyp-K} on $\condu$, the positivity  of $\teta$ \eqref{teta-pos}, and the H\"older inequality, we have
%\begin{equation}
%\label{I2-est}
\[
\begin{aligned}
&
|I_2| \leq C  \frac C{\ul\teta} \mathcal{O}(t)  \| \nabla
w\|_{L^2(\Omega;\RR^d)}
+  C \widetilde{\mathcal{O}}(t) \| \nabla w\|_{L^{{d+\epsilon}}(\Omega;\RR^d)}
\\
&
\quad \text{with }
\begin{cases}
\mathcal{O}(t) := \| \nabla
\teta(t) \|_{L^2(\Omega;\RR^d)} \in L^2(0,T) & \text{by \eqref{crucial-est3.2},}
\\
 \widetilde{ \mathcal{O}}(t) \EEE:=  \| \teta(t)^{(\kappa
+\alpha-2)/2} \nabla \teta (t) \|_{L^2(\Omega;\RR^d)}
\|\teta(t)^{(\kappa -\alpha)/2} \|_{L^{d^\star-\eta}(\Omega)}  \in
L^1(0,T) & \text {by  \eqref{additional-info},
\eqref{crucial-est3.2}, \eqref{necessary-added},}
\end{cases}
\end{aligned}
\]
with $\tfrac1{d+\epsilon} + \tfrac{1}{d^\star -\eta} +\tfrac12 \leq 1 $.
With analogous arguments, we find
\[
\begin{aligned}
&
 |I_3| \leq  \frac C{{\ul \teta }^2}   \mathcal{O}(t)^2 \| w\|_{L^\infty (\Omega)}
 + C   \overline{\mathcal{O}}(t) \|w \|_{L^\infty(\Omega)}
 \\
 &
 \text{with }  \overline{\mathcal{O}}(t)=  \int_\Omega \teta(t)^{\kappa+\alpha-2} |\nabla \teta(t)|^2 \dd x
 + \int_\Omega  |\nabla \teta(t)|^2 \dd x \in L^1(0,T)\  \text{ by  \eqref{additional-info} and  \eqref{crucial-est3.2}.}
 \end{aligned}
\]
All in all, we  infer that there exists a positive function  $\mathcal{C} \in L^1(0,T)$ such that
$
|  \int_\Omega \partial_t \log(\teta(t)) w \dd x | \leq \mathcal{C} (t)
\|w\|_{W^{1,d+\epsilon}(\Omega)} $  for 
 a.a.\ \EEE $ t \in (0,T).
$
Hence,
\begin{equation}\label{est6}
\|\partial_t\log(\teta)\|_{L^1(0,T; (W^{1,d+\epsilon}(\Omega)'))} \leq C.
% W^{1,d+\epsilon} (\Omega)^* )} \leq C
%W^{1,d+\epsilon}(\Omega))^*)}\leq C.
\end{equation}
%%%%%%%%  ************************************************************************
%%%%%%%%
%%%%%%%%	EIGHTH ESTIMATE
%%%%%%%%
%%%%%%%%  ************************************************************************
\paragraph{  \bf  Eighth  estimate [$ {\boldsymbol \kappa}  {\boldsymbol \in} 
{ \bf  (1,5/3)}$ if ${\bf d=3}$ and $\boldsymbol{\kappa} {\boldsymbol \in}
{\bf (1,2)}$ if ${ \bf d=2}$]: \EEE}
We multiply
 \eqref{e:teta} by a test function $w \in  W^{1,\infty}(\Omega)$ (which e.g.\ holds if $w \in W^{2,d+\epsilon}(\Omega)$ for
 $\epsilon>0$) and find
 % We will then use that $W^{2,d+\epsilon}(\Omega) \subset
 %\footnote{\beric qui non}
 \[
\begin{aligned}
\left|  \int_\Omega \teta_t w \dd x \right|
 \leq \left|\int_\Omega L w \dd x \right|
+ \left| \int_\Omega \mathsf{K}(\teta) \nabla \teta
\cdot \nabla w \dd x  \right| + \left|\int_{\partial\Omega} hw \dd S\right| \doteq I_1+I_2+I_3,
\end{aligned}
\]
where we have set
$$
	L= -c_t\teta - z_t \teta -\rho\teta \mathrm{div}(\uu_t)+g+ a(c,z)\eps(\uu_t):\vism \eps(\uu_t) +|c_t|^2 + |z_t|^2 + m(c,z)|\nabla \mu|^2.
$$
Therefore,
\[
|I_1| \leq \mathcal{L}(t) \|w\|_{L^\infty (\Omega)} \quad \text{with } \mathcal{L}(t):=\|L(t)\|_{L^1(\Omega)} \in L^1(0,T), \quad
|I_3| \leq \| h(t) \|_{L^2(\partial
\Omega)} \| w\|_{L^2(\partial
\Omega)}  \text{ with } h\in L^1(0,T)
\]
thanks to \eqref{heat-source},   \eqref{crucial-est3.2}, and   \eqref{est5} for $I_1$,  and \eqref{dato-h} for $I_3$.
We estimate  $I_2$
by proceeding exactly in the same way as for \cite[Sec.\ 3, \emph{Seventh estimate}]{RocRos14}. Namely, taking into account once again the growth condition
 \eqref{hyp-K} on $\condu$, we find
 \begin{equation}
\label{citata-dopo-ehsi}
|I_2|\leq   C\| \teta^{(\kappa-\alpha+2)/2} \|_{L^2(\Omega)} \|\teta^{(\kappa+\alpha-2)/2} \nabla \teta\|_{L^2(\Omega;\RR^d)} \|\nabla w\|_{L^\infty (\Omega;\RR^d)}
+   C \| \nabla \teta\|_{L^2(\Omega;\RR^d)}  \|\nabla w\|_{L^2 (\Omega;\RR^d)}.
 \end{equation}
 Observe that, since $\kappa <\frac53$ if $d=3$, and $\kappa <2$ if $d=2$, and  $\alpha$ can be chosen arbitrarily close to $1$, from estimate
\eqref{estetainterp}
we have that $\teta^{(\kappa-\alpha+2)/2}$ is bounded in $L^2(0,T; L^2(\Omega))$.
 Thus, also taking into account \eqref{crucial-est3.2}, we conclude that
  $|I_2|\leq C  \mathcal{L}^*(t) \|\nabla w \|_{L^\infty (\Omega)} $ for some
$\mathcal{L}^* \in L^1(0,T)$.
Hence,
\begin{equation}
\label{bv-esti-temp}
\|\teta_t\|_{L^1(0,T; W^{1,\infty}(\Omega)')} \leq C.
%W^{2,d+\epsilon}(\Omega)^*)} \leq C.
 %(W^{2,d+\epsilon}(\Omega))^*)}\leq C \qquad \text{for every }\epsilon >0.
%\quad\hbox{where $\sigma>1$ depends on $d$} \,.
\end{equation}
%%%%%%%%  ************************************************************************
%%%%%%%%
%%%%%%%%	NINTH ESTIMATE
%%%%%%%%
%%%%%%%%  ************************************************************************
\paragraph{\bf  Ninth estimate:}
We test \eqref{e:c} by $\Delta\mu$ and integrate in time. It follows
\begin{align}
\label{eqn:est9}
	\itt\io\dive\big(m(c,z)\nabla\mu\big)\Delta\mu\dxs
	=\itt\io c_t\Delta\mu\dxs. \EEE
\end{align}
The left-hand side  is estimated \EEE  below by exploiting Hypotheses (II)
and the boundedness $\|c\|_{L^\infty(Q)}+\|z\|_{L^\infty(Q)}\leq C$, viz. 
\begin{align*}
	\itt\io\dive\big(m(c,z)\nabla\mu\big)\Delta\mu\dxs
	&\geq\itt\io\big(\nabla m(c,z)\cdot\nabla\mu\big)\Delta\mu\dxs
		+m_0\itt\io|\Delta\mu|^2\dxs\\
	&\geq-C\itt\io(|\nabla c|+|\nabla z|)|\nabla\mu||\Delta\mu|\dxs
		+m_0\itt\io|\Delta\mu|^2\dxs.
\end{align*}
By using the interpolation inequality \eqref{interp2}
and by using analogous calculations as in the \textit{Fourth estimate}, we find
by Young's inequality
\begin{align*}
	&\itt\io(|\nabla c|+|\nabla z|)|\nabla\mu||\Delta\mu|\dxs\\
	&\qquad\leq C\itt\big(\|\nabla c\|_{L^{d+\zeta}(\Omega;\R^d)}+\|\nabla z\|_{L^{d+\zeta}(\Omega;\R^d)}\big)
		\|\nabla\mu\|_{L^{d^*-\eta}(\Omega;\R^d)}\|\Delta\mu\|_{L^2(\Omega)}\ds\\
	&\qquad\leq C\big(\|\nabla c\|_{L^\infty(0,T;L^{p}(\Omega;\R^d))}+\|\nabla z\|_{L^\infty(0,T;L^{p}(\Omega;\R^d))}\big)\itt\|\nabla\mu\|_{L^{d^*-\eta}(\Omega;\R^d)}\|\Delta\mu\|_{L^2(\Omega)}\ds\\
	&\qquad\leq C'\itt\big(\varrho\|\nabla\mu\|_{H^{1}(\Omega;\R^d)}+C_\varrho\|\nabla\mu\|_{L^2(\Omega;\R^d)}\big)\|\Delta\mu\|_{L^2(\Omega;\R^d)}\ds\\
	&\qquad\leq \varrho C'C_\delta\itt\|\nabla\mu\|_{H^{1}(\Omega;\R^d)}^2\ds
		+C' C_\varrho C_\delta\itt\|\nabla\mu\|_{L^{2}(\Omega;\R^d)}^2\ds
		+\delta C'\itt\|\Delta\mu\|_{L^2(\Omega)}^2\ds.
\end{align*}
By choosing suitable $\delta>0$ and $\varrho>0$, we see that
\begin{align*}
	&\itt\io\big(|\nabla c|+|\nabla z|\big)|\nabla\mu||\Delta\mu|\dxs
		\leq \epsilon\itt\|\mu\|_{H^{2}(\Omega)}^2\ds
		+C_\epsilon\itt\|\nabla\mu\|_{L^{2}(\Omega;\R^d)}^2\ds.
\end{align*}
All in all, we find from the above estimates
\begin{align*}
	\itt\|\Delta\mu\|_{L^2(\Omega)}^2 \ds \EEE
	\leq
        \epsilon\itt\|\mu\|_{H^2(\Omega)}^2\ds+C_\epsilon\itt\|c_t\|_{L^2(\Omega)}^2
         \ds \EEE
		+C_\epsilon\itt\|\nabla\mu\|_{L^{2}(\Omega;\R^d)}^2\ds,
\end{align*}
where the second and the third term on the right-hand side are bounded by \eqref{est5}
for fixed $\epsilon>0$.
By the $H^2$-elliptic regularity estimate for homogeneous Neumann problems, i.e.
\begin{align*}
	\|\mu\|_{H^2(\Omega)}^2\leq C\big(\|\Delta \mu\|_{L^2(\Omega)}^2+\|\mu\|_{H^1(\Omega)}^2\big),
\end{align*}
we conclude by choosing $\epsilon>0$ sufficiently small and by using the boundedness
of $\|\mu\|_{L^2(0,T;H^1(\Omega))}$ in \eqref{est5} that
\begin{align}
	\|\mu\|_{L^2(0,T;H^2(\Omega))}\leq C.
\end{align}

\QED

%%%%%%%%  ************************************************************************
%%%%%%%%
%%%%%%%%	TIME DISCRETIZATION
%%%%%%%%
%%%%%%%%  ************************************************************************
\section{\bf Time discretization  and regularizations}
\label{s:5} In this section we will introduce and motivate a
\textit{thermodynamically consistent time-discretization scheme} for system \eqref{eqn:PDEsystem} and
devote a large part of Sec. \ref{ss:5.2} to the proof  that it
admits solutions. Next, in Sec.\ \ref{ss:5.3} we will derive the
energy and entropy inequalities fulfilled by the discrete solutions,
and, starting from  them, we will obtain a series of a priori
estimates on the approximate solutions.
%%%%
\subsection{Setup of the time-discrete system}
\label{ss:5.1}
We consider an equidistant partition of $[0,T]$, with time-step $\tau>0$ and
nodes
\begin{align}
	t_\tau^k:=k\tau,
\label{time-nodes}
\end{align}
$k=0,\ldots,K_\tau$, and we  approximate
the data $\mathbf{f}$, $g$, and $h$
by local means, i.e.
setting for all $k=1,\ldots,K_{\tau}$
\begin{equation}
\label{local-means} \ftau{k}:=
  \frac{1}{\tau}\int_{t_\tau^{k-1}}^{t_\tau^k} \mathbf{f}(s)\dd s\,,
  \qquad  \gtau{k}:= \frac{1}{\tau}\int_{t_\tau^{k-1}}^{t_\tau^k} g(s)
  \dd s\,, \qquad \htau{k}:=
  \frac{1}{\tau}\int_{t_\tau^{k-1}}^{t_\tau^k} h(s) \dd s\,.
\end{equation}
In what follows, for a given $K_\tau$-tuple $(v_{\tau}^k)_{k=1}^{K_\tau}$
the time-discrete derivative is denoted by
\[
  \Dt(v) =\frac{v_{\tau}^k-v_{\tau}^{k-1}}{\tau} \quad \text{so that} \quad
  \Dt(\Dt(v)) = \frac{v_{\tau}^k-2v_{\tau}^{k-1} + v_{\tau}^{k-2}}{\tau^2}.
\]
Before stating the complete time-discrete scheme in Problem \ref{def:time-discrete}, 
we are going to introduce its main ingredients in   what follows.  \EEE

\paragraph{\bf Regularization of the coefficient functions depending on 
  $\bf c$ \EEE}
In the following we will analyze a specially chosen
time-discretization scheme for system \eqref{eqn:PDEsystem}.
To ensure suitable coercivity properties in the time-discrete system
needed for existence of solutions
we utilize the following $\omega$-regularizations which will eventually vanish as
$\omega\downarrow 0$:
\begin{itemize}
	\item[--]
		
    First of all, we will replace the maximally monotone operator $\beta$
       (the derivative of the convex part of the potential $\phi$ (see Hypothesis (I)) 
    by its Yosida regularization
    $\beta_\omega\in \rmC^0(\R)$
			with Yosida index $\omega\in(0,\infty)$.
    This will be crucial to render rigorously the \emph{Fifth a priori estimate}
    on the time-discrete level, cf.
    the calculations in Sec.\ \ref{ss:5.4}.  Observe that the Yosida approximation 
    $\widehat{\beta}_\omega\in C^1(\R)$ of $\widehat\beta$, fulfilling
    $\widehat{\beta}_\omega' = \beta_\omega$, is still convex,  and that $\beta_\omega(0)=0$. \EEE 
     For notational consistency we set $\phi_\omega:=\widehat{\beta}_\omega+\gamma$.
  \item[--]
  	Let $\{\C R_\omega\}_{ \omega>0\EEE}\subseteq C^2(\R)\cap W^{2,\infty}(\R)$ be a family of
  	functions (we can think of ``smoothed truncations'') such that:
  	\begin{align}
  		\forall M>0\quad\exists \omega_0>0\qquad\forall \omega\in(0,\omega_0),\;c\in(-M,M):
  			\qquad \C R_\omega(c)=c.
  	\label{Rtrunc}
  	\end{align}
	 They have the role to somehow provide for the  information that $c$ is bounded, which is not supplied by the  concentration potential $\phi$, defined on all
	of $\R$. In turn, this information is crucial in order to make some of the following calculations rigorous.
	The limit passage as $\omega \down 0$ will be possible thanks to an a priori bound for $c$ in $L^\infty (Q)$, cf.\ Sec.\ \ref{s:6} ahead. \EEE
	\par
  	We define the following regularizations for the elastic energy density:
  	\begin{align*}
  		&W^\omega(c,\e,z):=W(\C R_\omega(c),\e,z).
  	\end{align*}
  	and observe that for fixed $\omega>0$ and fixed $\e\in\R_{sym}^{d\times d}$ and $z\in\R$  (cf.\ also \eqref{later-ref}): \EEE
  	\begin{align}
  	\label{eqn:WtauEst}
	  	|W^\omega(c,\e,z)|+|W_{,c}^\omega(c,\e,z)|+|W_{,cc}^\omega(c,\e,z)|\leq C
	  	\qquad\text{ uniformly in }c\in\R.
  	\end{align}
\end{itemize}
Throughout \underline{this section} we neglect the subscript $\omega$ on
the solutions $c$, $\mu$, $z$, $\teta$ and $\uu$ for the sake of readability.

\paragraph{\bf  Convex-concave \EEE splitting of the coefficient functions}
Let us mention in advance how all the various nonlinear terms in
\eqref{eqn:PDEsystem} will be coped with in the discrete system
\eqref{PDE-discrete}, which is in fact carefully designed in such a
way as to ensure the validity of the \emph{discrete total energy
inequality}, cf.\ the forthcoming Lemma \ref{l:energy-est}.
To this aim, it will be crucial to employ the   \emph{convex-concave
splitting} of the functions $c\mapsto  W^\omega(c,\e,z)$,
$z \mapsto  W^\omega(c,\e,z)$,
$z\mapsto\sigma(z)$, as well as the specific splitting
\eqref{specific-splitting-phi} below (cf.\ also \eqref{decomposition}) for $\phi_\omega$.
Recall that, a convex-concave decomposition of some real-valued
$\mathrm{C}^2(I)$-function $\psi$
 with bounded second derivative on an interval $I$ 
may be canonically  given by $\psi = \conv{\psi} + \conc{\psi}$, with
\begin{align}
\label{eqn:splitting}
    &\conv{\psi}(x):=\psi(x)+\frac12\Big(\max_{y\in I}|\psi''(y)|\Big)x^2,
    &&\conc{\psi}(x):=-\frac12\Big(\max_{y\in I}|\psi''(y)|\Big)x^2.
\end{align}
Therefore, we will  proceed as follows:
\begin{itemize}
  \item[--]
    The nonlinear contribution $\sigma'(z)$ in \eqref{e:z} will be discretized via the
    convex-concave splitting \eqref{eqn:splitting}
     on $I=[0,1]$:
    \begin{align*}
      \sigma'(z)\;\text{ via }\;(\conv{\sigma})'(\zk)+(\conc{\sigma})'(\zkk).
    \end{align*}
  \item[--]
    For the time-discrete version of the term $\pd{c}(c,\e(\ub),z)$ in \eqref{e:mu} and
    $\pd{z}(c,\e(\ub),z)$ in \eqref{e:z} we will resort to partial convex-concave splittings
    of $ W^\omega$.
    To denote them, we will use the symbols \eqref{eqn:splitting}, combined with   subscripts \EEE
    to denote the variable with respect to which the splitting is computed.
    Therefore, we set
    \begin{subequations}
    \label{eqn:convConcSplittingWc}
    \begin{align}
      &\convWc(c,\e,z):= W^\omega(c,\e(\ub),z)+\frac12\Big (
	      \sup_{\widetilde c \in \R}| W_{,cc}^\omega(\widetilde c,\e,z)|\Big)c^2,\\
      &\concWc(c,\e,z):=-\frac12\Big(\sup_{\widetilde c \in \R}| W_{,cc}^\omega(\widetilde c,\e,z)|\Big)c^2,\\
      &\convWz(c,\e,z):=W^\omega(c,\e(\ub),z)+\frac12\Big (
	      \sup_{\widetilde z \in [0,1]}| W_{, zz\EEE}^\omega(c,\e,\widetilde z)|\Big)z^2,\\
      &\concWz(c,\e,z):=-\frac12\Big(\sup_{\widetilde z \in [0,1]}|W_{, zz}\EEE^\omega(c,\e,\widetilde z)|\Big)z^2.
    \end{align}
    \end{subequations}
     Note that these functions are well-defined for fixed $\omega>0$ due to \eqref{eqn:WtauEst}. 
%    The splitting functions $\convWz$  and $\concWz$ with respect to $z$ are
%    defined analogously
%     but with $[0,1]$
    The splitting of $ W^\omega$ with respect to $\e(\ub)$ is not needed due to the 
    convexity of $ W^\omega$ with respect to $\e(\ub)$  by  the structural assumption \eqref{eqn:assumptionW}
     and the non-negativity of $b$  in Hypothesis (V).
    We easily see that
    $$
         W^\omega=\convWc+\concWc=\convWz+\concWz
    $$
    and that
    \begin{align*}
        &&&&&\convWc(\cdot,\e,z)\text{ is convex on $\R $},
        &&\concWc(\cdot,\e,z)\text{ is concave on } \R \hspace*{2.2em}\bigg\}\text{ for all fixed }\e,z,\\
        &&&&& W^\omega(c,\cdot,z)\text{ is convex on $\R_{sym}^{n\times n}$}
        &&\hspace*{14.0em}\bigg\}\text{ for all fixed }c,z,\\
        &&&&&\convWz(c,\e,\cdot)\text{ is convex on $[0,1]$},
        &&\hspace*{0.07em}\concWz(c,\e,\cdot)\text{ is concave on }[0,1]\quad\bigg\}\text{ for all fixed }c,\e.
    \end{align*}
    We will
     replace the terms $W_{,c}$, 
    $W_{,\eps}$, and 
    $W_{,z}$ in system \eqref{eqn:PDEsystem}  by  their \EEE time-discretized  and regularized versions: \EEE
     \begin{align*}
        &&&&&\pd{c}(c,\e(\ub),z)&&\text{ via }&&\convWcp(\ck,\e(\ukk),\zkk)+\concWcp(\ckk,\e(\ukk),\zkk),&&&&&&\\
        &&&&&\pd{\eps}(c,\cdot,z)&&\text{ via }&&\pd{\eps}^\omega(\ck,\e(\uk),\zk),\\
        &&&&&\pd{z}(c,\e(\ub),z)&&\text{ via }&&\convWzp(\ck,\e(\ukk),\zk)+\concWzp(\ck,\e(\ukk),\zkk).
    \end{align*}
    By exploiting convexity and concavity estimates this time-discretization scheme leads to the crucial estimate
    \begin{align}
    \begin{aligned}
        &\Big(\convWcp(\ck,\e(\ukk),\zkk)+\concWcp(\ckk,\e(\ukk),\zkk)\Big)(\ck-\ckk)\\
        &+\pd{\eps}^\omega(\ck,\e(\uk),\zk):\e(\uk-\ukk)\\
        &+\Big(\convWzp(\ck,\e(\ukk),\zk)+\concWzp(\ck,\e(\ukk),\zkk)\Big)(\zk-\zkk)\\
        &\qquad\geq  W^\omega(\ck,\e(\ukk),\zkk)- W^\omega(\ckk,\e(\ukk),\zkk)\\
          &\qquad\quad+ W^\omega(\ck,\e(\uk),\zk)- W^\omega(\ck,\e(\ukk),\zk)\\
          &\qquad\quad+ W^\omega(\ck,\e(\ukk),\zk)- W^\omega(\ck,\e(\ukk),\zkk)\\
        &\qquad\geq  W^\omega(\ck,\e(\uk),\zk)- W^\omega(\ckk,\e(\ukk),\zkk),
    \end{aligned}
    \label{eqn:convConcWall}
    \end{align}
    which will be used later in the proof of the discrete total energy inequality.

  \item[--]
    We will discretize the (formally written) term $\phi'(c) = \beta(c) +\gamma'(c)$
    in \eqref{e:mu} in the following way:	
		As mentioned above
    the maximally monotone operator $\beta$
		is replaced
    by its Yosida regularization $\beta_\omega\in \rmC^0(\R)$.
    Hence, in view of  the $\lambda_\gamma$-convexity of
    $\gamma$ (cf. Remark \ref{rmk:l-convex-splitting}),  the functions \EEE
    \begin{equation}
      \label{specific-splitting-phi}
      \conv{\phi}_\omega(c): = \widehat{\beta}_\omega(c) + \lambda_\gamma \frac{c^2}2 \quad \text{and} \quad
      \conc{\phi}(c):=\gamma(c) - \lambda_\gamma \frac{c^2}2
    \end{equation}
    provide a convex-concave decomposition of
    $\phi_\omega:= \widehat{\beta}_\omega + \gamma$.
    Thus, we will approximate
    \begin{align*}
      \phi'(c)\;\text{ via }\;(\conv{\phi}_\omega)'(\ck)+(\conc{\phi})'(\ckk) \qquad
      \text{with } \conv{\phi}_\omega,\, \conc{\phi} \text{ given by \eqref{specific-splitting-phi}.}
    \end{align*} 
\end{itemize}
%%%%%%
%We are now in the position to give the
\paragraph{\bf Statement of the time-discrete problem and existence result}
In the following we are going to describe the time-discrete problem formally.
Later on the precise spaces and a weak notion  of solution \EEE will be fixed.
The time-discrete problem (formally) reads as follows:
\begin{problem}
  \upshape \label{def:time-discrete}
	 Let $\omega>0$ and $\tau>0$ be given.   
  Find functions
  $\{(\ck, \muk,\zk,\tk)\}_{k=0}^{K_\tau}$ and $\{\uk\}_{k=-1}^{K_\tau}$
  which satisfy for all $k\in\{1,\ldots,K_\tau\}$
  the following time-discrete version of \eqref{eqn:PDEsystem}:
  \begin{subequations}
  \label{PDE-discrete}
  \begin{itemize}
    \item[(i)] Cahn-Hilliard system:
	    \begin{align}
	    \label{eqn:discr1}
	      \Dt(c)={}&\dive\big(m(\ckk,\zkk)\nabla\muk\big),\\
        \notag
        \muk={}&-\Delta_p(\ck)+(\conv{\phi}_\omega)'(\ck) +(\conc{\phi})'(\ckk)
	        +\convWcp(\ck,\e(\ukk),\zkk)\\
	        &+\concWcp(\ckk,\e(\ukk),\zkk)-\tk+\Dt(c),
      \label{eqn:discr2}
      \end{align}
    \item[(ii)] damage equation:
      \begin{equation}
      \begin{aligned}
        \label{eqn:discr3}
          &\Dt(z)-\Delta_p(\zk)+\ellk +\zek  + (\conv{\sigma})'(\zk)+ (\conc{\sigma})'(\zkk)
        \\ & =-\convWzp(\ck,\e(\ukk),\zk) - \concWzp(\ck,\e(\ukk),\zkk)+\tk
      \end{aligned}
      \end{equation}
      with 
      \begin{align*}
        &\ellk\in \partial I_{[0,\infty)}\big(\zk\big),\qquad
        \zek\in \partial I_{(-\infty,0]}\big(\Dt(z)\big),
      \end{align*} 
    \item[(iii)] temperature equation:
      \begin{equation}
      \label{eqn:discr4}
      \begin{aligned}
        &\Dt(\vartheta)  - \mathrm{div}(\condu(\tk)\nabla \tk) \EEE +\Dt(c)\tk+\Dt(z)\tk+\rho\tk\dive(\Dt(\ub))\\
        &=\gk+|\Dt(c)|^2+|\Dt(z)|^2+m(\ckk,\zkk)|\nabla\muk|^2
        \\
        & \qquad \qquad \qquad +  a(\ckk,\zkk)\e(\Dt(\ub)):\vism\e(\Dt(\ub)),
      \end{aligned}
      \end{equation}
  	\item[(iv)] balance of forces:
      \begin{align}
        &\Dt(\Dt(\ub)) -\dive\Big( a(\ckk,\zkk)\vism\e(\Dt(\ub))
        +  W_{,\e}^\omega(\ck,\e(\uk),\zk) -\rho\tk\mathds 1\Big)=\fk,
      \label{eqn:discr5}
      \end{align}
  \end{itemize}
  \end{subequations}
    supplemented with the initial data
    \begin{align}
        &\hspace*{9.3em}\left.
        \begin{matrix}
            c_\tau^0=c^0,\qquad&
            z_\tau^0=z^0,\qquad&
            \vartheta_\tau^0=\vartheta^0,\qquad\\
            \ub_\tau^0=\ub^0,\qquad&
            \ub_\tau^{-1}=\ub^0-\tau\mathbf v^0\qquad
        \end{matrix}
        \right\}
        &&\text{a.e. in }\Omega\label{discre-initial-cond}
    \end{align}
    and the boundary data
    \begin{align}
        &\left.
        \begin{matrix}
            \nabla\ck\cdot { \bf  n \EEE}=0,\qquad&
            m(\ckk,\zkk)\nabla\muk\cdot { \bf  n \EEE}=0,\qquad&
            \nabla \zk\cdot { \bf  n \EEE}=0,\qquad\\
            \condu(\tk)\nabla\tk\cdot{ \bf  n \EEE}=\hk,\qquad&
            \uk=\dk
        \end{matrix}
        \right\}
        &&\text{a.e. on }\partial\Omega.\label{discre-boundary-cond}
    \end{align}
\end{problem}
\begin{remark}
\label{remark:discProbl}
  \upshape
   A few comments on Problem \ref{def:time-discrete} are in order: \EEE
  \begin{itemize}
    \item[(i)]
      It will turn out that a solution of the time-discrete problem always satisfies
      the constraints:
      \begin{align}
      \label{discre-constraints}
       % &&&\ck\in(\alpha,\beta),
        &&&\zk\in[0,1],&&\Dt(z)\leq 0,
        &&\vartheta_\tau^k\geq\ul\teta\quad(\text{for some }\ul\teta>0)
        &&\text{a.e. in }\Omega 
      \end{align}
%           where $\ul\teta$ is independent of $\tau$
      as long as the initial data satisfy \eqref{h:initial}.
%           \begin{align}
%           \label{initial-constraints}
%               &&&c^0\in(a,b),
%               &&z^0\in[0,1],
%               &&\teta^0\geq\teta^*\quad(\text{for a }\teta^*>0)
%               &&\text{a.e. in }\Omega.
%           \end{align}
	  \item[(ii)]
      Observe that the scheme is fully implicit and, in particular, the discrete temperature equation
      \eqref{eqn:discr4} is coupled with \eqref{eqn:discr2}, \eqref{eqn:discr3}, and \eqref{eqn:discr5}
      via the implicit term $\tk$ featuring in $\Dt(c)\tk$, $\Dt(z)\tk$, and
      $ \rho \EEE \, \tk\dive(\Dt(\ub))$. Indeed, having $\tk$ implicit in these terms is crucial for the  argument we will develop later on 
      for proving the positivity of $\tk$,  cf.\ the proof of Lemma \ref{l:positivityThetaDiscr}. \EEE
    \item[(iii)]
      The subgradients $\ellk$ and $\zek$  account for non-negativity as well as irreversibility constraints
      for $z$.  In the pointwise formulation we obtain by the sum rule for
      $\zkk \not= 0$ and by direct calculations for $\zkk=0$ \EEE
      %At least in the pointwise formulation we may apply the sum rule and obtain
      $$
        \partial I_{[0,\infty)}\big(\zk\big)
        +\partial I_{(-\infty,0]}\big(\Dt(z)\big)
        =  \partial I_{[0,\infty)}\big(\zk\big)
        +\partial I_{(-\infty,\zkk]}\big(\zk\big)
        =\partial I_{[0,\zkk]}(\zk) 
      $$
      and, consequently, the double inclusion in (ii) may be replaced by the single inclusion
      \begin{align*}
    		&\Dt(z)-\Delta_p(\zk)+\xk + (\conv{\sigma})'(\zk)+ (\conc{\sigma})'(\zkk)
    		\\ & =-\convWzp(\ck,\e(\ukk),\zk) - \concWzp(\ck,\e(\ukk),\zkk)+\tk
			\end{align*}
		  with
      \begin{align*}
        &\xk\in \partial I_{[0,\zkk]}(\zk).
    	\end{align*}
    \item[(iv)]
      By assuming the additional growth assumptions
      \begin{align*}
        &\sigma(0)\leq \sigma(z),\qquad b(c,0)\leq b(c,z)\text{ for all }c\in\R,z\in\R
        \text{ with }z<0,
      \end{align*}
      it is possible the prove a maximum principle for equation \eqref{eqn:discr3}
      which ensures $\zk\geq 0$ as long as $z^0\geq 0$.
      In this case the subdifferential term
      %$\partial I_{[0,\zkk]}(\zk)$
       $ \partial I_{[0,\infty)}(\zk) $ \EEE
    	in equation \eqref{eqn:discr3}
    	may be dropped. For details we refer to \cite[Proposition 5.5]{KRZ}.
  \end{itemize}
\end{remark}

We can now state our existence result for Problem \ref{def:time-discrete},
 where we also fix the concept of weak solution to system 
\eqref{PDE-discrete}. With this aim,
let us also \EEE introduce the nonlinear operator $\C A^k:X\to H^1(\Omega)'$, with
\begin{align}
&
  X:=\Big\{\theta\in H^1(\Omega)\;:\;\int_\Omega\condu(\theta)\nabla\theta\cdot\nabla v\dx
     \text{ is well-defined for all }v\in H^1(\Omega)\Big\}, \notag
\\
& 
  \big\langle\C A^k(\theta),v\big\rangle_{H^1}:=\int_\Omega \condu(\theta)\nabla\theta\cdot\nabla v\dx
     -\int_{\partial\Omega}\hk v\dx. \label{A-operator}
\end{align}
 %%%%
\begin{proposition}
\label{prop:exist-discr}
  Assume \textbf{Hypotheses (I)--(V)}, as well as \eqref{hyp:data}  on $(\mathbf{f},g,h)$ and \eqref{h:initial} on  $(c^0,z^0,\teta^0,\uu^0,\vv^0)$. \EEE

  Then, for every  $\omega>0$ and  $\tau>0$ Problem  \ref{def:time-discrete} admits a weak solution
  \begin{align}
  \label{eqn:regDiscSol}
    \{(\ck, \muk,\zk,\tk,\uk)\}_{k=1}^{K_\tau}\subseteq W^{1,p}(\Omega)\times \Hn(\Omega)\times W^{1,p}(\Omega)\times H^1(\Omega)\times   H^2(\Omega;\RR^d)\EEE
  \end{align}
%   such that
%   \begin{equation}
%   \label{strict-posit-discr}
%       \exists\,\underline{\teta}>0 \ \forall \tau>0\ \forall\, k = 1, \ldots, K_\tau\, : \quad \tk(x)\geq \underline{\teta}\;\text{ for a.e. }x\in\Omega
%   \end{equation}
  in the following sense:
  \begin{itemize}
    \item[--]
      \eqref{eqn:discr1} and \eqref{eqn:discr5} are fulfilled a.e. in $\Omega$,  with the boundary conditions $ \nabla\ck\cdot n=0$ and
           $ \uk=\dk$ a.e.\ in $\partial\Omega$, \EEE
    \item[--]
      \eqref{eqn:discr2} is fulfilled in $W^{1,p}(\Omega)'$,
    \item[--]
      \eqref{eqn:discr4} is fulfilled in $H^1(\Omega)'$,  in the form
      \[
       \begin{aligned}
        &\Dt(\vartheta)+\C A^k(\tk)+\Dt(c)\tk+\Dt(z)\tk+\rho\tk\dive(\Dt(\ub))\\
        &=\gk+|\Dt(c)|^2+|\Dt(z)|^2+m(\ckk,\zkk)|\nabla\muk|^2
        \\
        & \qquad \qquad \qquad +  a(\ckk,\zkk)\e(\Dt(\ub)):\vism\e(\Dt(\ub)),
      \end{aligned}
      \] \EEE
    \item[--]
      \eqref{eqn:discr3} is reformulated as (cf. Remark \ref{remark:discProbl} (ii))
      \begin{equation}
      \begin{aligned}
        &\Dt(z)-\Delta_p(\zk)+\xk + (\conv{\sigma})'(\zk)+ (\conc{\sigma})'(\zkk)\\
	        \label{eqn:discr3b}
        &=-\convWzp(\ck,\e(\ukk),\zk) - \concWzp(\ck,\e(\ukk),\zkk)+\tk
      \end{aligned}
      \end{equation}
  	  and fullfilled in $W^{1,p}(\Omega)'$ with
      $\xk\in \partial I_{Z_\tau^{k-1}}(\zk)$ where
      \begin{equation} \label{eqn:set_z} 
      Z_\tau^{k-1}:=\{z\in W^{1,p}(\Omega)\,|\,0\leq z\leq \zkk\},\EEE
      \end{equation} 
    \item[--]
      the initial conditions \eqref{discre-initial-cond}  and the boundary conditions \eqref{discre-boundary-cond} \EEE  are satisfied,
    \item[--]
      the  constraints \eqref{discre-constraints} are satisfied.
  \end{itemize}
\end{proposition}
We will prove Proposition \ref{prop:exist-discr} in the ensuing section by performing a double passage to the limit in
a carefully devised  approximation of system \eqref{PDE-discrete}, depending on two additional parameters $\nu$ and $\varrho$. 
%\section{Analysis}
%\subsection{Existence of weak solution -- time-discrete case}

\subsection{Proof of Proposition \ref{prop:exist-discr}}
\label{ss:5.2}
\noindent
We will split the proof of Prop.\  \ref{prop:exist-discr}  in several steps and obtain a series of intermediate results.
%Since throughout this section we shall adopt the assumptions from Proposition \ref{prop:exist-discr}, upon stating these auxiliary results
%we will omit to invoke such conditions.
%The proof of Proposition \ref{prop:exist-discr} is carried out in this subsection.
%To this end we in the whole subsection.
Our argument is based on  a double approximation procedure and two
consecutive limit passages.  More precisely, we approximate system \eqref{PDE-discrete} by
\begin{enumerate}
  \item adding the higher order terms
    \begin{align*}
      &&&&&&&+\nu\dive\big(|\nabla \muk|^{\varrho-2}\nabla\muk\big)-\nu\muk
        &&\text{to the right-hand sides of the discrete Cahn-Hilliard equation \eqref{eqn:discr1}},\\
      &&&&&&&+\nu|\ck|^{\varrho-2}\ck
        &&\text{to the right-hand sides of the discrete Cahn-Hilliard equation \eqref{eqn:discr2}},\\
      &&&&&&&+\nu|\zk|^{\varrho-2}\ck
        &&\text{to the left-hand sides of the discrete damage equation \eqref{eqn:discr3}},\\
      &&&&&&&-\nu\dive\big(|\e(\uk-\dk)|^{\varrho-2}\e(\uk-\dk)  \big) \EEE
        &&\text{to the left-hand side of the discrete momentum equation \eqref{eqn:discr5}}
    \end{align*}
    with $\nu>0$ and $\varrho>4$.
    In this way, the quadratic growth of the terms on the right-hand side of the temperature equation will be compensated
    and coercivity properties 
     of the elliptic operators involved in the time-discrete scheme \EEE
    ensured.
  \item Truncating the heat conduction function $\condu$ and replacing it with a bounded
    $\condu_M$ with $M\in\N$. In this way the elliptic operator in the discrete heat equation will be defined on $H^1(\Omega)$,
    with values in $H^1(\Omega)'$, but we will of course  loose \EEE
    the enhanced estimates on the temperature variable provided by the coercivity properties of $\condu$. That is why,
    we will have to  accordingly \EEE  truncate all occurrences of $\teta$ in the quadratic terms.
\end{enumerate}
Let us mention in advance that this  double approximation,
leading to system \eqref{discr-syst-appr} later on,  shall be
devised in such a way as to allow us  to prove the existence of
solutions to \eqref{discr-syst-appr},  by resorting to a result
from the theory of elliptic systems featuring pseudomonotone
operators,   cf.\ \cite{Rou05}.  \EEE

\textbf{A caveat on notation:} the solutions to the approximate discrete
system \eqref{discr-syst-appr} at the  $k$-th time step, with
\underline{given} $S_{\tau}^{k-1}:=(c_{\tau}^{k-1}, z_{\tau}^{k-1},
\uu_{\tau}^{k-1}, \teta_{\tau}^{k-1})$
and $\uu_\tau^{k-2}$ , will depend on the parameters $\tau$, $\nu$ and $M$
 (and on $\omega$ which we omit at the moment).
Therefore, we should
denote them by $ S_{\tau,\nu,M}^k:= (c_{\tau,\nu,M}^k,\mu_{\tau,\nu,M}^k,
z_{\tau,\nu,M}^k, \teta_{\tau,\nu,M}^k, \ub_{\tau,\nu,M}^k)$.
However, to increase readability, we will simply write $c^k$,
$\mu^k$, $z^k$, $\teta^k$ and $\ub^k$ and use the notation $c^k_M,
\ldots, \ub_M^k$ ($c^k_\nu, \ldots, \ub_\nu^k$, respectively), only
upon addressing the limit passage as $M\to\infty$ (as $\nu \down 0$,
respectively).

 \paragraph{\bf Outline of the proof of Proposition  \ref{prop:exist-discr}:} \EEE
For given $\tau>0$, the construction of 
 the solution quintuples \EEE
$S_{\tau,\nu,M}^k$ and the limit passages  as $M\to\infty$ and as $\nu \down 0$ \EEE are
performed recursively over $k=1,\ldots,K_\tau$ in the following order:
\begin{align*}
	&\qquad\vdots
		&&\qquad\vdots
		&&\quad\vdots
		&&\qquad\vdots
		&&\quad\vdots
		&&\quad\vdots
		&&\quad\vdots\\
	&(S_{\tau}^{k-2},\ub_{\tau}^{k-3})
		&&\xmapsto[\text{Step 1}]{\text{pseudo-mon. op. theory}}
		&&S_{\tau,\nu,M}^{k-1}
		&&\xrightarrow[\text{Step 2}]{\;M\to\infty\;}
		&&S_{\tau,\nu}^{k-1}
		&&\xrightarrow[\text{Step 3}]{\;\nu\downarrow0\;}
		&&S_{\tau}^{k-1}\\
	&(S_{\tau}^{k-1},\ub_{\tau}^{k-2})
		&&\xmapsto[\text{Step 1}]{\text{pseudo-mon. op. theory}}
		&&S_{\tau,\nu,M}^{k}
		&&\xrightarrow[\text{Step 2}]{\;M\to\infty\;}
		&&S_{\tau,\nu}^{k}
		&&\xrightarrow[\text{Step 3}]{\;\nu\downarrow0\;}
		&&S_{\tau}^{k}\\
	&(S_{\tau}^{k},\ub_{\tau}^{k-1})
		&&\xmapsto[\text{Step 1}]{\text{pseudo-mon. op. theory}}
		&&S_{\tau,\nu,M}^{k+1}
		&&\xrightarrow[\text{Step 2}]{\;M\to\infty\;}
		&&S_{\tau,\nu}^{k+1}
		&&\xrightarrow[\text{Step 3}]{\;\nu\downarrow0\;}
		&&S_{\tau}^{k+1}\\
	&\qquad\vdots
		&&\qquad\vdots
		&&\quad\vdots
		&&\qquad\vdots
		&&\quad\vdots
		&&\quad\vdots
		&&\quad\vdots
\end{align*}
The construction of $S_{\tau,\nu,M}^{k}$  will be tackled in Subsection \ref{sss:4.2.1},   the limit passage as $M\to\infty$
to $S_{\tau,\nu}^k$  in Subsection \ref{sss:4.2.2.}, and the one as $\nu \down 0$
  to $S_\tau^k$ in Subsection \ref{sss:4.2.3.}. 
  Throughout all of them, we will work under the assumptions of Proposition  \ref{prop:exist-discr}, and omit to explicitly invoke them in the 
  following statements. \EEE
%are tackled in the following subsections.

\subsubsection{\textbf{Step 1: Existence and uniform estimates of the
    time-discrete system with 
 ${\boldsymbol \nu}$- and ${\bf M}$-regularization.}}
\label{sss:4.2.1}
\noindent
From now on let $\nu>0$, $\varrho>4$ and $M\in\N$. Let
\begin{equation}
\label{def-k-m} \condu_M(r):=
	\left\{ \begin{array}{ll} \condu(0) & \text{if } r <0, \\
    \condu(r)   & \text{if } 0\leq r \leq M, \\
    \condu(M) & \text{if } r >M
    \end{array}
  \right.
\end{equation}
and accordingly  we \EEE introduce the quasilinear 
operator  $\mathcal{A}_M^k$  in analogy to \eqref{A-operator}\EEE:
\begin{equation}
\label{M-operator}
    \mathcal{A}_M^k: H^1(\Omega)  \to H^1(\Omega)'
    \  \text{  defined by }
    \pairing{}{H^1(\Omega)}{\mathcal{A}_M^k(\theta)}{v}:= \int_\Omega \condu_M(\theta) \nabla \theta \cdot \nabla v \dd x -
    \int_{\partial \Omega} \htau{k}v \dd S
\end{equation}
Observe that, thanks to \eqref{hyp-K}  there still holds $\condu_M(r) \geq c_{0} $ for all $r \in \R$, and therefore
by the trace theorem
\begin{equation}
\label{ellipticity-retained}
    \pairing{}{H^1(\Omega)}{ \mathcal{A}_M^{k}  (\theta)}{\theta} \geq   \tilde{c}_0 \|\nabla \theta\|_{L^2(\Omega)}^2-c_1\|\theta\|_{L^2(\Omega)}^2  -c_1\|\hk\|_{L^2(\partial\Omega)}^2 \qquad \text{for all } \theta \in H^1(\Omega).
\end{equation}
We also introduce the truncation operator $\mathcal{T}_M : \R \to \R$
\begin{equation}
\label{def-truncation-m} \mathcal{T}_M(r):=
	\left\{ \begin{array}{ll} 0 & \text{if } r <0,\\
    r  & \text{if } 0\leq r \leq M, \\
    M  & \text{if } r >M.
   \end{array}
  \right.
\end{equation}

The $(\nu,M)$-regularized time-discrete system  at time step $k$  reads as follows:
\begin{subequations}
\label{discr-syst-appr}
\begin{align}
  &D_k(c)=\dive\Big(m(c^{k-1},z^{k-1})\nabla\mu^k\Big)+\nu\dive\Big(|\nabla\mu^k|^{\varrho-2}\nabla\mu^k\Big)-\nu\mu^k,
    \label{discr-syst-appr-c}\\
  &\mu^k=-\Delta_p(c^k)+(\conv{\phi}_\omega)'(c^k)+(\conc{\phi})'(c^{k-1})+\convWcp(c^k,\e(\ub^{k-1}), z^{k-1})+\concWcp(c^{k-1},\e(\ub^{k-1}), z^{k-1})\notag\\
    &\qquad -\C T_M(\teta^k)+D_k(c)+\nu|c^k|^{\varrho-2}c^k,
    \label{discr-syst-appr-mu}\\
  &D_k(z)-\Delta_p(z^k)+\xi^k+(\conv{\sigma})'(z^k) + (\conc{\sigma})'(z^{k-1})+\nu|z^k|^{\varrho-2}z^k\notag\\
    &\quad=-\convWzp( c^{k},\e(\ub^{k-1}),z^k)-\concWzp( c^{k},\e(\ub^{k-1}),z^{k-1})+\C T_M(\teta^k)
    \quad\text{with }\xi^k\in \partial I_{[0,z^{k-1}]}(z^k),
    \label{discr-syst-appr-z}\\
  &D_k(\teta) + \mathcal{A}_M^k(\teta^k)+D_k(c)\C T_M(\teta^k)+D_k(z)\C T_M(\teta^k)+\rho\C T_M(\teta^k)\dive(D_k(\ub))\notag\\
    &\quad=g^k+|D_k(c)|^2+|D_k(z)|^2+  a(c^{k-1},z^{k-1})\e(D_k(\ub)):\vism\e(D_k(\ub))
    +m(c^{k-1},z^{k-1})|\nabla\mu^k|^2,
    \label{discr-syst-appr-teta}\\
  &D_k(D_k(\ub))-\dive\Big(a(c^{k-1},z^{k-1})\vism\e(D_k(\ub))
    + W_{,\e}^\omega(c^{k},\e(\ub^k),z^k) \EEE-\rho\C T_M(\teta^k)\mathds 1 \Big)\notag\\
    &\qquad-\nu\dive\Big(|\e(\ub^k-\db^k)|^{\varrho-2}\e(\ub^k-\db^k)\Big)=\bold f^k,
    \label{discr-syst-appr-u}
\end{align}
\end{subequations}
 supplemented with the previously given boundary conditions. \EEE
Please note that the functions $c^{k}$, $\mu^k,z^k$, $\teta^k$ and $\ub^k$
depend on $M$, $\nu$, $\tau$  and $\omega$  whereas the functions from the previous time steps
$c^{k-1}$, $\mu^{k-1},z^{k-1}$, $\teta^{k-1}$, $\ub^{k-1}$ and $\ub^{k-2}$
only depend on $\tau$  and $\omega$  and do \textbf{not} depend on $M$ and $\nu$.

We are now in the position to prove existence of weak solutions for system
\eqref{discr-syst-appr} by resorting to an existence result for pseudomonotone operators
from \cite{Rou05}, which is in turn  based on a fixed point  argument. \EEE

%%%%
\begin{lemma}[Existence of the time-discrete system for $\nu>0$ and $M\in\N$]
\label{l:exist-approx-discr}
%  Assume \textbf{Hypotheses (I)--(V)}, as well as \eqref{hyp:data}  on $(\mathbf{f},g,h)$ and \eqref{h:initial} on $(c_0,z_0,\teta_0,\uu_0,\vv_0)$.
   Let  $\omega>0$,  $\tau>0$, $k\in\{1,\ldots,K_\tau\}$, $\nu>0$ and $M\in\N$
   be given.
  We assume that
  \begin{align*}
      (c^{k-1},\mu^{k-1},z^{k-1},\teta^{k-1},\ub^{k-1},\ub^{k-2})\in W^{1,p}(\Omega)\times H^{2}(\Omega)\times W^{1,p}(\Omega)\times H^1(\Omega)\times H^2(\Omega;\R^d)\times H^2(\Omega;\R^d).
  \end{align*}
  
   Then\, there exists a weak solution
    \begin{align*}
      (c^k, \mu^k,z^k,\teta^k,\ub^k)\in W^{1,p}(\Omega)\times W^{1,\varrho}(\Omega)\times W^{1,p}(\Omega)\times H^1(\Omega)\times  W^{1,\varrho}(\Omega;\R^d)
    \end{align*}
    to system \eqref{discr-syst-appr}  at time step $k$\, in the following sense:
    \begin{itemize}
      \item[--]
        \eqref{discr-syst-appr-c} is fulfilled in $W^{1,\varrho}(\Omega)'$,
      \item[--]
        \eqref{discr-syst-appr-mu} is fulfilled in $W^{1,p}(\Omega)'$,
      \item[--]
        \eqref{discr-syst-appr-z} is fulfilled in $W^{1,p}(\Omega)'$
        with $\xi^k\in \partial I_{Z^{k-1}}(z^k)$,
      \item[--]
        \eqref{discr-syst-appr-teta} is fulfilled in $H^{1}(\Omega)'$,
      \item[--]
        \eqref{discr-syst-appr-u} is fulfilled in $W_0^{1,\varrho}(\Omega;\R^d)'$,
      \item[--]
        the initial conditions \eqref{discre-initial-cond}
				 and the boundary condition $\uu^k=\db^k$ a.e. on $\partial\Omega$ 
				are satisfied,
      \item[--]
   			the constraints \eqref{discre-constraints} are satisfied.
    \end{itemize}
% for every $k=1,\ldots, K_\tau$ there holds
%  $\tk \geq \underline{\teta}$ a.e.\ in $\Omega$  with  the constant $\underline\teta$ from \eqref{strict-posit-discr}, \emph{independent} of $M$ and $\nu$.
\end{lemma}
\begin{proof}
%\paragraph{\bf Step $1$}
%\textbf{Step $1$:}
  Our approach for finding a solution to \eqref{discr-syst-appr} for a given $k$
  is to rewrite the system as   %to a double inclusion problem viz
  \begin{align}
  \label{label:inclusion2}
    0\in \bA(c^k,\mu^k,z^k,\teta^k,\ub^k-\db^k)+
    \partial\Psi(c^k,\mu^k,z^k,\teta^k,\ub^k  -\db^k \EEE), 
  \end{align}
  where $\bA$ is a (to be specified) pseudomonotone and coercive operator
  and  $\partial\Psi$  is \EEE the  subdifferential of  a (to be specified) proper, convex and
  l.s.c. potential $\Psi$.  % and $\Psi^2$, respectively.
  Note that both the operator $\bA$ as well as $\Psi$ %and $\Psi^2$
  will depend
  on the discrete functions obtained in previous time step $k-1$, but we choose not to highlight this for notational simplicity. 

  To be more precise, we introduce the space
  $$
     \mathbf{X} \EEE:=W^{1,p}(\Omega)\times W^{1,\varrho}(\Omega)\times W^{1,p}(\Omega)\times H^1(\Omega)\times  W_0^{1,\varrho}(\Omega;\R^d)
  $$
  and the announced operator
  \begin{align*}
    &\bA=
    \begin{bmatrix}
      A_1\\A_2\\A_3\\A_4\\A_5
    \end{bmatrix}
    : \mathbf{X} \EEE\to  \mathbf{X}' \EEE
  \end{align*}
  given component-wise by
  \begin{align*}
    A_1(c,\mu,z,\vartheta,\widetilde\ub)={}&-\mu-\Delta_p(c) +\nu |c|^{\varrho-2}c
  	  +(\conv{\phi}_\omega)'(c) 
      +(\conc{\phi})'(c^{k-1})
      +\convWcp(c,\e(\uu^{k-1}),  z^{k-1})\\
      &+\concWcp(c^{k-1},\e(\uu^{k-1}), z^{k-1})
      -\C T_M(\vartheta)+(c-c^{k-1})\tau^{-1},\\
   	A_2(c,\mu,z,\vartheta,\widetilde\ub)={}&-\dive(m(c^{k-1},z^{k-1})\nabla\mu)-\nu\dive(|\nabla\mu|^{\varrho-2}\nabla\mu)+\nu\mu+(c-c^{k-1})\tau^{-1},\\
  	A_3(c,\mu,z,\vartheta,\widetilde\ub)={}&-\Delta_p(z)+\nu|z|^{\varrho-2}z+(z-z^{k-1})\tau^{-1}+(\conv{\sigma})'(\C T(z))+(\conc{\sigma})'(z^{k-1})\\
    	&\quad
      +\convWzp( c^{k},\e(\ub^{k-1}),\C T(z))+\concWzp( c^{k},\e(\ub^{k-1}),z^{k-1})-\C T_M(\vartheta),\\
  \end{align*}
  \begin{align*}
    A_4(c,\mu,z,\vartheta,\widetilde\ub)={}& \mathcal{A}_M^k(\teta)+(\vartheta-\teta^{k-1})\tau^{-1}+(c-c^{k-1})\tau^{-1}\C T_M(\vartheta)
      +(z-z^{k-1})\tau^{-1}\C T_M(\vartheta)
      \\
      & \quad
      +\rho\C T_M(\vartheta)\dive(\widetilde\uu+\db^k-\ub^{k-1})\tau^{-1}-\gk
      -|(c-c^{k-1})\tau^{-1}|^2-|(z-z^{k-1})\tau^{-1}|^2
      \\
      & \quad
      -a(c^{k-1},z^{k-1})\e((\widetilde\uu+\db^k-\ub^{k-1})\tau^{-1}):\vism\e((\widetilde\ub+\db^k-\ub^{k-1})\tau^{-1})
      -m(c^{k-1},z^{k-1})|\nabla\mu|^2,
      \\
	  A_5(c,\mu,z,\vartheta,\widetilde\ub)={}&(\widetilde\ub+\db^k-2\ub^{k-1}+\ub^{k-2})\tau^{-2}
      -\nu\dive\big(|\e(\widetilde\ub)|^{\varrho-2}\e(\widetilde\ub)\big)\\
      &\quad-\dive\Big(a(c^{k-1},z^{k-1})\vism\e((\widetilde\ub+\db^k-\ub^{k-1})\tau^{-1} ) \EEE+
       W_{,\e}^\omega( c^{k},\e(\widetilde\ub+\db^k),\C T(z)) -\rho\C T_M(\vartheta)\mathds 1\Big)\\
			&\quad-\fk,
  \end{align*}
  where we make use of the truncation  operator $\C T$  %given by
  \begin{align*}
    &\C T(z):=
    \begin{cases}
      0&\text{if }z < 0,\\
      z&\text{if }0 < z < 1,\\
      1&\text{if }z > 1.
    \end{cases}
  \end{align*}
  The potential $\Psi: \mathbf{X} \EEE \to (-\infty,+\infty]$ 
   featuring in \eqref{label:inclusion2} is \EEE
   given by
  \begin{align*}
    &\Psi(c,\mu,z,\vartheta,\widetilde\ub):=I_{Z^{k-1}}(z)=
	    \begin{cases}
        0&\text{if }0\leq z\leq z^{k-1}\text{ a.e. in }\Omega,\\
        \infty&\text{else,}
      \end{cases}
  \end{align*}
 where the set $ Z^{k-1}$ is defined in \eqref{eqn:set_z}. \EEE

 We remark that for solutions of \eqref{label:inclusion2}
  the truncation operator $\C T$  will disappear in the resulting system
  since
  $
	  \mathrm{dom}(\partial\Psi)\subseteq
    \{(c,\mu,z,\teta,\widetilde\uu)\in  \mathbf{X} \EEE \,|\, 0\leq z\leq 1\text{ a.e. in }\Omega\}.
  $
  It is merely used as an auxiliary construction to ensure coercivity of the
  operator $\bA$. 
	 Furthermore,
	the boundary values for the displacement variable are shifted to $0$
	in order to obtain a vector space structure for the domain   $\mathbf{X}$ of $\bA$. \EEE 
	As a result, we have to add $\db^k$ to the displacement $\widetilde\uu$ of the solution afterwards.

  In following we are going to verify coercivity of $\bA$.
  To this end, we will estimate $\langle \bA(\bx),\bx\rangle_{ \mathbf{X}\EEE}$ for every $\bx=(c,\mu,z,\teta,\widetilde\uu)\in  \mathbf{X}\EEE$
  from below:
  \begin{equation}
  \label{to-refer-to-later}
  \begin{aligned}
      \langle \bA(\bx),\bx\rangle_{ \mathbf{X}\EEE}=\big\langle \bA(c,\mu,z,\teta,\widetilde\uu),(c,\mu,z,\teta,\widetilde\uu)\big\rangle_{ \mathbf{X} \EEE}
      ={}&
      \big\langle A_1(c,\mu,z,\teta,\widetilde\uu),c\big\rangle_{W^{1,p}(\Omega)}
      +\big\langle A_2(c,\mu,z,\teta,\widetilde\uu),\mu\big\rangle_{W^{1,\varrho}(\Omega)}\\
      &+\big\langle A_3(c,\mu,z,\teta,\widetilde\uu),z\big\rangle_{W^{1,p}(\Omega)}
      +\big\langle A_4(c,\mu,z,\teta,\widetilde\uu),\teta\big\rangle_{H^1(\Omega)}\\
      &+\big\langle A_5(c,\mu,z,\teta,\widetilde\uu),\widetilde\uu\big\rangle_{W^{1,\varrho}(\Omega;\R^d)}\\
      &=:I_1+I_2+\ldots+I_5.
  \end{aligned}
  \end{equation}
  We now estimate the partial derivatives $\convWcp$ and $\concWcp$ of $\convWc$ and $\concWc$ w.r.t.\ $c$, i.e.\ \EEE 
	\begin{align*}
		&\convWcp(c,\e(\ub),z)= W_{,c}^\omega(c,\e(\ub),z)+\big(\sup_{\widetilde{c}\in \R}| W_{,cc}^\omega(\widetilde c,\e(\ub),z)|\big)\,c,\\
		&\concWcp(c,\e(\ub),z)=-\big(\sup_{\widetilde{c}\in \R}| W_{,cc}^\omega(\widetilde c,\e(\ub),z)|\big)\,c.
	\end{align*}		    
	  Taking into account  \eqref{eqn:convConcSplittingWc}
  and Hypothesis (V) (cf.\ also \eqref{later-ref}),  \EEE we obtain
  \begin{align}
   	&\left| \convWcp(c,\e(\ub),z)\right|
      \leq C(|c| +1 )(1+|\e(\ub)|^2),
		  \label{est-quoted-5.1}\\    		
	  &\left| \concWcp(c,\e(\ub),z)\right|
      \leq C|c|(1+|\e(\ub)|^2)
      \label{est-quoted-5.2}
  \end{align}  
     We can also verify that \EEE
  \begin{align}
  \label{oh-yes-quote}
   	&  \left| W_{,\e}^\omega(c,\e(\ub),z)\right|
      \leq C(1+|\e(\ub)|),
  \end{align}
  and
  \begin{align}
  \label{est-quoted-5.3}
    &\left| \convWzp(c,\e(\ub),z)\right| \leq C(1+|\e(\ub)|^2),\\
     \label{est-quoted-5.4}
    &\left| \concWzp(c,\e(\ub),z) \right| \leq C(1+|\e(\ub)|^2)\,.
  \end{align}
   Estimates
  \eqref{est-quoted-5.1}--\eqref{est-quoted-5.4} are valid \EEE
    for all   $c\in\R$, $z\in[0,1]$ and $\ub\in\R^d$,
  and fixed $C>0$.
	 Taking also the boundedness properties
  \[
  	\C T(z),\;z^{k-1}\in{ }[0,1]\qquad\;\;\text{a.e. in }\Omega,\qquad
  	\C T_M(\teta)\in{} [0,M]\qquad\text{a.e. in }\Omega
  \]
   into account,  we obtain
  \begin{align*}
    \convWcp(c,\e(\uu^{k-1}), z^{k-1})
 	    \geq{}& -C(|c| +1 )(1+|\e(\uu^{k-1})|^2),\\
    \concWcp(c^{k-1},\e(\uu^{k-1}), z^{k-1})
      \geq{}& -C|c^{k-1}|(1+|\e(\uu^{k-1})|^2), \\
     W_{,\e}^\omega(c,\e(\widetilde\uu+\db^k),\C T(z))
     	\geq{}&-C(1+|\e(\widetilde\uu)|^2+|\e(\db^k)|^2),\\
    \convWzp( c,\e(\ub^{k-1}),\C T(z))
      \geq{}& -C(1+|\e(\ub^{k-1})|^2),\\
    \concWzp( c,\e(\ub^{k-1}),z^{k-1})
      \geq{}& -C(1+|\e(\ub^{k-1})|^2).
  \end{align*}
  Together with Young's inequality  and estimates  \eqref{est-quoted-5.1}--\eqref{est-quoted-5.4}, \EEE a calculation reveals for the terms $I_1,\ldots,I_5$ from \eqref{to-refer-to-later} the following bounds
  (hereafter, we will write $\|\cdot\|_{L^p}$ in place of $\|\cdot\|_{L^p(\Omega)}$ for shorter notation  and we will denote by $\delta$ a  positive constant, to be chosen later, and by $C_\delta>0$ a constant depending on $\delta$): \EEE 
  \begin{align*}
      I_1={}&\|\nabla c\|_{L^p}^p+\nu\|c\|_{L^\varrho}^\varrho
        +\tau^{-1}\|c\|_{L^2}^2-\tau^{-1}\int_\Omega c^{k-1} c\dx-\int_\Omega\mu c\dx\\
        &+\int_\Omega\Big(\beta_\omega(c)+\lambda_\gamma c
        +\gamma'(c^{k-1})-\lambda_\gamma c^{k-1}
        +\convWcp(c,\e(\uu^{k-1}), z^{k-1})\Big)c\dx\\
        &+\int_\Omega\Big(\concWcp(c^{k-1},\e(\uu^{k-1}), z^{k-1})-\C T_M(\vartheta)\Big)c\dx\\
      \geq{}&\|\nabla c\|_{L^p}^p+\nu\|c\|_{L^\varrho}^\varrho
        -\delta\|\mu\|_{L^2}^2
        -C_\delta\|c\|_{L^2}^2
        -C_\delta\|\e(\uu^{k-1})\|_{L^4}^4
				-C_\delta,\\
      I_2={}&\int_\Omega m(c^{k-1},z^{k-1})|\nabla\mu|^2\dx
        +\nu\|\nabla\mu\|_{L^\varrho}^\varrho+\nu\|\mu\|_{L^2}^2+\tau^{-1}\int_\Omega (c-c^{k-1})\mu\dx\\
      \geq{}&
        \nu\|\nabla\mu\|_{L^\varrho}^\varrho+\nu\|\mu\|_{L^2}^2-\delta\|\mu\|_{L^2}^2
        -C_\delta\|c\|_{L^2}^2-C_\delta\\
      I_3={}&\|\nabla z\|_{L^p}^p+\nu\|z\|_{L^\varrho}^\varrho+\tau^{-1}\|z\|_{L^2}^2-\tau^{-1}\int_\Omega z^{k-1} z\dx\\
        &+\int_\Omega\Big((\conv{\sigma})'(\C T(z))+(\conc{\sigma})'(z^{k-1})
        +\convWzp( c,\e(\ub^{k-1}),\C T(z))+\concWzp( c,\e(\ub^{k-1}),z^{k-1})-\C T_M(\vartheta)
        \Big)z\dx\\
      \geq{}&\|\nabla z\|_{L^p}^p+\nu\|z\|_{L^\varrho}^\varrho-\delta\|z\|_{L^2}^2
        -C_\delta\|\e(\ub^{k-1})\|_{L^4}^4-C_\delta,\\
      I_4={}&\int_\Omega \condu_M(\teta)|\nabla \teta|^2\dx
        -\int_{\partial \Omega} \htau{k}\teta \dx
        +\tau^{-1}\|\teta\|_{L^2}^2-\tau^{-1}\int_\Omega\teta^{k-1}\teta\dx\\
        &+\tau^{-1}\int_\Omega\Big((c-c^{k-1})+(z-z^{k-1})
        +\rho\dive(\widetilde\uu+\db^k-\ub^{k-1})\Big)\C T_M(\teta)\teta\dx
        -\int_\Omega g^k\teta\dx\\
        &-\int_\Omega\Big(|(c-c^{k-1})\tau^{-1}|^2+|(z-z^{k-1})\tau^{-1}|^2+
        a(c^{k-1},z^{k-1})\e\Big(\frac{\widetilde\uu+\db^k-\ub^{k-1}}{\tau}\Big):\vism\e\Big(\frac{\widetilde\uu+\db^k-\ub^{k-1}}{\tau}\Big)\Big)\teta\dx\\
        &-\int_\Omega m(c^{k-1},z^{k-1})|\nabla\mu|^2\teta\dx\\
      \geq{}&c_0\|\nabla\teta\|_{L^2}^2+\tau^{-1}\|\teta\|_{L^2}^2
        -\delta\|\teta\|_{H^1}^2-C_\delta\|h^k\|_{H^{1/2}(\partial\Omega)}^2
        -C_\delta\|\teta^{k-1}\|_{L^2}^2
        -C_\delta\|c\|_{L^4}^4
        -C_\delta\|z\|_{L^4}^4
        -C_\delta\|\e(\widetilde\uu)\|_{L^4}^4\\
        &-C_\delta\|\e(\db^k)\|_{L^4}^4
				-C_\delta\| c^{k-1} \|_{L^4}^4 - C_\delta \| z^{k-1} \|_{L^4}^4 
        -C_\delta\|\e(\ub^{k-1})\|_{L^4}^4
        -C_\delta\|\nabla\mu\|_{L^4}^4
        -C_\delta\|g^k\|_{L^2}^2
        -C_\delta,\\
  \end{align*}
  \begin{align*}
      I_5={}&\nu\|\e(\widetilde\uu)\|_{L^\varrho}^\varrho
				+\tau^{-2}\|\widetilde\uu\|_{L^2}^2
        +\tau^{-2}\int_\Omega(\db^k-2\ub^{k-1}+\ub^{k-2})\cdot\widetilde\uu\dx
        +\tau^{-1}\int_\Omega a(c^{k-1},z^{k-1})\VV\e(\widetilde\uu):\e(\widetilde\uu)\dx\\
        &+\tau^{-1}\int_\Omega a(c^{k-1},z^{k-1})\VV\e(\db^k-\ub^{k-1}):\e(\widetilde\uu)\dx
        +\int_\Omega W_{,\e}^\omega(c,\e( \widetilde\uu \EEE+\db^k),\C T(z)):\e(\widetilde\ub)\dx\\
        &-\int_\Omega \rho\C T_M(\teta)\dive(\widetilde\uu)\dx
        -\int_\Omega \fb^k\cdot\widetilde\uu\dx\\
      \geq{}&\nu\|\e(\widetilde\uu)\|_{L^\varrho}^\varrho+\tau^{-2}\|\widetilde\uu\|_{L^2}^2
        -\delta\|\widetilde\uu\|_{H^1}^2-C_\delta\|\ub^{k-1}\|_{H^1}^2-C_\delta\|\db^k\|_{H^1}^2-C_\delta\|\ub^{k-2}\|_{L^2}^2
        -C_\delta\|\fb^k\|_{L^2}^2-C_\delta.
  \end{align*}
%  with $C_\delta$  a positive constant depending on $\delta>0$.
  In conclusion, choosing $\delta>0$ sufficiently small
   in such a way as to absorb the negative terms multiplied by $\delta$ into suitable positive contributions, \EEE
   we obtain   constants $c',C>0$ such that
  \begin{align*}
    \langle \bA(\bx),\bx\rangle_{ \mathbf{X}\EEE}
    \geq{}&
        c'\Big(\|\nabla c\|_{L^p}^p
        +\|c\|_{L^\varrho}^\varrho
        +\|\nabla\mu\|_{L^\varrho}^\varrho
        +\|\mu\|_{L^2}^2
        +\|\nabla z\|_{L^p}^p
        +\|z\|_{L^\varrho}^\varrho
        +\|\nabla\teta\|_{L^2}^2+\|\teta\|_{L^2}^2\Big)\\
        &+c'\Big(\|\e(\widetilde\uu)\|_{L^\varrho}^\varrho
				+\|\widetilde\uu\|_{L^2}^2\Big)-C
  \end{align*}
  which leads to coercivity of $\bA$ by using Korn's inequality.
  The pseudomonotonicity follows from standard arguments
  in the theory of quasilinear elliptic equations, cf.\  \cite[Chapter 2.4]{Rou05}.

  By virtue of the existence theorem in \cite[Theorem 5.15]{Rou05} together with
  \cite[Lemma 5.17]{Rou05}, %(in order to construct the Yosida approximation $\Psi_\varepsilon$ of $\Psi$)
  we find an $\bx\in  \mathbf{X} \EEE$ solving \eqref{label:inclusion2}.
  Thus a solution of \eqref{label:inclusion2} proves the claim.
\end{proof}
We now derive the incremental energy inequality satisfied  by the solutions to  system \eqref{discr-syst-appr}. This will be the starting point for the derivation of all a priori estimates allowing us to pass to the limit, first  as $M\to\infty$ and then $\nu \to 0$. 
\begin{lemma}[ Incremental energy inequality for the approximate discrete system]
\label{l:energy-est}
  Let $(c^k,\mu^k,z^k,\teta^k,\ub^k)$ be  the\, weak solution to
  system \eqref{discr-syst-appr}  at time step $k$  according to
  Lemma \ref{l:exist-approx-discr}.
  Then, for every $M\in\N$ and $\nu>0$ the following energy inequality holds:
  \begin{equation}
  \label{discr-total-ineq}
    \begin{aligned}
      &\mathscr{E}_\omega(c^k,z^k,\teta^k,\ub^k,\vb^k)
      +\frac\nu\varrho\|c^k\|_{L^\varrho(\Omega)}^\varrho
      +\frac\nu\varrho\|z^k\|_{L^\varrho(\Omega)}^\varrho
      +\frac\nu\varrho\|\e(\ub^k)\|_{L^\varrho(\Omega)}^\varrho
      +\nu\tau\Big(\|\nabla\mu^k\|_{L^\varrho(\Omega;\R^d)}^\varrho
      +\|\mu^k\|_{L^2}^2\Big)\\
    &\leq\mathscr{E}_\omega(c^{k-1},z^{k-1},\teta^{k-1},\ub^{k-1},\vb^{k-1})
      +\frac\nu\varrho\|c^{k-1}\|_{L^\varrho(\Omega)}^\varrho
      +\frac\nu\varrho\|z^{k-1}\|_{L^\varrho(\Omega)}^\varrho
      +\frac\nu\varrho\|\e(\ub^{k-1})\|_{L^\varrho(\Omega)}^\varrho\\
			&\qquad+\tau\Big(\int_\Omega g^k\dx+ \int_{\partial\Omega} h^k \dx +\int_\Omega \bold f^k \cdot\vb^k\dx\Big)\\
			&\qquad+\tau\io D_k(\vb)\cdot D_k(\db)\dx
			+\tau\io a(c^{k-1},z^{k-1})\VV\e(\vb^k):\e(D_k(\db))\dx\\
			&\qquad+\tau\io W_{,\e}^\omega(c^{k},\e(\uu^k),z^k):\e(D_k(\db))\dx
			-\tau\io\rho\C T_M(\teta^k)\dive(D_k(\db))\dx
			-\tau\int_\Omega \bold f^k \cdot D_k(\db)\dx
  \end{aligned}
  \end{equation}
  where we  set \EEE    $\vb^k:=D_k(\ub)$ and
  denote by $\mathscr{E}_\omega$   the approximation of the total energy
  $\mathscr{E}$ from \eqref{total-energy} obtained by replacing $\phi$ with
  $\phi_\omega = \widehat{\beta}_\omega +\gamma$
   and $W$ with $W^\omega$.
\end{lemma}
\begin{proof}
  The convex-concave splitting give rise to the following crucial estimates,
   (cf.\  also \eqref{eqn:convConcWall}):
  \begin{subequations}
  \label{eqn:convConcEst}
  \begin{align}
    &\Big((\conv{\phi}_\omega)'(c^k)+(\conc{\phi})'(c^{k-1})\Big)(c^k-c^{k-1})\geq\phi_\omega(c^k)-\phi_\omega(c^{k-1}),\\
   	&\Big((\conv{\sigma})'(z^k)+(\conc{\sigma})'(z^{k-1})\Big)(z^k-z^{k-1})\geq\sigma(z^k)-\sigma(z^{k-1}),\\
 &\Big(\convWcp(c^k,\e(\ub^{k-1}), z^{k-1})+\concWcp(c^{k-1},\e(\ub^{k-1}), z^{k-1})\Big)(c^k-c^{k-1})\notag\\
      &\quad + 	 W_{,\e}^\omega( c^{k},\e(\ub^k),z^k):\e(\ub^k-\ub^{k-1}) +  \Big(\convWzp( c^{k},\e(\ub^{k-1}),z^k)+\concWzp( c^{k},\e(\ub^{k-1}),z^{k-1})\Big)(z^k-z^{k-1})\notag
      \\
&
      \geq 
 W^\omega(c^{k},\e(\ub^k),z^k)- W^\omega(c^{k-1},\e(\ub^{k-1}),z^{k-1}) \,. \EEE
  \end{align}
  \end{subequations}
   Moreover, we will make use of standard convexity estimates:
  \begin{subequations}
  \label{eqn:stdConvEst}
  \begin{align}
    &|\nabla c^k|^{p-2}\nabla c^k\cdot\nabla (c^k-c^{k-1})
        \geq \frac 1p|\nabla c^k|^p-\frac 1p|\nabla c^{k-1}|^p,\\
    &|c^k|^{\varrho-2}c^k (c^k-c^{k-1})
        \geq \frac 1\varrho|c^k|^\varrho-\frac 1\varrho|c^{k-1}|^\varrho,\\
    &|\nabla z^k|^{p-2}\nabla z^k\cdot\nabla (z^k-z^{k-1})
        \geq \frac 1p|\nabla z^k|^p-\frac 1p|\nabla z^{k-1}|^p,\\
    &|z^k|^{\varrho-2}z^k (z^k-z^{k-1})
        \geq \frac 1\varrho|z^k|^\varrho-\frac 1\varrho|z^{k-1}|^\varrho,\\
    &|\e(\ub^k)|^{\varrho-2}\e(\ub^k):\e(\ub^k-\ub^{k-1})
        \geq \frac 1\varrho|\e(\ub^k)|^\varrho-\frac 1\varrho|\e(\ub^{k-1})|^\varrho,\\
    &\Big(\ub^k-2\ub^{k-1}+\ub^{k-2}\Big)\cdot(\ub^k-\ub^{k-1})\geq  \frac12|\ub^k-\ub^{k-1}|^2-\frac12|\ub^{k-1}-\ub^{k-2}|^2.
  \end{align}
  \end{subequations}
  To obtain the energy estimate, we test the time-discrete system
  \eqref{discr-syst-appr} as follows:
  \begin{align*}
    &
    \text{\eqref{discr-syst-appr-c}}\times(c^{ k}-c^{ k-1})\;+\;
    \text{\eqref{discr-syst-appr-mu}}\times\tau\mu^{ k}\;+\;
    \text{\eqref{discr-syst-appr-z}}\times(z^{k}-z^{ k-1})\;+\;
    \text{\eqref{discr-syst-appr-teta}}\times\tau\;\\
    &+\text{\eqref{discr-syst-appr-u}}\times(\ub^{ k}-\ub^{ k-1}-(\db^k-\db^{k-1}))
  \end{align*}
  and exploit  estimates \eqref{eqn:convConcEst}
  and \eqref{eqn:stdConvEst}.
  %Here we have used the fact that summing \eqref{eqn:convConcEstWc}, \eqref{eqn:convConcEstWz}
  %and \eqref{eqn:convConcEstWe} yields
%  on the right-hand side of the estimate (see \eqref{eqn:convConcWall}).
\end{proof}

\begin{remark}
	We note that in comparison with the calculations in the \textit{First estimate}
	in Section \ref{s:4}, where we assumed spatial $H^2$-regularity for $\uu$, we
	cannot test the weak formulation \eqref{discr-syst-appr-u} with $\ub^k-\ub^{k-1}$ because
	the boundary values of $\ub^k-\ub^{k-1}$ are not necessarily $0$.
\end{remark}

\begin{lemma}[Positivity of $\teta^k$]
\label{l:positivityThetaDiscr}
  There exists a 	constant \EEE $\ul\teta>0$,  independent of $\omega$,  $\tau$, $k$, $M$ and $\nu$,   such that $\teta^k\geq\ul\teta$ a.e. in $\Omega$.
 % where $\ul\teta$ is.
\end{lemma}
\begin{proof}
	
	The proof is carried out in two steps: At first we show non-negativity
	of $\teta^k$ and then, in the second step, strictly positivity as claimed is shown.
		
	\begin{itemize}
		\item[]\textit{Step 1:}	
			Testing the discrete heat equation \eqref{discr-syst-appr-teta} with
			$-(\teta^k)^-:=\min\{\teta^k,0\}$ shows after integration over $\Omega$:
			\begin{align*}
				&\int_\Omega \frac{1}{\tau}\underbrace{\teta^k(-(\teta^k)^-)}_{=|(\teta^k)^-|^2}\underbrace{-\frac{1}{\tau}\teta^{k-1}(-(\teta^k)^-)}_{\geq 0}
				+\Big(D_k(c)+D_k(z)+\rho\dive(D_k(\ub))\Big)\underbrace{\C T_M(\teta^k)(-(\teta^k)^-)}_{=0}\dx\\
				&=\int_\Omega\underbrace{\Big(g^k+|D_k(c)|^2+|D_k(z)|^2+ a(c^{k-1},z^{k-1})\e(D_k(\ub)):\vism\e(D_k(\ub))
		    	}_{\geq 0}\underbrace{(-(\teta^k)^-)}_{\leq 0}\dx\\
				&\quad+\int_\Omega\underbrace{
		        m(c^{k-1},z^{k-1})|\nabla\muk|^2}_{\geq 0}\underbrace{(-(\teta^k)^-)}_{\leq 0}\dx.
			\end{align*}	
			Here we have merely used the information that $\teta^{k-1}\geq 0$ a.e. in $\Omega$.
			We obtain
			$$
				\int_\Omega|(\teta^k)^-|^2\dx\leq 0
			$$
			and thus $\teta^k\geq 0$ a.e. in $\Omega$.
		\item[]\textit{Step 2:}	
			
			The proof follows the very same lines as the argument developed in \cite[Lemma 4.4 - Step 3]{RocRos14}, hence we will just outline it and  refer to  \cite{RocRos14} for all details.
			Namely, repeating the arguments formally developed in Sec.\ \ref{s:4} (cf.\ \eqref{formal-positivity}),  we deduce from
			\eqref{discr-syst-appr-teta}
			that
			there exists $C>0$ such that
			\[
			\int_\Omega D_k(\teta)w \dd x   + \int_\Omega \mathsf{K}_M(\teta^k) w \dd x \geq - C \int_\Omega (\teta^k)^2 w \dd x \qquad \text{for every }w \in W_{+}^{1,2}(\Omega)\,.
			\]
			Then, we compare the functions $(\teta^k)_{k=1}^{K_\tau}$ with the solutions
			 $(v^k)_{k=1}^{K_\tau}$ of the finite difference equation $\frac{v_k-v_{k-1}}{\tau} = - C v_k^2$, with $v_0 = \teta_*$, and we conclude that
			 $\teta^k \geq v_k $ a.e.\ in $\Omega$. Finally, with a comparison argument we prove that
			 \[
			\teta^k \geq v_k  \geq \frac{\teta_*}{1+CT\teta_*} \doteq \ul\teta \quad \aein\, \Omega \qquad\text{for all } k =1,\ldots,K_\tau.
			\]
	\end{itemize}
\end{proof}
Lemma \ref{l:energy-est} and Lemma \ref{l:positivityThetaDiscr} give rise to
the following uniform estimates:
\begin{lemma}
\label{lemma:firstEstDiscr}
  The following estimates hold uniformly in  $\nu>0$ and $M\in\N$:
  \begin{subequations}
  \label{est-5.7}
  \begin{align}
    &\|c^k\|_{W^{1,p}(\Omega)}+\|z^k\|_{W^{1,p}(\Omega)}+\|\vb^k\|_{L^2(\Omega;\R^d)}+\|\teta^k\|_{L^1(\Omega)}\leq C,
    \label{est-5.7.1}\\
    &\nu^{\frac1\varrho}\|c^k\|_{L^\varrho(\Omega)}+\nu^{\frac1\varrho}\|z^k\|_{L^\varrho(\Omega)}+\nu^{\frac1\varrho}\|\e(\ub^k)\|_{L^\varrho(\Omega;\R^{d\times d})}\leq C,
    \label{est-5.7.2}\\
    &\nu\tau\Big(\|\nabla\mu^k\|_{L^\varrho(\Omega)}^\varrho+\|\mu^k\|_{L^2(\Omega)}^2\Big)\leq C.
    \label{est-5.7.3}
  \end{align}
  \end{subequations}
\end{lemma}

\begin{proof}
 In order to deduce estimates \eqref{est-5.7}, \EEE
	 it \EEE suffices to estimate the terms of the $k$-th time step
	on the  right-hand side of the incremental energy inequality
	\eqref{discr-total-ineq} from Lemma \ref{l:energy-est}.
	The following calculations are an adaption of the
	calculations performed First estimate in Section \ref{s:4}.
	\begin{itemize}
		\item[--]
			At first we observe by Young's inequality
			\begin{align*}
				&\tau\io \fb^k\cdot\vb^k\dx
					\leq \delta\|\vb^k\|_{L^2(\Omega;\R^d)}^2+C_\delta\|\fb^k\|_{L^2(\Omega;\R^d)}^2,\\
				&\tau\io D_k(\vb)\cdot D_k(\db)\dx
					\leq \delta\|\vb^k\|_{L^2(\Omega;\R^d)}^2
						+\delta\|\vb^{k-1}\|_{L^2(\Omega;\R^d)}^2+C_\delta\|D_k(\db)\|_{L^2(\Omega;\R^d)}^2,\\
				&-\tau\io \fb^k\cdot D_k(\db)\dx
					\leq C\|\fb^k\|_{L^2(\Omega;\R^d)}^2+C\|D_k(\db)\|_{L^2(\Omega;\R^d)}^2.
			\end{align*}
			By choosing $\delta>0$ sufficiently small, the term
			$\delta\|\vb^k\|_{L^2(\Omega;\R^d)}^2$
			is absorbed by the left-hand side of \eqref{discr-total-ineq}.
			The remaining terms are bounded due to \eqref{hyp:data}.
		\item[--]
			We continue with the next term
			 on the right-hand side of  \eqref{discr-total-ineq} \EEE
			 by using that $\vb^k=D_k(\db)$ a.e. on $\partial\Omega$,
			the trace theorem and Young's inequality
			\begin{align*}
				\qquad&\tau\io a(c^{k-1},z^{k-1})\VV\e(\vb^k):\e(D_k(\db))\dx\\
				&=-\tau\io\vb^k\cdot\dive\big(a(c^{k-1},z^{k-1})\VV\e(D_k(\db))\big)\dx
					+\tau\int_{\partial\Omega} \vb^k\cdot\big(a(c^{k-1},z^{k-1})\VV\e(D_k(\db))n\big)\dd S\\
				&=-\tau\io\vb^k\cdot\Big(\big(a_{,c}(c^{k-1},z^{k-1})\nabla c^{k-1}
						+a_{,z}(c^{k-1},z^{k-1})\nabla z^{k-1}\big)\VV\e(D_k(\db))\Big)\dx\\
					&\quad-\tau\io\vb^k\cdot a(c^{k-1},z^{k-1})\dive\big(\VV\e(D_k(\db))\big)\dx
					+\tau\int_{\partial\Omega} D_k(\db)\EEE\cdot\big(a(c^{k-1},z^{k-1})\VV\e(D_k(\db))n\big)\dd S\\
				&\leq \delta\|\vb^k\|_{L^2(\Omega;\R^d)}^2
					+C_\delta\|\e(D_k(\db))\|_{L^\infty(\Omega;\R^{d\times d})}^2
					\Big(\|a_{,c}(c^{k-1},z^{k-1})\|_{L^\infty(\Omega)}^2\|\nabla c^{k-1}\|_{L^2(\Omega;\R^d)}^2\\
						&\hspace*{20em}+\|a_{,z}(c^{k-1},z^{k-1})\|_{L^\infty(\Omega)}^2\|\nabla z^{k-1}\|_{L^2(\Omega;\R^d)}^2\Big)\\
					&\quad+C_\delta\|a(c^{k-1},z^{k-1})\|_{L^\infty(\Omega)}^2\|\tau\dive(\VV\e(D_k(\db)))\|_{L^2(\Omega;\R^d)}^2\\
					&\quad+C\|D_k(\db^k)\|_{H^1(\Omega;\R^d)}^2\|a(c^{k-1},z^{k-1})\|_{L^\infty(\Omega)}\|\tau\e(D_k(\db))n\big)\|_{H^1(\Omega;\R^{d\times d})}^2.
			\end{align*}
			Taking Hypothesis (IV) and \eqref{dirichlet-data} into account,  we ultimately find that \EEE
			\begin{align*}
				&\tau\io a(c^{k-1},z^{k-1})\VV\e(\vb^k):\e(D_k(\db))\dx\\
					&\qquad\leq\delta\|\vb^k\|_{L^2(\Omega;\R^d)}^2
						+C_\delta\big(\|\nabla c^{k-1}\|_{L^2(\Omega;\R^d)}^2+\|\nabla z^{k-1}\|_{L^2(\Omega;\R^d)}^2+1\big).
			\end{align*}
			For small $\delta>0$ the first term on the right-hand side can be absorbed
			 into \EEE the left-hand side of \eqref{discr-total-ineq}.
		\item[--]
			Moreover, we estimate the $\tau\int W_{,\e}^\omega(\ldots):\ldots$-term on the right-hand side of
			\eqref{discr-total-ineq} as follows
			\begin{align*}
				\qquad&\tau\io W_{,\e}^\omega(c^{k},\e(\uu^k),z^k):\e(D_k(\db))\dx\\
					&  \quad\leq \delta\|b(c^{k},z^{k})\|_{L^\infty(\Omega)}\io \underbrace{\frac 12b(c^{k},z^{k})\CC(\e(\uu^k)-\e^*(c)):(\e(\uu^k)-\e^*(c))}_{=W(c^k,\e(\uu^k),z^k)}\dx
					+C_\delta\|\e(D_k(\db))\|_{L^2(\Omega;\R^{d\times d})}^2,
			\end{align*}
			which can be absorbed by the left-hand side of \eqref{discr-total-ineq}
			for small $\delta>0$.
		\item[--]
			Finally,
			\begin{align*}
				&-\tau\io\rho\C T_M(\teta^k)\dive(D_k(\db))\dx
					\leq\tau\rho\|\dive(D_k(\db))\|_{L^\infty(\Omega)}\int_\Omega|\teta^k|\dx.
			\end{align*}
	\end{itemize}
	 In the end, by  \EEE choosing $\tau>0$ small enough depening only on $\rho$ and the data
			$\db$, the right-hand side can be absorbed by the left-hand side of
			\eqref{discr-total-ineq} (recall that $\teta^k$ is positive). 
\end{proof}
\begin{remark}
	We see that the calculation above takes advantage of the fact
	that the $W_{,\e}(\ldots)$-term in the discrete force balance equation
	is discretized fully implicit.
\end{remark}

\subsubsection{\textbf{Step 2: Limit passage $M\to\infty$.}}
\label{sss:4.2.2.}
\noindent
In the following we focus on the limit passage $M\to\infty$
and keep $M$ as a subscript in $c_M^k$, $\mu_M^k$, $z_M^k$, $\ub_M^k$ and $\teta_M^k$.
By adapting the proof in \cite[Proof of Lemma 4.4 - Step 4]{RocRos14} to our situation, we obtain
enhanced estimates for $(\teta_M^k)_M$.
\begin{lemma}
\label{lemma:theteEstDiscr}
  The following estimate holds uniformly in $M\in\N$:
  \begin{align}
  \label{eqn:thetaEstDiscr}
    \|\tM\|_{H^1(\Omega)}\leq C.
  \end{align}
\end{lemma}
\begin{proof}
    In \cite[Proof of Lemma 4.4 - Step 4]{RocRos14}  estimate \eqref{eqn:thetaEstDiscr}  is obtained in two steps
    which can be both applied in our case since the additional variable $c$
    enjoys the same regularity properties and estimates as $z$.
    At first \eqref{discr-syst-appr-teta} is tested by $\C T_M(\tM)$ leading to
    the estimates
    \begin{align*}
        \|\C T_M(\tM)\|_{H^1(\Omega)}+\|\C  T_M \EEE (\tM)\|_{L^{3\kappa+6}(\Omega)}\leq C.
    \end{align*}
    Secondly, \eqref{discr-syst-appr-teta} is tested by $\tM$ leading to
    the claimed estimate \eqref{eqn:thetaEstDiscr}.
\end{proof}
\begin{lemma}
\label{lemma:discrConvergence}
    For given $\nu>0$\, there exist functions
    \begin{align*}
        &c^k\in W^{1,p}(\Omega), \EEE % \text{ with }c\in(\alpha,\beta)\text{ a.e. in }\Omega,
        &&\mu^k\in W^{1,\varrho}(\Omega),
        &&z^k\in W^{1,p}(\Omega)\text{ with }z\in[0,1]\text{ a.e. in }\Omega,
        \\
        &\teta^k\in H^1(\Omega)\text{ with }\teta^k\geq\ul\teta>0\text{ a.e. in }\Omega,
        &&\ub^k\in  W^{1,\varrho}(\Omega;\R^d)
    \end{align*}
    such that for a subsequence $M\to\infty$
    \begin{subequations}
    \label{eqn:discrConv}
    \begin{align}
        &c_M^k \to c^k\text{ strongly in }W^{1,p}(\Omega),\label{eqn:strongConvCDiscr}\\
        &\mu_M^k \to \mu^k\text{ strongly in }W^{1,\varrho}(\Omega),\label{eqn:strongConvMuDiscr}\\
        &z_M^k \to z^k\text{ strongly in }W^{1,p}(\Omega),\label{eqn:strongConvZDiscr}\\
        &\teta_M^k \weaklim \teta^k\text{ weakly in }H^{1}(\Omega), \label{eqn:strongConvTDiscr}\\
        &\ub_M^k \to \ub^k\text{ strongly in }W^{1,\varrho}(\Omega;\R^d).\label{eqn:strongConvUDiscr}
    \end{align}
    \end{subequations}
\end{lemma}
\begin{proof}
   First of all, observe that  \EEE the a priori estimates in Lemma \ref{lemma:firstEstDiscr} and Lemma \ref{lemma:theteEstDiscr}
  imply \eqref{eqn:discrConv} with weak instead of strong topologies.

  The strong convergence \eqref{eqn:strongConvZDiscr} may be shown by rewriting
  \eqref{discr-syst-appr-z} as a variational inequality
	\begin{equation}
  \begin{aligned}
	  &-\int_\Omega|\nabla \zM|^{p-2}\nabla \zM\cdot \nabla (\zeta-\zM)\dx
	  \\
	  &\qquad\leq\int_\Omega\Big(\frac{\zM- z^{k-1} }{\tau}+(\conv{\sigma})'(z_M^k) +
	  (\conc{\sigma})'(z^{k-1})+\nu|z_M^k|^{\varrho-2}z_M^k\Big)(\zeta-z_M^k)\dx\\
	  &\qquad\quad+\int_\Omega\Big(\convWzp( c^{k}, \e(\ub^{k-1}),  \zM)
	  +\concWzp( c^{k},\e(\ub^{k-1}),z^{k-1})-
	  \C T_M(\teta_M^k)\Big)(\zeta-\zM)\dx
	\label{eqn:varIneqZDiscr}
  \end{aligned}
  \end{equation}
  holding for all $\zeta\in W^{1,p}(\Omega)$ with $0\leq \zeta\leq z^{k-1}$ a.e. in $\Omega$.

  To proceed we can argue by recovery sequences:
  By now we know the following:
  \begin{align*}
    &0\leq z_M^k\leq z^{k-1}\\
    &\qquad\downarrow\qquad\downarrow\qquad\text{ weakly in }W^{1,p}(\Omega)\text{ as }M\to\infty.\\
    &0\leq z^k\;\leq z^{k-1}
  \end{align*}
  Due to the compact embedding $W^{1,p}(\Omega)\hookrightarrow \rmC^0(\ol\Omega)$, we find another
  sequence denoted by $\widetilde z_M^k$ such that
  \begin{align*}
    &\widetilde z_M^k\to z^k\text{ strongly in }W^{1,p}(\Omega)\text{ as }M\to\infty\quad\text{ and }\quad
    0\leq\widetilde z_M^k\leq z^{k-1}. 
  \end{align*}
  We may take, for instance, $\widetilde z_M^k:=\max\{z^k-\delta_M,0\}$ for suitable values $\delta_M>0$ with $\delta_M\to 0$
  as $M\to\infty$.

  We test \eqref{eqn:varIneqZDiscr} with the admissible function $\zeta=\widetilde z_M^k$. Taking into account the
  already proved weak convergences \eqref{eqn:discrConv} as well
  as the growth properties of the functions $(\conv{\sigma})'$ and
  $\convWzp$ (cf.\ \eqref{est-quoted-5.3}), we manage to pass to
  the limit on the right-hand side of \eqref{eqn:varIneqZDiscr} and conclude that 
  \begin{align}
  \label{label-2-added}
    &\limsup_{M\to\infty}\int_\Omega-|\nabla z_M^k|^{p-2}\nabla z_M^k\cdot \nabla (\widetilde z_M^k-z_M^k)\dx\leq 0.
  \end{align}
  By exploiting the  uniform $p$-convexity of
  the $\|\cdot\|_{L^p(\Omega)}^p$-function
  and strong $W^{1,p}(\Omega)$-convergence of the recovery sequence, 
   from \eqref{label-2-added} we deduce that \EEE
   $\|\nabla(\widetilde z_M^k-z_M^k)\|_{L^p(\Omega)}\to 0$ as $M\to\infty$.
  Together with $\|\nabla(\widetilde z_M^k-z^k)\|_{L^p(\Omega)}\to 0$, 
  property \eqref{eqn:strongConvZDiscr} is shown.

  To prove the strong convergences \eqref{eqn:strongConvCDiscr},
  \eqref{eqn:strongConvMuDiscr} and \eqref{eqn:strongConvUDiscr},
  we use a $\limsup$--argument.
  We adapt the proof from \cite[Proof of Lemma 4.4 - Step 4]{RocRos14} to our situation:
  \begin{itemize}
    \item[--]
      Let $\Lambda\in L^{\varrho/(\varrho-1)}(\Omega;\R^d)$ be a weak cluster point of $|\nabla\mu_M^k|^{\varrho-2}\nabla\mu_M^k$.
      Testing \eqref{discr-syst-appr-c} 
      with $\mu_M^k$ yields by exploiting a lower   semicontinuity \EEE argument
      \begin{align*}
        \limsup_{M\to\infty}\nu\int_\Omega|\nabla\mu_M^k|^{\varrho}\dx
	        ={}&\limsup_{M\to\infty}\int_\Omega-\frac{c_M^k-c^{k-1}}{\tau}\mu_M^k-m(c^{k-1},z^{k-1})|\nabla\mu_M^k|^2-\nu|\mu_M^k|^2\dx\\
          \leq{}&\int_\Omega-\frac{ c^k-c^{k-1}}{\tau}\mu^k-\liminf_{M\to\infty}\int_\Omega m(c^{k-1},z^{k-1})|\nabla\mu_M^k|^2\dx-\int_\Omega\nu|\mu|^2\dx\\
          \leq{}&\int_\Omega-\frac{ c^k-c^{k-1}}{\tau}\mu^k-m(c^{k-1},z^{k-1})|\nabla\mu|^2-\nu|\mu|^2\dx.
      \end{align*}
      However, the right-hand side equals $\nu\int_\Omega\Lambda\cdot\nabla\mu\dx$
      by passing to the limit $M\to\infty$ in \eqref{discr-syst-appr-mu}
      and testing the limit equation with $\mu$.
      In conclusion, taking into account the previously
      proved convergences we have that 
      \begin{align*}
        \limsup_{M\to\infty}\int_\Omega|\nabla\mu_M^k|^{\varrho}\dx
          \leq \int_\Omega\Lambda\cdot\nabla\mu\dx,
      \end{align*}
      which results in \eqref{eqn:strongConvMuDiscr}.
    \item[--]
      Convergence \eqref{eqn:strongConvCDiscr} can be gained with a similar argument as above, whereas 
      \eqref{eqn:strongConvUDiscr} can be shown as in \cite[Proof of Lemma 4.4 - Step 4]{RocRos14}.
  \end{itemize}
\end{proof}
We are now in the position to carry out the limit passage as $M\to\infty$  and conclude the existence of a solution to 
an intermediate approximate version of the time-discrete system \eqref{PDE-discrete}, only featuring the higher regularizing terms and the $\omega$-regularizations,
i.e.\ \eqref{discr-syst-appr2}
 below.   \EEE
\begin{lemma}[Existence of the time-discrete system for $\nu>0$ and $M\to\infty$]
\label{lemma:5.7}
	
	Let the assumption from Lemma \ref{l:exist-approx-discr} be fulfilled.
  Then\, for every  $\nu>0$\, there exists a weak solution
  \begin{align*}
    \{(c^k, \mu^k,z^k,\teta^k,\ub^k)\}_{k=1}^{K_\tau}\subseteq W^{1,p}(\Omega)\times W^{1,\varrho}(\Omega)\times W^{1,p}(\Omega)\times H^1(\Omega)\times  W^{1,\varrho}(\Omega;\R^d)
  \end{align*}
  to the following time-discrete PDE system:
  \begin{subequations}
  \label{discr-syst-appr2}
  \begin{align}
    &D_k(c)=\dive\Big(m(c^{k-1},z^{k-1})\nabla\mu^k\Big)+\nu\dive\Big(|\nabla\mu^k|^{\varrho-2}\nabla\mu^k\Big)-\nu\mu^k
      &&\text{in }W^{1,\varrho}(\Omega)',
      \label{discr-syst-appr-c2}\\
    &\mu^k=-\Delta_p(c^k)+(\conv{\phi}_\omega)'(c^k)+(\conc{\phi})'(c^{k-1})
    	+\convWcp(c^k,\e(\ub^{k-1}), z^{k-1})\notag\\
      &\qquad +\concWcp(c^{k-1},\e(\ub^{k-1}), z^{k-1})-\teta^k+D_k(c)+\nu|c^k|^{\varrho-2}c^k
      &&\text{in }W^{1,p}(\Omega)',
      \label{discr-syst-appr-mu2}\\
    &D_k(z)-\Delta_p(z^k)+\xi^k+(\conv{\sigma})'(z^k) + (\conc{\sigma})'(z^{k-1})+\nu|z^k|^{\varrho-2}z^k\notag\\
      &\quad=-\convWzp( c^{k},\e(\ub^{k-1}),z^k)-\concWzp( c^{k},\e(\ub^{k-1}),z^{k-1})+\teta^k\notag\\
      &\quad\text{with }\xi^k\in \partial I_{Z^{k-1}}(z^k)
      &&\text{in }W^{1,p}(\Omega)',
      \label{discr-syst-appr-z2}\\
    &D_k(\teta) + \mathcal{A}^k(\teta^k)+D_k(c)\teta^k+D_k(z)\teta^k+\rho\teta^k\dive(D_k(\ub))\notag\\
      &\quad=g^k+|D_k(c)|^2+|D_k(z)|^2+ a(c^{k-1},z^{k-1})\e(D_k(\ub)):\vism\e(D_k(\ub))\notag\\
      &\qquad+m(c^{k-1},z^{k-1})|\nabla\mu^k|^2
      &&\text{in }H^{1}(\Omega)',
      \label{discr-syst-appr-teta2}\\
    &D_k(D_k(\ub))-\dive\Big(a(c^{k-1},z^{k-1})\vism\e(D_k(\ub))
      + W_{,\e}^\omega( c^{k},\e(\ub^k),z^k)\big) -\rho\teta^k\mathds 1 \Big)\notag\\
      &\qquad-\nu\dive\Big(|\e(\ub^k-\db^k)|^{\varrho-2}\e(\ub^k-\db^k)\Big)=\fb^k
      &&\text{in }W_0^{1,\varrho}(\Omega;\R^d)'
      \label{discr-syst-appr-u2}
  \end{align}
  \end{subequations}
  satisfying the initial conditions \eqref{discre-initial-cond},
   the boundary condition $\uu^k=\db^k$ a.e. on $\partial\Omega$ 
  and the  constraints \EEE  \eqref{discre-constraints}.
\end{lemma}
\begin{proof}
  At the beginning we notice that
  \begin{align}
  \label{eqn:TtetaConv}
    \C T_M(\teta_M^k)\to\teta^k\text{ strongly in }L^{p^*-\epsilon}(\Omega)\text{ for all }
      \epsilon\in  (0,p^*-1] \EEE \text{ as }M\to\infty,
  \end{align}
  which follows from the pointwise convergence $\C T_M(\teta_M^k)\to\teta^k$ as $M\to\infty$
  a.e. in $\Omega$ and  the \EEE uniform boundedness of $\|\C T_M(\teta_M^k)\|_{L^{p^*}(\Omega)}$ with respect to $M$.

  We see that with the help of Lemma \ref{lemma:discrConvergence} and
  \eqref{eqn:TtetaConv}, also taking into account the growth
  properties of %$W_{,\eps}$, 
   $W_{,\e}^\omega$ (cf.\ \eqref{oh-yes-quote}), \EEE
  we can pass to $M\to\infty$   along \EEE  a subsequence in \eqref{discr-syst-appr-c}  for $c$ 
  and \eqref{discr-syst-appr-u} for $\uu$ \EEE and obtain
  \eqref{discr-syst-appr-c2} and \eqref{discr-syst-appr-u2}, respectively.
  The limit passages for the remaining equations are carried out  as follows: %in the following:
  \begin{itemize}
    \item[--]
      It follows from
      $\| c_M^k \|_{W^{1,p}(\Omega)} \leq C$ and from the Lipschitz continuity of $\beta_\omega$  that
      $(\conv{\phi}_\omega)'(c_M^k) = \beta_\omega(c_M^k) + \lambda c_M^k$ is  bounded, uniformly in $M$ and $k=1,\ldots,K_\tau$, in
      $L^\infty(\Omega)$. This  and the growth properties of $\convWcp$ and $\concWcp$  (cf.\ \eqref{est-quoted-5.1}--\eqref{est-quoted-5.2}), \EEE
      together with Lemma \ref{lemma:discrConvergence} and  convergence \EEE
      \eqref{eqn:TtetaConv},
      enable us to pass to $M\to\infty$
      in  equation  \eqref{discr-syst-appr-mu} for $\mu$. \EEE  We find \eqref{discr-syst-appr-mu2}.
    \item[--]
      The limit passage  in equation   \eqref{discr-syst-appr-z}   for $z$ \EEE  is managed via the variational
      formulation \eqref{eqn:varIneqZDiscr}.
      To this end we pick an arbitrary test-function $\zeta\in W^{1,p}(\Omega)$ with
      $0\leq \zeta\leq z^{k-1}$ and construct the recovery sequence
      \begin{align*}
        \zeta_M:=\max\{z^{k-1}-\delta_M,0\}
      \end{align*}
      for suitable values $\delta_M>0$ with $\delta_M\to 0$ such that
      $0\leq \zeta_M\leq  z^{k-1} $ is fulfilled for all $M\in\N$.
      Now, testing \eqref{eqn:varIneqZDiscr} with $\zeta_M$ and
      passing to $M\to\infty$ with the help of Lemma \ref{lemma:discrConvergence}
      and \eqref{eqn:TtetaConv} yields \eqref{discr-syst-appr-z2}.
    \item[--]
      By exploiting Lemma \ref{lemma:discrConvergence}, property \eqref{eqn:TtetaConv}
      and a comparison argument as done in \cite[Lemma 4.4, Step 3]{RocRos14}
      we find
      \begin{align*}
        \C A_M^k(\teta_M^k)\weaklim \C A^k(\teta^k) \ \ \text{  weakly \EEE  in }H^1(\Omega)'
 	       \text{ as }M\to\infty.
      \end{align*}
      This allows us to pass to the limit $M\to\infty$ in  equation
      \eqref{discr-syst-appr-teta} for $\teta$ \EEE  in order to obtain \eqref{discr-syst-appr-teta2}.
  \end{itemize}
\end{proof}

\subsubsection{\textbf{Step 3: Limit passage ${ \boldsymbol \nu } {
      \boldsymbol \downarrow  }{\bf 0}$. \EEE}}
\label{sss:4.2.3.}
\noindent

We now address the limit passage $\nu \down 0$ and denote by $(c_\nu^k,\mu_\nu^k,z_\nu^k,\uu_\nu^k)_\nu$ the family of solutions to system
\eqref{discr-syst-appr2}
found in Lemma \ref{lemma:discrConvergence}. By lower semicontinuity, estimates \eqref{est-5.7} from
Lemma \ref{lemma:firstEstDiscr} are thus inherited  by the functions $(c_\nu^k,\mu_\nu^k,z_\nu^k,\uu_\nu^k)_\nu$. \EEE
 Furthermore, we obtain a uniform $H^1(\Omega)$-estimate for
$(\teta_\nu^k)_\nu$. Indeed, since the higher order terms
$$
	\nu\dive\big(|\nabla  \mu_\nu^k|^{\varrho-2}\nabla\mu_\nu^k\big)-\nu\mu_\nu^k,
		\ldots \ldots, -\nu\dive\big(|\e(\uu_\nu^k-\db^k)|^{\varrho-2}\e(\uu_\nu^k-\db^k))\EEE
$$
vanish  as $\nu \down 0$, \EEE  we  loose \EEE the $L^2(\Omega)$-estimate for  the right-hand side of the discrete temperature equation \eqref{discr-syst-appr-teta2}.
Therefore, to prove this $H^1$-bound for $\teta_\nu^k$ we have to resort to the arguments from the proof of the \emph{Second a priori estimate} in Sec.\ \ref{s:4},
 and in particular fully exploit the coercivity properties of the function $\condu$. \EEE
\begin{lemma}
\label{lemma:theteEstDiscr-nu}
  The following estimates holds uniformly in $\nu>0$:
  \begin{align}
  \label{eqn:thetaEstDiscr-nu}
    \|\teta_\nu^k \|_{H^1(\Omega)}\leq C,\qquad
    \|(\teta_\nu^k)^{(\kappa+\alpha)/2} \|_{H^1(\Omega)} \leq C_\alpha
    	\quad\text{for all }\alpha\in(0,1).
  \end{align}
\end{lemma}
\begin{proof}
	We test \eqref{discr-syst-appr-teta2} by $(\teta_\nu^k)^{\alpha-1}$, with $\alpha \in (0,1)$. With the very same calculations as for the \emph{Second a priori estimate} (cf.\ 
	  \eqref{all-in-all} and \EEE
	also the proof of Prop.\ \ref{prop:aprio-discr} later on), we conclude
	\[
		\begin{aligned}
			&c\int_\Omega \condu(\teta_\nu^k) |\nabla( \teta_\nu^k)^{\alpha/2}|^2 \dd x + c\int_\Omega 
			\left( 
			\left| \eps\left( D_k ( u_\nu^k \EEE)\right) \right|^2  + |\nabla \mu_\nu^K|^2 \right) \EEE
				(\teta_\nu^k)^{\alpha-1}\dd x + c \int_\Omega \left(  \left| D_k( z_\nu^k \EEE )\right|^2 + \left| D_k( c_\nu^k \EEE)\right|^2 \right)
				(\teta_\nu^k)^{\alpha-1}\dd x\\
				&\leq C + C\int_\Omega (\teta_\nu^k)^{\alpha+1} \dd x\,.
		\end{aligned}
	\]
	Then,
	with the same arguments as  in Sec.\ \ref{s:4},
	we arrive at
	$\int_\Omega |\nabla (\teta_\nu^k)^{(\kappa+\alpha)/2}|^2 \dd x \leq C$
	for a constant independent of $\nu$. Ultimately, we conclude
	\eqref{eqn:thetaEstDiscr-nu}.
\end{proof}

By comparison arguments based on the \textit{Third estimate}
we   then \EEE obtain uniform estimates for $(\ub_\nu^k)_\nu$ and for $(\mu_\nu^k)_\nu$
with respect to $\nu$.
\begin{lemma}
	The following estimates hold uniformly in $\nu>0$:
    \begin{align}
    \label{eqn:uMuEstDiscr-nu}
        \|\ub_\nu^k \|_{H^1(\Omega;\R^d)} + \|\mu_\nu^k\|_{H^1(\Omega)}\leq C.
    \end{align}
\end{lemma}
\begin{proof}
	We proceed as in the \textit{Third estimate} in Section \ref{s:4}: Testing the
	time-discrete heat equation \eqref{discr-syst-appr-teta2} with $\tau$,
	and subtracting the resulting equation from the incremental energy inequality \eqref{discr-total-ineq}
	(the limit version $M\to\infty$).
	In particular we obtain boundedness with respect to $\nu$ of
	\begin{align*}
		\int_\Omega a(c^{k-1},z^{k-1})\frac{\e(\ub_\nu^k-\ub^{k-1})}{\tau}:\VV\frac{\e(\ub_\nu^k-\ub^{k-1})}{\tau}\dx
		+
		\int_\Omega m(c^{k-1},z^{k-1})|\nabla\mu_\nu^k|^2\dx\leq C.
	\end{align*}
	Hence $\|\e(\ub_\nu^k)\|_{L^2(\Omega;\R^{d\times d})}$ and $\|\nabla\mu_\nu^k\|_{L^2(\Omega)}$ are bounded in $\nu$.
	Korn's inequality applied to $\uu_\nu^k-\db^k$
	shows the first part of the claim, namely boundedness
	of $\|\ub_\nu^k \|_{H^1(\Omega;\R^d)}$.
	
	The proof of the second part makes use of the Poincar\'e inequality. To this end
	boundedness of the spatial mean of $\mu_\nu^k$ has to be shown.
	Testing the time-discrete equation \eqref{discr-syst-appr-mu2} with $1/|\Omega|$ shows
	\begin{align*}
    \dashint_\Omega\mu_\nu^k\dx={}&\dashint_\Omega(\conv{\phi}_\omega)'(c_\nu^k)+(\conc{\phi})'(c^{k-1})
    	+\convWcp(c_\nu^k,\e(\ub^{k-1}),z^{k-1})+\concWcp(c^{k-1},\e(\ub^{k-1}),z^{k-1})\dx\\
	    &+\dashint_\Omega-\teta_\nu^k+\frac{c_\nu^k-c^{k-1}}{\tau}+\nu|c_\nu^k|^{\varrho-2}c_\nu^k\dx.
	\end{align*}
	By the known boundedness properties of $(c_\nu^k)_\nu$, $(\ub_\nu^k)_\nu$, $(z_\nu^k)_\nu$ and $(\teta_\nu^k)_\nu$,
	and the growth of $\convWcp$, $\concWcp$ 
	 (cf.\ \eqref{est-quoted-5.1}--\eqref{est-quoted-5.2}), \EEE
	and of $(\conv{\phi}_\omega)'$ (affine-linear growth
	in $c$ due to Yosida approximation with parameter $\tau$),
	we then infer boundedness of $\dashint_\Omega\mu_\nu^k$.
	Together with boundedness of $\|\nabla\mu_\nu^k\|_{L^2(\Omega)}$ we  conclude \EEE the second
	part of the claim by the Poincar\'e inequality.
\end{proof}

We then have the following counterpart to Lemma \ref{lemma:discrConvergence}, which reflects the lesser regularity of
the solution components $\mu^k$ and $\uu^k$ as a result of the limit passage  as $\nu \down 0$.
Its proof is a straightforward adaptation of the argument developed
for Lemma \ref{lemma:discrConvergence}.
\begin{lemma}
\label{lemma:discrConvergence-nu}
    There exist $(c^k,\mu^k, z^k,\teta^k,\ub^k) \in W^{1,p}(\Omega) \times H^1(\Omega)\times W^{1,p}(\Omega) \times H^1(\Omega)\times  H^1(\Omega;\R^d)$\EEE
    and a (not relabeled) subsequence $\nu \down 0$
    such that convergences \eqref{eqn:strongConvCDiscr}, \eqref{eqn:strongConvZDiscr}--\eqref{eqn:strongConvTDiscr} hold, as well as
    \begin{subequations}
    \label{eqn:discrConv-nu}
    \begin{align}
      &\mu_\nu^k \to \mu^k\text{ strongly in }H^1(\Omega),\label{eqn:strongConvMuDiscr-nu}\\
			&\nu|\nabla\mu_\nu^k|^{\varrho-2}\nabla\mu_\nu^k \to 0 \text{ strongly in } L^{\varrho/(\varrho-1)} (\Omega;\R^{d}),\label{conv-nu-muk-zero}\\
			&\ub_\nu^k \to \ub^k\text{ strongly in }H^1(\Omega;\R^d),\label{eqn:strongConvUDiscr-nu}\\
			& \nu |\eps(\ub_\nu^k-\db^k)|^{\varrho-2}\eps(\ub_\nu^k-\db^k) \to 0 \text{ strongly in } L^{\varrho/(\varrho-1)} (\Omega;\R^{d\times d}). \label{conv-nu-uk-zero}
    \end{align}
    \end{subequations}
\end{lemma}

We are  now  in the position to carry out the
\underline{\textbf{limit passage as  ${\boldsymbol \nu }{\boldsymbol
      \down} {\bf 0}$ \EEE in system
\eqref{discr-syst-appr2}}}.  The aguments for taking the limits in
\eqref{discr-syst-appr-c2}, \eqref{discr-syst-appr-mu2},
\eqref{discr-syst-appr-z2}, and \eqref{discr-syst-appr-u2} are
completely analogous to those developed in the proof of Lemma
\ref{lemma:5.7}.

Hence we only comment on the limit passage in the discrete heat
equation \eqref{discr-syst-appr-teta2}.
 Estimate \eqref{eqn:thetaEstDiscr-nu} allows us to conclude that, up to a subsequence,
$(\teta_\nu^k)^{(\kappa+\alpha)/2} \weakto
(\teta^k)^{(\kappa+\alpha)/2}$ in $H^1(\Omega)$, hence
$(\teta_\nu^k)^{(\kappa+\alpha)/2} \to
(\teta^k)^{(\kappa+\alpha)/2}$ in $L^{6-\epsilon}(\Omega)$ for all
$\epsilon>0$, whence, taking into account the growth condition on
$\condu$, that
\[
\condu(\teta_\nu) \to \condu(\teta) \qquad \text{in } %L^{3+\alpha/\kappa -\epsilon }(\Omega)
 L^\gamma(\Omega)  \quad \text{with } \gamma =
 \frac{(6-\epsilon)(\kappa+\alpha)}{2\kappa}  \quad \text{for all }
\epsilon>0.
\]
This allows us to pass to the limit in the term $\condu(\teta_\nu)
\nabla \teta_\nu$, tested  against   $v \in W^{1,s}(\Omega)$ for
some sufficiently  large \EEE $s>0$.
 All in all, we infer that
$(c,\mu,z,\teta,\uu,\chi)$  solves \EEE system \eqref{PDE-discrete}, with
\eqref{eqn:discr2} and \eqref{eqn:discr3} in $W^{1,p}(\Omega)'$, and
with the discrete heat equation \eqref{eqn:discr4} understood in
$W^{1,s}(\Omega)'$.

In the next step, we will address    enhancements of the \EEE regularities of
$\uu$ and $\mu$.

As a by-product we will obtain the discrete heat equation \eqref{eqn:discr4}
understood in the $H^1(\Omega)'$-sense.

%\subsubsection{\textbf{Step 4: Improvement of the discrete damage equation and Cahn-Hilliard system}}

\subsubsection{\textbf{Step 4:  ${\bf H^2}$-regularity of ${\bf \ub^k}$ and 
${ \boldsymbol \mu^{\bf k}}$ \EEE and conclusion of the proof of Prop.\ \ref{prop:exist-discr}}}
\noindent

To complete the \underline{\textbf{proof of Proposition \ref{prop:exist-discr}}}
we have to improve the regularity of $\ub^k$ and $\mu^k$.
This is achieved by transforming the corresponding equations
in a way that enables us to apply standard elliptic regularity results.
\begin{lemma}
\label{lemma:4.16}
	We get $\mu^k\in \Hn(\Omega)$ and $\ub^k\in H^2(\Omega; \RR^d)\EEE$
	for the functions obtained in Lemma \ref{lemma:discrConvergence-nu}.
\end{lemma}
\begin{proof}
	We will use an iteration argument as in \cite{RocRos14,hr} (see also \cite{bm})
	and sketch the proof for the case $d=3$,  since the calculations for $d=2$ are completely
	analogous. \EEE
	
	We already know that $\mu^k\in H^1(\Omega)$ satisfies the elliptic equation
	$$
		\int_\Omega m(c^{k-1},z^{k-1})\nabla\mu^k\cdot\nabla w\dx=\int_\Omega -D_k(c)w\dx
		\qquad\text{for all }w\in H^1(\Omega).
  $$
  Substituting $w=\frac{\zeta}{m(c^{k-1},z^{k-1})}\in H^1(\Omega)$ for an arbitrarily
  chosen  test-function $\zeta\in H^1(\Omega)$ yields
  \begin{align*}
		\int_\Omega \nabla\mu^k\cdot\nabla\zeta\dx
		=\int_\Omega\Big(\frac{-D_k(c)}{m(c^{k-1},z^{k-1})}+\frac{m_{,c}(c^{k-1},z^{k-1})\nabla c^{k-1}+m_{,z}(c^{k-1},z^{k-1})\nabla z^{k-1}}{ m(c^{k-1},z^{k-1})\EEE}\cdot\nabla\mu^k\Big)\zeta\dx
  \end{align*}
  valid for all $\zeta\in H^1(\Omega)$.
  Note that, due to Hypothesis (II) and
   the fact that \EEE
   $c^{k-1},z^{k-1}\in W^{1,p}(\Omega)$
  and $\nabla \mu^k\in L^2(\Omega;\R^d)$, the function in the bracket on the right-hand
  side is in $L^{2p/(2+p)}(\Omega)$.
  Applying  a higher elliptic regularity result for homogeneous Neumann problems
  with $L^{2p/(2+p)}(\Omega)$-right-hand side proves $\mu^k\in W^{2,2p/(2+p)}(\Omega)$
  and thus
  $\nabla\mu^k\in L^{6p/(6+p)}(\Omega;\R^d)$.
  Due to $p>3$ we end up with $\mu^k\in \Hn(\Omega)$ after repeating this procedure
  finitely many times (cf. \cite[Proof of Lemma 4.1]{hr}). 
  
  The proof for obtaining $\ub^k\in H^2(\Omega;\R^d)$
  from the elliptic equation \eqref{discr-syst-appr-u} in $H_0^1(\Omega;\R^n)'$
  works as in \cite[Proof of Lemma 4.4 - Step 6]{RocRos14}  (cf.\ also \cite{hr}), \EEE
  with the exception that one needs to take the Dirichlet data
  $\db^k\in H^2(\Omega;\R^d)$ into account.   This is the very point where   we need to assume that $\VV=\omega \CC$ for some $\omega>0$ (cf.~\eqref{eqn:assbV}).\EEE
\end{proof}

The enhanced regularity for $\ub^k$ yields, by a comparison
argument in \eqref{eqn:discr4}, that \eqref{eqn:discr4}
not only holds in $W^{1,s}(\Omega)'$ for large $s>1$ but even in
$H^1(\Omega)'$.

Finally, we end up with a quintuple
$\{(\ck, \muk,\zk,\tk,\uk)\}_{k=1}^{K_\tau}\subseteq W^{1,p}(\Omega)\times \Hn(\Omega)\times W^{1,p}(\Omega)\times H^1(\Omega)\times  H^2(\Omega;\R^d)$
satisfying the assertion stated in Proposition \ref{prop:exist-discr}.
\par
 This concludes the proof. \EEE
\QED

%%%%%%%%  ************************************************************************
%%%%%%%%
%%%%%%%%	TIME-DISCRETE SYSTEM: ENERGY AND ENTROPY INEQUALITIES
%%%%%%%%
%%%%%%%%  ************************************************************************
\subsection{Discrete energy and entropy inequalities}
\label{ss:5.3}

We introduce the left-continuous and  right-continuous piecewise constant,
and the piecewise linear interpolants  for a given sequence 
$\{\mathfrak{h}_\tau^k\}_{k={ 0}}^{K_\tau}$
 on the nodes $\{t_\tau^k\}_{k=0}^{K_\tau}$ (see \ref{time-nodes}) 
by
\begin{align*}
	\left.
	\begin{array}{llll}
		& \pwc  {\mathfrak{h}}{\tau}: (0,T) \to B  & \text{defined by}  &
		\pwc {\mathfrak{h}}{\tau}(t): = \mathfrak{h}_\tau^k
		\\
		& \upwc  {\mathfrak{h}}{\tau}: (0,T) \to B  & \text{defined by}  &
		\upwc {\mathfrak{h}}{\tau}(t) := \mathfrak{h}_\tau^{k-1}
		\\
		 &
		\pwl  {\mathfrak{h}}{\tau}: (0,T) \to B  & \text{defined by} &
		 \pwl {\mathfrak{h}}{\tau}(t):
		=\frac{t-t_\tau^{k-1}}{\tau} \mathfrak{h}_\tau^k +
		\frac{t_\tau^k-t}{\tau}\mathfrak{h}_\tau^{k-1}
	\end{array}
	\right\}
 \qquad \text{for $t \in (t_\tau^{k-1}, t_\tau^k]$.}
\end{align*}

Furthermore, we   denote by  $\pwc{\mathsf{t}}{\tau}$ and by
$\upwc{\mathsf{t}}{\tau}$ the left-continuous and right-continuous
piecewise constant interpolants associated with the partition, i.e.
 $\pwc{\mathsf{t}}{\tau}(t) := t_\tau^k$ if $t_\tau^{k-1}<t \leq t_\tau^k $
and $\upwc{\mathsf{t}}{\tau}(t):= t_\tau^{k-1}$ if $t_\tau^{k-1}
\leq t < t_\tau^k $. Clearly, for every $t \in [0,T]$ we have
$\pwc{\mathsf{t}}{\tau}(t) \downarrow t$ and
$\upwc{\mathsf{t}}{\tau}(t) \uparrow t$ as $\tau\downarrow 0$.

\begin{proposition}
\label{prop:energyEntropyIneq}
	Let the assumptions of Proposition \ref{prop:exist-discr} be satisfied. 
	Then the
	time-discrete solutions $\{(\ck,\muk,\zk,\tk,\uk)\}_{k=1}^{K_\tau}$
	to Problem \ref{def:time-discrete} fulfill for all $0\leq s\leq t\leq T$
	\begin{itemize}
		\item[(i)]
			the discrete entropy inequality
	    \begin{equation}
        \label{discr-entropy-ineq}
        \begin{aligned}
          &\begin{aligned}
          	\int_{\ttau(s)}^{\ttau(t)} \int_\Omega (\log(\ul\teta_\tau) + \ul c_\tau+\ul z_\tau)\partial_t\varphi_\tau  \dd x \dd r
          		&-\rho \int_{\ttau(s)}^{\ttau(t)} \int_\Omega \dive(\partial_t\uu_\tau)\ol\varphi_\tau  \dd x \dd r\\
          		&-\int_{\ttau(s)}^{\ttau(t)} \int_\Omega  \condu(\ol\teta_\tau) \nabla \log(\ol\teta_\tau) \cdot \nabla\ol\varphi_\tau  \dd x \dd r
  	      \end{aligned}\\
          &\begin{aligned}
            \leq
              \int_\Omega (\log(\ol\teta_\tau(t))+\ol c_\tau(t)+\ol z_\tau(t)){\ol\varphi_\tau(t)} \dd x
              &-\int_\Omega (\log(\ol\teta_\tau(s))+\ol c_\tau (s)+\ol z_\tau(s)){\ol\varphi_\tau(s)} \dd x\\
              &-\int_{\ttau(s)}^{\ttau(t)} \int_\Omega \condu(\ol \teta_\tau)|\nabla\log(\ol\teta_\tau)|^2\ol\varphi_\tau\dd x \dd r
          \end{aligned}\\
  	      &\quad-\int_{\ttau(s)}^{\ttau(t)}  \int_\Omega \left(\ol g_\tau +|\partial_t c_\tau|^2+ |\partial_t z_\tau|^2
         	+a(\ul c_\tau,\ul z_\tau) \eps(\partial_t\uu_\tau):\vism \eps(\partial_t\uu_\tau)
          +m(\ul c_\tau,\ul z_\tau)|\nabla \ol \mu_\tau|^2\right)\frac{\ol\varphi_\tau}{\ol\teta_\tau} \dd x \dd r\\
	        &\quad-\int_{\ttau(s)}^{\ttau(t)} \int_{\partial\Omega}\ol h_\tau \frac{\ol\varphi_\tau}{\ol\teta_\tau}  \dd S \dd r,
        \end{aligned}
      \end{equation}
      for all $\varphi \in \mathrm{C}^0 ([0,T]; W^{1,d+\epsilon}(\Omega))  \cap H^1 (0,T; L^{({d^\star})'}(\Omega))$
      for some $\epsilon>0$,  with $\varphi \geq 0$;
		\item[(ii)]
			the discrete total energy inequality
	    \begin{align}
        \label{discr-energy-ineq}
        \begin{aligned}
	        &\mathscr{E}_\omega(\ol c_\tau(t),\ol z_\tau(t),\ol\teta_\tau(t),\ol\ub_\tau(t),\ol \vb_\tau(t))\\
					&\qquad\leq\mathscr{E}_\omega (\ol c_\tau(s),\ol z_\tau(s),\ol\teta_\tau(s),\ol\ub_\tau(s),\ol \vb_\tau(s))\\
		      	&\qquad\quad+\int_{\ttau(s)}^{\ttau(t)}\int_\Omega \ol g_\tau\dx\dr +\int_{\ttau(s)}^{\ttau(t)}\int_\Omega\ol{\bold f}_\tau \cdot\ol \vb_\tau\dx\dr
		      	+\int_{\ttau(s)}^{\ttau(t)}\int_{\partial\Omega}\ol h_\tau\dS\dr\\
		      	&\qquad\quad+\int_{\ttau(s)}^{\ttau(t)}\int_{\partial\Omega}(\sigmab_\tau { \bf  n \EEE})\cdot\partial_t\db_\tau\dd S\dr
        \end{aligned}
	    \end{align}
	    with the discrete stress tensor
	    \begin{align*}
	    	\sigmab_\tau:=a(\upwc c{\tau},\upwc z{\tau})\VV\e(\partial_t \uu_\tau)
	    		+W_{,\e}^\omega(\pwc c{\tau},\e(\pwc \uu{\tau}),\pwc z{\tau})
	    		-\rho\pwc\teta{\tau}\mathds 1.
	    \end{align*}
	\end{itemize}
\end{proposition}
\begin{proof}
	$\,$
	\begin{itemize}
		\item[To (i):]
			The proof is based on \cite[Proof of Proposition 4.8]{RocRos14}.
			Testing the time-discrete heat equation \eqref{eqn:discr4} for time step $k$ 
			with $\frac{\varphi_\tau^k}{\teta_\tau^k}\in H^1(\Omega)$ shows
      \begin{align*}
      	&\int_\Omega\bigg(g_\tau^k+|\Dt(c)|^2+|\Dt(z)|^2+m(\ckk,\zkk)|\nabla\muk|^2\bigg)\frac{\varphi_\tau^k}{\tk}\dx\\
      		&+\int_\Omega a(\ckk,\zkk)\e(\Dt(\ub)):\vism\e(\Dt(\ub))\frac{\varphi_\tau^k}{\tk}\dx
      		+\int_{\partial\Omega}\hk\frac{\varphi_\tau^k}{\tk}\dS\\
        &\leq\int_\Omega\bigg(\condu(\tk)\nabla\tk\cdot\nabla\frac{\varphi_\tau^k}{\tk}
        	+\bigg(\frac{1}{\tau}\big(\log(\tk)-\log(\tkk)\big)+\Dt(c)+\Dt(z)+\rho\dive(\Dt(\ub))\bigg)\varphi_\tau^k\dx
      \end{align*}
      by using the concavity estimate
      $$
      	\frac{\tk-\tkk}{\tk}
      	\leq\log(\tk)-\log(\tkk).
      $$
      Summing over   $k=\frac{\ttau(s)}\tau+1,\ldots,\frac{\ttau(t)}\tau$  \EEE and using discrete by-part-integration
      proves \eqref{discr-entropy-ineq}.
		\item[To (ii):]
			The total energy inequality
			is inherited by the incremental energy inequality \eqref{discr-total-ineq}
			of the $(M,\nu)$-regularized system in Lemma \ref{l:energy-est}.
			Indeed, let $0\leq s\leq t\leq T$.
			Passing to the limits $M\to\infty$ and $\nu\downarrow 0$
			in \eqref{discr-total-ineq} by means of lower   semicontinuity \EEE arguments
			and then summing over $j=\frac{\ttau(s)}{\tau}+1,\ldots,\frac{\ttau(t)}{\tau}$ yields
	    \begin{align*}
        &\mathscr{E}_\omega(\pwc c{\tau}(t),\pwc z{\tau}(t),\pwc \teta{\tau}(t),\pwc \uu{\tau}(t),\pwc \vb{\tau}(t))\\
	      &\qquad\leq\mathscr{E}_\omega (\pwc c{\tau}(s),\pwc z{\tau}(s),\pwc \teta{\tau}(s),\pwc \uu{\tau}(s),\pwc \vb{\tau}(s))
	      	+\int_{\ttau(s)}^{\ttau(t)}\bigg(\int_\Omega \pwc g{\tau}\dx + \int_{\partial\Omega} \pwc h{\tau} \dS +\int_\Omega \pwc{\bold f}{\tau} \cdot\pwc\vb{\tau}\dx\bigg)\dr\\
					&\left.
					\begin{aligned}
						&\qquad\quad+\int_{\ttau(s)}^{\ttau(t)}\io \partial_t \vb_\tau\cdot \partial_t \db_\tau\dx\dr
						+\int_{\ttau(s)}^{\ttau(t)}\io a(\upwc c{\tau},\upwc z{\tau})\VV\e(\pwc \vb{\tau}):\e(\partial_t \db_\tau)\dx\dr\\
						&\qquad\quad+\int_{\ttau(s)}^{\ttau(t)}\io W_{,\e}^\omega(\pwc c{\tau},\e(\pwc\uu{\tau}),\pwc z{\tau}):\e(\partial_t\db_\tau)\dx\dr
						-\int_{\ttau(s)}^{\ttau(t)}\io\rho\pwc\teta{\tau}\dive(\partial_t\db_\tau)\dx\dr\\
						&\qquad\quad-\int_{\ttau(s)}^{\ttau(t)}\int_\Omega \pwc{\bold f}{\tau} \cdot \partial_t\db_\tau\dx\dr.
					\end{aligned}
					\right\}
					=:I_1
	    \end{align*}
	    Finally, integration by parts in space and using \eqref{eqn:discr5} shows
	    \begin{align*}
					I_1={}&\int_{\ttau(s)}^{\ttau(t)}\io\Big(\underbrace{\partial_t \vb_\tau
						-\dive\Big(a(\upwc c{\tau},\upwc z{\tau})\VV\e(\pwc \vb{\tau})
						+W_{,\e}^\omega(\pwc c{\tau},\e(\pwc\uu{\tau}),\pwc z{\tau})
						-\rho\pwc\teta{\tau}\mathds 1\Big)-\pwc{\bold f}{\tau}}_{=0}\Big)\cdot \partial_t \db_\tau\dx\dr\\
						&+\int_{\ttau(s)}^{\ttau(t)}\int_{\partial\Omega}
						\Big(a(\upwc c{\tau},\upwc z{\tau})\VV\e(\pwc \vb{\tau})
						+W_{,\e}^\omega(\pwc c{\tau},\e(\pwc\uu{\tau}),\pwc z{\tau})
						-\rho\pwc\teta{\tau}\mathds
                                                1\Big){ \bf  n \EEE} \cdot\partial_t \db_\tau\dd S\dr\\
					={}&\int_{\ttau(s)}^{\ttau(t)}\int_{\partial\Omega}(\sigmab_\tau 
{ \bf n } \EEE)\cdot\partial_t\db_\tau\dd S\dr.
	    \end{align*}
	\end{itemize}
\end{proof}

%%%%%%%%  ************************************************************************
%%%%%%%%
%%%%%%%%	TIME-DISCRETE SYSTEM: A PRIORI ESTIMATES
%%%%%%%%
%%%%%%%%  ************************************************************************
\subsection{A priori estimates}
\label{ss:5.4}
The aim of this section is to customize the a priori  estimates which we have
 developed \EEE in Section \ref{s:4} \EEE to the time-discrete setting described in
Problem \ref{def:time-discrete},  for a time-discrete solution  
$(\ol c_\tau,\ul c_\tau, c_\tau,\ol\mu_\tau,\ol z_\tau,\ul z_\tau, z_\tau,
\ol\teta_\tau, \ul\teta_\tau, \teta_\tau,
\ol\ub_\tau,\ul\ub_\tau,\ol\vb_\tau,\vb_\tau)$ (recall that $\mathbf{v}^k = D_k (\mathbf{u})$ for all $k \in \{1, \ldots, K_\tau\}$. 
 Let us mention in advance that, in  this  \EEE time-discrete setting we
are only able to estimate  (cf.\ \eqref{thetaBound1}  below) the \EEE
supremum of the total variation 
	  	$\langle\log(\ol\teta_\tau),\varphi\rangle_{W^{1,d+\epsilon}}$
	  	over all test-functions $\varphi\in W^{1,d+\epsilon}(\Omega)$
	  	with $\|\varphi\|_{W^{1,d+\epsilon}(\Omega)}\leq 1$,
			which is a slightly weaker result than the \textbf{Seventh a priori estimate}
			in Section \ref{s:4} however strong enough to apply  the compactness
			result proved  in  \EEE \cite[Theorem A.5]{RocRos14}.
\begin{proposition}
\label{prop:aprio-discr}
	Let the assumptions of Proposition \ref{prop:exist-discr} be satisfied. 
	Then the
	time-discrete solutions $\{(\ck,\muk,\zk,\tk,\uk)\}_{k=1}^{K_\tau}$
	to Problem \ref{def:time-discrete} fulfill the following a priori estimates
	uniformly in $\omega>0$ and $\tau>0$:
	\begin{align}
		&\|\ol c_\tau\|_{L^\infty(0,T;W^{1,p}(\Omega))}+\|\ul c_\tau\|_{L^\infty(0,T;W^{1,p}(\Omega))}\leq C,
			\label{cBound1}\\
		&\|c_\tau\|_{H^1(0,T;L^2(\Omega))\cap L^\infty(0,T;W^{1,p}(\Omega))}\leq C,
			\label{cBound2}\\
		&\|\Delta_p(\ol c_\tau)\|_{L^2(0,T;L^{2}(\Omega))}\leq C,
			\label{cBound3}\\
		&\|\ol\eta_\tau\|_{L^2(0,T;L^{2}(\Omega))}\leq C\qquad
			\text{with }\ol\eta_\tau:=\beta_\omega(\ol c_\tau),
			\label{etaBound}\\
		&\|\ol\mu_\tau\|_{L^2(0,T;H^2(\Omega))}\leq C,
			\label{muBound}\\
		&\|\ol z_\tau\|_{L^\infty(0,T;W^{1,p}(\Omega))}+\|\ul z_\tau\|_{L^\infty(0,T;W^{1,p}(\Omega))}\leq C,
			\label{zBound1}\\
		&\|z_\tau\|_{H^1(0,T;L^2(\Omega))\cap L^\infty(0,T;W^{1,p}(\Omega))}\leq C,
			\label{zBound2}\\
		&\|\ol\teta_\tau\|_{L^2(0,T;H^1(\Omega))\cap L^\infty(0,T;L^1(\Omega))}\leq C,
			\label{thetaBound1}\\
		&\big\|(\ol\teta_\tau)^{\frac{\kappa+\alpha}{2}}\big\|_{L^2(0,T;H^1(\Omega))}\leq C_\alpha
			\text{ for all }\alpha\in(0,1),
			\label{thetaBound2}\\
		&\|\log(\ol\teta_\tau)\|_{L^2(0,T;H^1(\Omega))}\leq C,
			\label{thetaBound3}\\
		&\|\ol\ub_\tau\|_{L^\infty(0,T;H^2(\Omega;\R^d))}+\|\ul\ub_\tau\|_{L^\infty(0,T;H^2(\Omega;\R^d))}\leq C,
			\label{uBound1}\\
		&\|\ub_\tau\|_{H^1(0,T;H^2(\Omega;\R^d))\cap W^{1,\infty}(0,T;H^1(\Omega;\R^d))}\leq C,
			\label{uBound2}\\
		&\|\vb_\tau\|_{L^2(0,T;H^2(\Omega;\R^d))\cap H^1(0,T;L^2(\Omega;\R^d))}\leq C
			\label{vBound}
	\end{align}
	as well as
	\begin{align}
	\label{thetaBound4}
		&\sup_{\varphi\in W^{1,d+\epsilon}(\Omega),  \|\varphi\|_{W^{1,d+\epsilon}(\Omega)} \EEE \leq 1}
			\mathrm{Var}\big(\langle\log(\ol\teta_\tau),\varphi\rangle_{W^{1,d+\epsilon}};[0,T]\big)
			\leq C_\epsilon
			\quad\text{for all }\epsilon>0.
	\end{align}
	Under the additional assumption \eqref{range-k-admissible}
%	$\kappa\in(1,5/3)$ for $d=3$ and $\kappa\in(1,2)$ for $d=2$
	we also have
	\begin{align}
	\label{thetaBoundAdd}
		\|\teta_\tau\|_{ \mathrm{BV}\EEE([0,T];W^{2,d+\epsilon}(\Omega)')}\leq C_\epsilon
			\quad\text{for all }\epsilon>0.
	\end{align}
\end{proposition}
\begin{proof}
	The proof mainly follows the lines in Section \ref{s:4}.
	Besides this, the estimates for the time-discrete variables $z_\tau$, $\teta_\tau$ and $\ub_\tau$ are based
	on \cite[Proof of Proposition 4.10]{RocRos14}.
	To avoid repetition we will refer to the estimates in Section \ref{s:4}
	when necessary.
	\begin{itemize}
		\item[(i)]
			The time-discrete total energy inequality from
			Proposition \ref{prop:energyEntropyIneq} (ii)
			implies the following estimates
			(see \textbf{First a priori estimate}):
			\begin{align*}
				\|\ol c_\tau\|_{L^\infty(0,T;W^{1,p}(\Omega))}
				+\|\nabla\ol z_\tau\|_{L^\infty(0,T;W^{1,p}(\Omega;\R^d))}
				+\|\ol\teta_\tau\|_{L^\infty(0,T;L^1(\Omega))}
				  +\|\vb_\tau\|_{L^\infty(0,T;L^2(\Omega;\R^d))} \EEE
				\leq C.
			\end{align*}
		\item[(ii)]
			The \textbf{Second a priori estimate} is performed by
			testing the time-discrete heat equation \eqref{eqn:discr4} with $F'(\tk)=(\tk)^{\alpha-1}$
			with $\alpha\in(0,1)$ and the concave function $F(\teta):=\teta^\alpha/\alpha$,
			we obtain
			\begin{align*}
				&\int_\Omega\Big(\gk+|\Dt(c)|^2+|\Dt(z)|^2+m(\ckk,\zkk)|\nabla\muk|^2\Big)F'(\tk)\dx\\
				&+\int_\Omega a(\ckk,\zkk)\e(\Dt(\ub)):\vism\e(\Dt(\ub))F'(\tk)\dx
					+\int_{\partial\Omega}\hk F'(\tk)\dS\\
				&\leq\int_\Omega\frac{F(\tk)-F(\tkk)}{\tau}+\condu(\tk)\nabla\tk\cdot\nabla(F'(\tk))\dx\\
				&\qquad+\int_\Omega\Big(\Dt(c)+\Dt(z)+\rho\dive(\Dt(\ub))\Big)\tk F'(\tk)\dx
			\end{align*}
			by using the  concavity \EEE estimate $(\tk-\tkk)F'(\tk)\leq F(\tk)-F(\tkk)$.
			Multplication by $\tau$ and summing over $k=1,\ldots,\ttau(t)/\tau$ shows
			for every $t\in(0,T]$ the precise time-discrete analogon to
			\eqref{eqn:secondEstPre}.
			With the same calculations as in Section \ref{s:4} we end up with
			\begin{align*}
				\|\ol\teta_\tau\|_{L^2(0,T;H^1(\Omega))}\leq C,\qquad
				\|(\ol\teta_\tau)^{(\kappa+\alpha)/2}\|_{L^2(0,T;H^1(\Omega))}\leq C_\alpha.
			\end{align*}
		\item[(iii)]
			By testing the time-discrete heat equation \eqref{eqn:discr4} with $\tau$, integrating over $\Omega$, summing over $k$ and
			subtracting the result from the total energy inequality \eqref{discr-energy-ineq}
			we obtain the \textbf{Third a priori estimate}:
			\begin{align*}
				\|\partial_t c_\tau\|_{L^2(Q)} + \|\nabla \ol\mu_\tau \|_{L^2(Q;\R^d)} + \|\partial_t z_\tau\|_{L^2(Q)}+\|\partial_t \uu_\tau\|_{L^2(0,T; H^1(\Omega;\RR^d))} \leq C
			\end{align*}
			as well as 
			\begin{align*}
				\|\ol z_\tau \|_{L^\infty (0,T;W^{1,p}(\Omega))}
				+\|\ol \ub_\tau \|_{L^\infty (0,T;H^{1}(\Omega;\R^d))}
				\leq C.
			\end{align*}
		\item[(iv)]
			The \textbf{Fourth a priori estimate} is obtained by
			testing the time-discrete force balance equation \eqref{eqn:discr5}
			by $-\tau\dive(\VV\e(\vk))$.
			The calculation in Section \ref{s:4} carry over to the time-discrete setting.
			However, 
			 let us point out that
			the discrete analogue of  \eqref{est-added-0}  is given by \EEE the convexity estimate
			\begin{align*}
				-\int_0^{\ttau(t)} \io  \partial_t \vb_\tau\cdot \mathbf{\dive} (\vism\eps(\ol\vb_\tau)) \dd x \dd s
				\geq{}&
					-\int_0^{\ttau(t)}
                                        \int_{\partial\Omega}\partial_t
                                        \vb_\tau\cdot (\vism\eps(\ol\vb_\tau) { \bf  n \EEE}) \dd S \dd s\\
					&+\io \frac12 \eps(\ol\vb_\tau(t)):\vism \eps(\ol\vb_\tau(t)) \dd x
					-\io \frac12 \eps(\vb^0):\vism \eps(\vb^0) \dd x.
			\end{align*}
	    With analogous calculations we arrive at
	    \begin{align*}
	    	\qquad&\|\ub_\tau\|_{H^1(0,T;H^2(\Omega;\R^d))}
	    		+\|\ol\ub_\tau\|_{L^\infty(0,T;H^2(\Omega;\R^d))}\leq
                        C, \EEE\\
	    	&\|\vb_\tau\|_{H^1(0,T;L^2(\Omega;\R^d))\cap L^\infty(0,T;H^1(\Omega;\R^d))\cap L^2(0,T;H^2(\Omega;\R^d))}
	    		+\|\ol\vb_\tau\|_{L^\infty(0,T;H^1(\Omega;\R^d))\cap L^2(0,T;H^2(\Omega;\R^d))}\leq C.
	    \end{align*}
	  \item[(v)]
	  	For the \textbf{Fifth a priori estimate} we test \eqref{eqn:discr2}
	  	with $\ck-\mathfrak{m}_0$ where $\mathfrak{m}_0:=\dashint_\Omega c^0\dx$.
	  	With exactly the same calculations as in Section \ref{s:4} we find
	    \begin{align*}
	    	\|\ol\mu_\tau\|_{L^2(0,T;H^1(\Omega))}\leq C.
	    \end{align*}
	  \item[(vi)]
	   	A comparison in \eqref{eqn:discr2} as done in
	   	the \textbf{Sixth a priori estimate} gives
	    \begin{align*}
		   	\| \Delta_p(\ol c_\tau)\|_{L^2(0,T; L^2(\Omega))}
		   	+ \|\ol\eta_\tau \|_{L^2(0,T; L^2(\Omega))} \leq C.
	    \end{align*}
	  \item[(vii)]
			 Estimate \eqref{thetaBound1} \EEE can be shown by utilizing the calculations
	  	in \cite[Proof of Proposition 4.10 - Sixth estimate]{RocRos14}
	  	and additionally noticing that $\{\ol c_\tau\}_{\tau>0}$
	  	is bounded in $ \mathrm{BV}\EEE([0,T];L^2(\Omega))$ due to the \textit{Third estimate}.
	  	We thus  obtain \eqref{thetaBound4}.
	  \item[(viii)]
	  	The \textbf{Eighth a priori estimate} works as in Section \ref{s:4}
	  	and yields \eqref{thetaBoundAdd}.
	  \item[(ix)]
	  	The \textbf{Ninth a priori estimate} works as in Section \ref{s:4}
	  	and yields \eqref{muBound}.
	\end{itemize}
\end{proof}
\begin{remark}
	We observe that \eqref{thetaBound4} implies the uniform bound
	\begin{align*}
		\|\log(\ol\teta_\tau)\|_{L^\infty(0,T;W^{1,d+\epsilon}(\Omega))}\leq C_\epsilon.
	\end{align*}
	Moreover, by interpolation we infer from \eqref{thetaBound1} that (see \eqref{estetainterp})
	\begin{align*}
		\|\ol\teta_\tau\|_{L^h(Q)}\leq C
	\end{align*}
	with $h=8/3$ for $d=3$ and $h=3$ for $d=2$.
\end{remark}

%%%%%%%%  ************************************************************************
%%%%%%%%
%%%%%%%%	PROOF OF MAIN RESULT
%%%%%%%%
%%%%%%%%  ************************************************************************
\section{\bf Proof of Theorem \ref{thm:1}}
\label{s:6}
In this last section we are going to perform the limit passages as 
$\tau\downarrow 0$ and $\omega\downarrow 0$ in the time-discrete system   \eqref{PDE-discrete}, for which the existence of solutions was
proved in Proposition \ref{prop:exist-discr}.  \EEE
 This will  lead  us \EEE to prove Theorem \ref{thm:1}.
\subsection{Compactness}
\label{ss:6.1}
We  shall adopt the notation from the previous section.
In particular for fixed $\omega>0$ we let
$(\ol c_\tau,\ul c_\tau, c_\tau,\ol\mu_\tau,\ol z_\tau,\ul z_\tau, z_\tau,
\ol\teta_\tau, \ul\teta_\tau, \teta_\tau,
\ol\ub_\tau,\ul\ub_\tau,\ol\vb_\tau,\vb_\tau)$
be a time-discrete solution on an equi-distant partition of $[0,T]$ with fineness $\tau>0$
according to Proposition \ref{prop:exist-discr}.
\begin{lemma}
\label{lemma:discr-conv}
	Let the assumptions from Proposition \ref{prop:exist-discr} be satisfied
	and $\omega>0$ be fixed.
	Then there exists a quintuple $(c,\mu,z,\teta,\ub)$ satisfying
	\eqref{reg-c}--\eqref{reg-u} such that   along a  (not relabeled) subsequence, as $\tau \down 0$, the following convergences hold: \EEE
	\begin{align}
		&\pwc c{\tau},\,\upwc c{\tau} \weakstarlim c
			&&\text{ weakly-star in }L^\infty(0,T;W^{1,p}(\Omega)),
			\label{cConv1}\\
		&c_\tau \weakstarlim c
			&&\text{ weakly-star in }L^\infty(0,T;W^{1,p}(\Omega))\cap H^{1}(0,T;L^2(\Omega)),
			\label{cConv2}\\
		&\Delta_p(\pwc c{\tau}) \weaklim \Delta_p(c)
			&&\text{ weakly in }L^2(0,T;L^2(\Omega)),
			\label{cConv3}\\
		&\pwc c{\tau},\,\upwc c{\tau}\to c
			&&\text{ strongly in }L^s (0,T; W^{1,p}(\Omega))\text{ for all }s\in[1,\infty),
			\label{cConv4}\\
		&\pwc \eta{\tau} \weaklim \eta
			&&\text{ weakly in }L^2(0,T;L^2(\Omega))\text{ with }\eta=\beta_\omega(c)\text{ a.e. in }Q,
			\label{etaConv}\\
		&\pwc \mu{\tau} \weaklim \mu
			&&\text{ weakly in }L^2(0,T;\Hn(\Omega)),
			\label{muConv}\\
		&\pwc z{\tau},\,\upwc z{\tau} \weakstarlim z
			&&\text{ weakly-star in }L^\infty (0,T; W^{1,p}(\Omega)),
			\label{zConv1}\\
		&z_\tau \weakstarlim z
			&&\text{ weakly-star in }L^\infty (0,T; W^{1,p}(\Omega)) \cap H^1 (0,T;L^2(\Omega)),
			\label{zConv2}\\
		&\pwc z{\tau},\,\upwc z{\tau}\to z
			&&\text{ strongly in }L^\infty (0,T; X)\text{ for all $X$ such that }W^{1,p}(\Omega) \Subset X\subseteq L^2(\Omega),
			\label{zConv3}\\
%	\end{align}
%	\begin{align}
		&z_\tau \to z
			&&\text{ strongly in }\mathrm{C}^0([0,T]; X)\text{ for all $X$ such that }W^{1,p}(\Omega) \Subset X\subseteq L^2(\Omega),
			\label{zConv4}\\
	 	&\pwc \teta{\tau}\weakto \teta
	 		&&\text{ weakly in }L^2 (0,T; H^1(\Omega)),
			\label{thetaConv1}\\
		&\log(\pwc\teta{\tau})  \weakstarlim \log(\teta)
			&&\text{ weakly-star in } L^2 (0,T; H^1(\Omega))  \cap  L^\infty (0,T; W^{1,d+\epsilon}(\Omega)')   \quad \text{for every } \epsilon>0,
			\label{thetaConv2}\\
		&\log(\pwc\teta{\tau})  \to  \log(\teta)
			&&\text{ strongly in	}L^2(0,T;L^s(\Omega)) \text{ for all $ s \in [1,6)$ if $d=3$, and all $s\in [1,\infty) $ if $d=2$,}
			\label{thetaConv3}\\
		&\log(\pwc\teta{\tau}(t)) \weakto \log(\teta(t))
			&&\text{ weakly in $H^1(\Omega)$ for almost all $t \in (0,T)$}
			\label{thetaConv4}\\
			&&&\text{ (the chosen subsequence for $\tau\downarrow 0$ does not depend on $t$)},\notag\\
		&\pwc\teta{\tau}\to \teta
			&&\text{ strongly in } L^q(\Omega\times (0,T))
			\text{ for all }q\in [1,8/3)\text{ for }d=3\text{ and all }q\in [1, 3)\text{ if }d=2,\EEE
			\label{thetaConv5}
	\end{align}
	\begin{align}
		&\pwc \uu{\tau},\upwc \uu{\tau} \weakstarlim \uu  
			&&\text{ weakly-star in }L^\infty(0,T;H^2(\Omega;\R^d)),
			\label{uConv1}\\
		&\pwl \uu{\tau} \weakstarlim \uu  
			&&\text{ weakly-star in }H^1(0,T;H^2(\Omega;\R^d))\cap	W^{1,\infty}(0,T;H^1(\Omega;\R^d)),
			\label{uConv2}\\
		&\pwc \uu{\tau},\, \upwc \uu{\tau} \to \uu
			&&\text{ strongly in }L^\infty(0,T;X)\text{ for all $X$ such that }H^{2}(\Omega;\R^d) \Subset X\subseteq L^2(\Omega;\R^d),
			\label{uConv3}\\
		&\pwl \uu{\tau} \to \uu
			&&\text{ strongly in }\mathrm{C}^0([0,T]; X \EEE)\text{ for all $X$ such that }H^{2}(\Omega;\R^d) \Subset X\subseteq L^2(\Omega;\R^d),
			\label{uConv4}\\
		&\pwc \vb{\tau} \weakto \uu_{t}
			&&\text{ weakly in }L^2(0,T;H^2(\Omega;\R^d)),
			\label{vConv1}\\
		&\vb_\tau \weaklim \uu_t
			&&\text{ weakly in }H^1(0,T;L^2(\Omega;\R^d))\cap L^2(0,T;H^2(\Omega;\R^d)).
			\label{vConv2}
	\end{align}
	Under the additional assumption \eqref{range-k-admissible}
	we also have for all $\epsilon>0$ that $\teta\in  \mathrm{BV} \EEE([0,T]; W^{2,d+\epsilon}(\Omega)')$ and
	\begin{align}
		&\pwc \teta{\tau}\to \teta
			&&\text{ strongly in }L^2 (0,T; Y)\text{ for all $Y$ such that }H^{1}(\Omega) \Subset Y\subset W^{2,d+\epsilon}(\Omega)',
			\label{tetaConvAdd1}\\
		&\pwc \teta{\tau}(t)\to \teta(t)
			&&\text{ strongly in }W^{2,d+\epsilon}(\Omega)'\text{ for all }t\in[0,T].
			\label{tetaConvAdd2}
	\end{align}
\end{lemma}
\begin{proof}
	We immediately obtain \eqref{cConv1}, \eqref{cConv2},
	\eqref{muConv}, \eqref{zConv1}, \eqref{zConv2}, \eqref{thetaConv1}, \eqref{uConv1}, \eqref{uConv2}, \eqref{vConv1} and \eqref{vConv2}
	from the estimates \eqref{cBound1}--\eqref{vBound} in Proposition \ref{prop:aprio-discr}
	by standard  weak \EEE compactness arguments.
	
	From the regularity result \cite[Thm.\ 2, Rmk.\ 3.5]{savare98},
 	we infer for every $1 \leq \delta\EEE< \frac1p$ the enhanced regularity
  $\pwc c{\tau},\upwc c{\tau} \in L^2 (0,T; W^{1+\delta\EEE,p}(\Omega))$
  together with the estimate
  \begin{align*}
  	\|\pwc c{\tau}\|_{L^2(0,T;W^{1+\delta,p}(\Omega))}+\|\upwc c{\tau}\|_{L^2(0,T;W^{1+\delta,p}(\Omega))}\leq C_{\delta\EEE}.
  \end{align*}
  In combination with \eqref{cBound1} and   \eqref{cBound2}, \EEE the application of the Aubin-Lions compactness theorem
  yields \eqref{cConv4}.
  Now we choose a subsequence $\tau\downarrow 0$ such that
  $\Delta_p(\pwc c{\tau})\weaklim S$ in $L^2(Q)$ for an element $S\in L^2(Q)$
  possible due to \eqref{cBound3}.
  Taking $\pwc c{\tau}\to c$ in $L^2(Q)$ into account, we may identify
  $S=\Delta_p(c)$ by the strong-weak closedness of the maximal monotone graph of    
  $\Delta_p:L^2(Q)\to L^2(Q)$.  We then conclude 
  \eqref{cConv3}. 
  Analogously,  \eqref{etaConv} ensues from the strong-weak closedness of the graph of the maximal monotone operator (induced by $\beta_\omega$)
  $\beta_\omega : L^2(Q) \to  L^2(Q) $. \EEE
    In addition, \EEE \eqref{zConv3}, \eqref{zConv4}, \eqref{uConv3} and \eqref{uConv4}
  follow from \eqref{zBound1}, \eqref{zBound2},   \eqref{uBound1} \EEE and \eqref{uBound2} via
  Aubin-Lions compactness results (see \cite{simon}).
	
	It remains to show the convergences for $\pwc\teta{\tau}$ and $\log(\pwc\teta{\tau})$.
	Here we proceed as in \cite[Proof of Lemma 5.1]{RocRos14}. We use the boundedness 
	properties \eqref{thetaBound3} and \eqref{thetaBound4},
	and apply the compactness result \cite[Theorem A.5]{RocRos14}
	which is based on Helly's selection principle.
	We obtain a function $\lambda\in L^2(0,T;H^1(\Omega))\cap L^\infty(0,T;W^{1,d+\epsilon}(\Omega)')$ for all $\epsilon>0$
	and a further,  (again not relabeled),  subsequence \EEE such that
	\begin{align*}
		&\log(\pwc\teta{\tau})\weakstarlim \lambda\text{ weakly-star in }L^2(0,T;H^1(\Omega))\cap L^\infty(0,T;W^{1,d+\epsilon}(\Omega)'),\\
		&\log(\pwc\teta{\tau}(t))\weaklim \lambda(t)\text{ weakly in }H^1(\Omega)
			\quad\text{for  a.a. \EEE }t\in(0,T).
	\end{align*}
	Here the chosen subsequence for $\tau\downarrow 0$ does not depend on $t$ in the latter convergence.
	We also infer from above that
	$$
		\log(\pwc\teta{\tau}(t))\to \lambda(t)\text{ strongly in }L^s(\Omega)
		\quad\text{for a.a. $t\in(0,T)$ and all $s$ from \eqref{thetaConv3}.}
	$$
	By also exploiting the boundedness of $\|\log( \pwc\teta{\tau})\|_{L^2(0,T;H^1(\Omega))\cap L^\infty(0,T;W^{1,d+\epsilon}(\Omega)')}$
	and the interpolation inequality \eqref{interpolationIneq} with $X=H^1(\Omega)$,
	$Y=L^s(\Omega)$ and $Z=W^{1,d+\epsilon}(\Omega)$ we infer
	that the sequence $\{\log(\pwc\teta{\tau})\}_{\tau}$ is uniformly integrable
	in $L^2(0,T;L^s(\Omega))$.
	Application  of Vitali \EEE convergence theorem proves
	$$
		\log(\pwc\teta{\tau})\to \lambda\text{ strongly in }L^2(0,T;L^s(\Omega))
		\quad\text{for all $s$ from \eqref{thetaConv3}.}
	$$
	Comparison with \eqref{thetaConv1} yields $\lambda=\log(\teta)$ and hence
	\eqref{thetaConv2}, \eqref{thetaConv3} and \eqref{thetaConv4}.
	The uniform bound \eqref{thetaBound1} shows uniform integrability of
	$\{\pwc\teta{\tau}\}_{\tau}$ in  $L^q(Q)$ with $q\in[1,8/3)$ for $d=3$ and
	$q\in[1,3)$ for $d=2$  (cf. \eqref{estetainterp}). \EEE
	Vitali's convergence theorem proves the strong convergence \eqref{thetaConv5}.
	
	In particular we find $\pwc\teta{\tau}(t)\to \teta(t)$ strongly in $L^1(\Omega)$
	for a.e. $t\in(0,T)$ (where the subsequence of $\tau\downarrow 0$ is independent of $t$).
	By the boundedness $\|\pwc\teta{\tau}(t)\|_{L^1(\Omega)}\leq C$ uniformly in $t$ and $\tau$
	(see \eqref{thetaBound1})
	we infer by lower  semicontinuity \EEE that $\teta\in L^\infty(0,T;L^1(\Omega))$.
	Furthermore, by considering a weak cluster point
	$(\ol\teta_\tau)^{\frac{\kappa+\alpha}{2}}\weaklim S$ in
	$L^2(0,T;H^1(\Omega))$ and identifying $S=(\ol\teta_\tau)^{\frac{\kappa+\alpha}{2}}$
	via a.e. limits from above we also obtain
	$\teta^{\frac{\kappa+\alpha}{2}}\in L^2(0,T;H^1(\Omega))$.
	
	Finally, under the additional assumption \eqref{range-k-admissible}
	 convergences \EEE \eqref{tetaConvAdd1} and \eqref{tetaConvAdd2} follow from
	an Aubin-Lions type compactness result for $\mathrm{BV}$-functions  (cf.\ e.g.\ \cite[Chap.\ 7, Cor.\ 4.9]{Rou05}), 
	combining estimate  \EEE \eqref{thetaBound1} together with the
	$\BV$-bound \eqref{thetaBoundAdd}. For  further details \EEE we refer to
	\cite[Proof of Lemma 5.1]{RocRos14}.
\end{proof}
\subsection{Conclusion of the proof of Theorem \ref{thm:1}}
\label{ss:6.2}
 Here is the outline of the proof: 
\begin{enumerate}
\item
First,  for fixed $\omega>0$  we will pass to the limit as $\tau\downarrow 0$, (along the same  subsequence for which the convergences in Lemma \ref{lemma:discr-conv} hold), 
in the time-discrete system \eqref{PDE-discrete}. 
We will thus obtain an \emph{entropic weak solution} (in the sense of Definition \ref{def-entropic}), to the (initial-boundary value problem for the) PDE
system \eqref{eqn:PDEsystem}, where the maximal monotone operator $\beta$ and the elastic energy density $W$ are replaced by their regularized versions 
$\beta_\omega$ and $W^\omega$. 
\item 
Secondly, we will tackle the limit passage as $\omega \down 0$. 
\end{enumerate}
Observe that  the  limit passages $\tau \down 0 $ and  $\omega\downarrow 0$ cannot be performed
simultaneously, \EEE because
in the time-discrete system from Problem \ref{def:time-discrete}
the (partial)   derivatives of the convex- and the concave-decompositions  \EEE  \eqref{eqn:convConcSplittingWc}
may ``explode'' as $\omega\downarrow 0$.
However,  the convex-concave splitting  shall \EEE disappear
in the limit $\tau\downarrow 0$ for fixed $\omega>0$ in the corresponding PDE system.
\paragraph{\bf Limit passage $\tau\downarrow 0$}
%stated in Problem \ref{def:time-discrete}
%and solved in Proposition \ref{prop:exist-discr}.
 First of all, \EEE we mention that from the time-discrete damage equation
\eqref{eqn:discr3} we derive the following inequalities
(for details we refer to \cite[Section 5.2]{hk1};
see also \cite[Proof of Theorem 1]{RocRos14} and \cite[Proof of Theorem 4]{RocRos12}):
\begin{align}
	&\textit{-- damage energy-dissipation inequality:}\text{ for all $t \in (0,T]$, for $s=0$, and for almost all $0< s\leq t$:}
		\notag\\
	&\qquad
    \begin{aligned}
    \label{discr-energy-diss-ineq}
      &\int_{\ttau(s)}^{\ttau(t)}\io |\partial_t z_\tau|^2 \dd x \dd r
      	+\io\left(\frac1p  |\nabla\pwc z{\tau}(t)|^p +  (\conv{\sigma})'(\pwc z{\tau}(t))+ (\conc{\sigma})'(\upwc z{\tau}(t))\right)\dd x\\
      &\qquad\leq\io\left(\frac1p |\nabla\pwc z{\tau}(s)|^p+ (\conv{\sigma})'(\pwc z{\tau}(s))+ (\conc{\sigma})'(\upwc z{\tau}(s))\right)\dd x\\
	      &\qquad\quad+\int_{\ttau(s)}^{\ttau(t)}\io \partial_t z_\tau
	      \left(-\convWzp(\pwc c{\tau},\e(\upwc\uu{\tau}),\pwc z{\tau}) - \concWzp(\pwc c{\tau},\e(\upwc\uu{\tau}),\upwc z{\tau})+\pwc \teta{\tau}\right)\dd x \dd r;
    \end{aligned}\\
	&\textit{-- damage variational inequality:}\text{ for all $\zeta\in L^\infty(0,T;W^{1,p}(\Omega))$ with $0\leq \zeta\leq\upwc z{\tau}$:}
		\notag\\
  &\qquad
    \begin{aligned}
		  \label{var-ineq-z-bis}
  	  &\int_0^T\int_\Omega  \Big(|\nabla \pwc z{\tau}|^{p-2} \nabla \pwc z{\tau} \cdot \nabla (\zeta-\pwc z{\tau})
  	  	\big((\partial_t z_\tau) + (\conv{\sigma})'(\pwc z{\tau})+ (\conc{\sigma})'(\upwc z{\tau})\big)(\zeta-\pwc z{\tau})\Big)\dx\dr\\
  	  &+\int_0^T\int_\Omega\Big(\convWzp(\pwc c{\tau},\e(\upwc\uu{\tau}),\pwc z{\tau})
  	  	+ \concWzp(\pwc c{\tau},\e(\upwc\uu{\tau}),\upwc z{\tau})
  	  	- \pwc \teta{\tau}\Big)(\zeta-\pwc z{\tau})\dx\dr\geq 0.
    \end{aligned}
\end{align}

The limit passage $\tau\downarrow 0$ in the damage energy-dissipation inequality 
\eqref{discr-energy-diss-ineq}, in the damage variational inequality
\eqref{var-ineq-z-bis},
in the entropy inequality \eqref{discr-entropy-ineq},
in the total energy inequality \eqref{discr-energy-ineq}
and in the equation for the balance of forces \eqref{eqn:discr5}
works exactly as outlined in \cite[Proof of Theorem 1]{RocRos14}
by taking the growth  properties \EEE \eqref{est-quoted-5.1}--\eqref{est-quoted-5.4}
into account (for \textbf{fixed}  $\omega>0$\EEE)
and needs no repetition here.

We end up with properties (ii), (iii), (iv) and (v) of Definition \ref{def-entropic},
 keeping in mind that 
 $W(c,\e(\uu),z)$, $W_{,z}(c,\e(\uu),z)$ and $W_{,\e}(c,\e(\uu),z)$
are replaced by their \EEE $\omega$-regularized versions
$W^\omega(c,\e(\uu),z)$, $W_{,z}^\omega(c,\e(\uu),z)$ and $W_{,\e}^\omega(c,\e(\uu),z)$,
respectively.
Let us comment that in the limit $\tau\downarrow 0$ of \eqref{var-ineq-z-bis}
we are only able to obtain a ``one-sided variational inequality''
which still suffices to obtain a weak solution in the sense of
Definition \ref{def-entropic} (see \eqref{var-ineq-z}).
Furthermore,  following \EEE the approach from \cite[Proof of Theorem 4.4]{hk1},
the subgradient $\xi\in L^2(0,T;L^2(\Omega))$
 fulfilling $\xi \in \partial I_{[0,+\infty)}(z)$ a.e.\ in $Q$ \EEE
 can be specified precisely as
\begin{align*}
	\xi=-\1_{\{z=0\}}\Big(\sigma'(z) + \pd{z}(c,\e(\ub), z)-\teta\Big)^+ \qquad   \aein\, Q, \EEE
\end{align*}
where $\1_{\{z=0\}}:Q\to\{0,1\}$ denotes the characteristic function
of the set $\{z=0\}\subseteq Q$ and $(\cdot)^+:=\max\{0,\cdot\}$.

It remains to show the limit passage as  $\tau\downarrow 0$ in the Cahn-Hilliard system \eqref{eqn:discr1}--\eqref{eqn:discr2}.
This can be achieved via standard convergence methods by exploiting
the convergences shown in Lemma \ref{lemma:discr-conv}
and noticing the growth properties \eqref{est-quoted-5.1}--\eqref{est-quoted-5.4}.
This leads to property (i) from Definition \ref{def-entropic}
where $W_{,c}(c,\e(\uu),z)$ and $\beta$ should be replaced by
$W_{,c}^\omega(c,\e(\uu),z)$ and $\beta_\omega$, respectively.
\\
\paragraph{\bf Limit passage  $ {\boldsymbol \omega} {\boldsymbol \downarrow}
  {\bf  0}$ \EEE}
In the subsequent argumentation we let
$S_\omega=(c_\omega,\mu_\omega,z_\omega,\teta_\omega,\uu_\omega)$
be an $\omega$-regularized weak solution, i.e. an  entropic \EEE weak solution
in the sense of Definition \ref{def-entropic} where
the $W$, $W_{,c}$, $W_{,\e}$, $W_{,z}$ and $\beta$-terms
are replaced by $W^\omega$, $W_{,c}^\omega$, $W_{,\e}^\omega$, $W_{,z}^\omega$ and $\beta_\omega$,
respectively.
We observe that the a priori estimates
in Proposition \ref{prop:aprio-discr} are inherited
by the weak solutions $S_\omega$ via lower  semicontinuity \EEE arguments.
Hence we obtain the same convergence properties for $\omega\downarrow 0$
as for $\tau\downarrow 0$ in Lemma \ref{lemma:discr-conv} where \eqref{etaConv}
should be replaced by
\begin{align}
	\pwc \eta{\omega} \weaklim \eta
		\quad\text{ weakly in }L^2(0,T;L^2(\Omega))\text{ as }\omega\downarrow 0\text{ with }\eta\in\beta(c)\text{ a.e. in }Q.
		\label{etaConv2}
\end{align}
 Indeed, to prove \EEE \eqref{etaConv2}, let $\eta_{\omega}=\beta_\omega(c_\omega)\weaklim S$ in $L^2(Q)$
for $\omega\downarrow 0$  for some \EEE element $S\in L^2(Q)$.
By convexity of the operator $\widehat\beta_\omega:L^2(Q)\to\R$ we find
\begin{align}
\label{convBeta}
	\forall w\in L^2(Q):\quad \widehat\beta_\omega(c_\omega)+\langle \beta_\omega(c_\omega), w-c_\omega\rangle_{L^2(Q)}\leq \widehat\beta_\omega(w).
\end{align}
Since $\{\beta_\omega\}$ is the Yosida-approximation of $\beta$ we conclude that (cf. \cite[Lemma 5.17]{Rou05})
\begin{align}
	\forall w\in L^2(Q):\quad \widehat\beta_\omega(w)\to \widehat\beta(w)\text{ strongly in $L^2(Q)$ as }\omega\downarrow0
	 	\qquad\text{ and }\qquad\liminf_{\omega\downarrow0}\widehat\beta_\omega(c_\omega)\geq \widehat\beta(c).
		\label{auxConv}
\end{align}
Thus by \eqref{etaConv2} and \eqref{auxConv} we can pass to the limit $\omega\downarrow 0$ for a subsequence in \eqref{convBeta}
and obtain $\eta\in\partial\beta(c)$.

The main feature for the passage $\omega\downarrow 0$ in the PDE system is the following observation:
From \eqref{cBound1} and \eqref{cBound2}
we infer via the compact embedding $W^{1,p}(\Omega)\Subset L^\infty(\Omega)$
that for all $\omega>0$
\begin{align*}
	\|c_\omega\|_{L^\infty(Q)}\leq C.
\end{align*}
An important consequence is that in combination with the definition of $\C R_\omega$ in \eqref{Rtrunc} we find
for all sufficiently small $\omega>0$ that $\C R_\omega(c_\omega)=c_\omega$ a.e. in $Q$ and thus
\begin{align*}
	\left.
	\begin{aligned}
		&W(c_\omega,\e(\uu_\omega),z_\omega)=W^\omega(c_\omega,\e(\uu_\omega),z_\omega),
		&W_{,c}(c_\omega,\e(\uu_\omega),z_\omega)=W_{,c}^\omega(c_\omega,\e(\uu_\omega),z_\omega),\\
		&W_{,\e}(c_\omega,\e(\uu_\omega),z_\omega)=W_{,\e}^\omega(c_\omega,\e(\uu_\omega),z_\omega),
		&W_{,z}(c_\omega,\e(\uu_\omega),z_\omega)=W_{,z}^\omega(c_\omega,\e(\uu_\omega),z_\omega).
	\end{aligned}
	\quad\right\}\quad
	\text{a.e. in }Q.
\end{align*}
 Then, the  limit \EEE passage $\omega\downarrow 0$ in the $\omega$-regularized versions
of (i)-(v) in Definition \ref{def-entropic} works as for $\tau\downarrow 0$. 

 This concludes the proof of 
Theorem \ref{thm:1}. \EEE
\QED

\noindent
{\bf \large Acknowledgments.} 
  Christian Heinemann and Christiane Kraus have been partially supported
by ECMath SE 4 and SFB 1114. \EEE
The work of Elisabetta Rocca was
supported by the FP7-IDEAS-ERC-StG Grant \#256872
(EntroPhase), by GNAMPA (Gruppo Nazionale per l'Analisi Matematica, la Probabilit\`a e le loro Applicazioni) of INdAM (Istituto Nazionale di Alta Matematica), and by IMATI -- C.N.R. Pavia. Riccarda Rossi was
partially supported by a MIUR-PRIN'10-'11 grant for the project
``Calculus of Variations'',  and by GNAMPA (Gruppo Nazionale per l'Analisi Matematica, la Probabilit\`a e le loro Applicazioni)  of  INdAM (Istituto Nazionale di Alta Matematica).  

\bibliographystyle{alpha}

\end{document}